\documentclass{amsart}

\usepackage{pstricks}
\usepackage{pst-node}
\usepackage{graphicx}
\usepackage{amsmath,amsfonts,amssymb}

\newtheorem{theorem}{Theorem}[section]
\newtheorem{lemma}[theorem]{Lemma}
\newtheorem{proposition}[theorem]{Proposition}
\newtheorem{corollary}[theorem]{Corollary}
\newtheorem{conjecture}[theorem]{Conjecture}

\theoremstyle{definition}
\newtheorem{definition}[theorem]{Definition}

\theoremstyle{remark}
\newtheorem{remark}[theorem]{Remark}

\numberwithin{equation}{section}

\newcommand{\iph}{\ensuremath i\!+\!3}
\newcommand{\ipt}{\ensuremath i\!+\!2}
\newcommand{\ipo}{\ensuremath i\!+\!1}
\newcommand{\imo}{\ensuremath i\!-\!1}
\newcommand{\imt}{\ensuremath i\!-\!2}
\newcommand{\imh}{\ensuremath i\!-\!3}
\newcommand{\imf}{\ensuremath i\!-\!4}
\newcommand{\triple}{\ensuremath i\!-\!1,i,i\!+\!1}

\newcommand{\npo}{\ensuremath n\!+\!1}
\newcommand{\nmo}{\ensuremath n\!-\!1}

\newcommand{\G}{\ensuremath\mathcal{G}}
\newcommand{\C}{\ensuremath\mathcal{C}}

\newcommand{\LSP}{\ensuremath\mathrm{LSP}}
\newcommand{\LSF}{\ensuremath\mathrm{LSF}}

\newcommand{\B}{\bullet}
\newcommand{\cs}[1]{c@{\hskip #1ex}}

\newlength\cellsize 

\savebox2{%
\begin{picture}(12,12)
\put(0,0){\line(1,0){12}}
\put(0,0){\line(0,1){12}}
\put(12,0){\line(0,1){12}}
\put(0,12){\line(1,0){12}}
\end{picture}}

\newcommand\cellify[1]{\def\thearg{#1}\def\nothing{}%
\ifx\thearg\nothing
\vrule width0pt height\cellsize depth0pt\else
\hbox to 0pt{\usebox2\hss}\fi%
\vbox to 12\unitlength{
\vss
\hbox to 12\unitlength{\hss$#1$\hss}
\vss}}

\newcommand\tableau[1]{\vtop{\let\\=\cr
\setlength\baselineskip{-12000pt}
\setlength\lineskiplimit{12000pt}
\setlength\lineskip{0pt}
\halign{&\cellify{##}\cr#1\crcr}}}

\newlength\smcellsize 

\savebox3{%
\begin{picture}(8,8)
\put(0,0){\line(1,0){8}}
\put(0,0){\line(0,1){8}}
\put(8,0){\line(0,1){8}}
\put(0,8){\line(1,0){8}}
\end{picture}}

\newcommand\smcellify[1]{\def\thearg{#1}\def\nothing{}%
\ifx\thearg\nothing
\vrule width0pt height\smcellsize depth0pt\else
\hbox to 0pt{\usebox3\hss}\fi%
\vbox to 8\unitlength{
\vss
\hbox to 8\unitlength{\hss$#1$\hss}
\vss}}

\newcommand\smtableau[1]{\vtop{\let\\=\cr
\setlength\baselineskip{-8000pt}
\setlength\lineskiplimit{8000pt}
\setlength\lineskip{0pt}
\halign{&\smcellify{##}\cr#1\crcr}}}

\newcommand{\e}{\mbox{}}
\definecolor{boxgray}{gray}{.7}
\newcommand{\cb}{\color{boxgray}\rule{1\cellsize}{1\cellsize}\hspace{-\cellsize}\usebox2}

\newcommand{\stab}[3]{\begin{array}{c}\rnode{#1}{\tableau{#2}}\\\rnode{#1#1}{_{#3}}\end{array}}

\newcommand{\sbull}[2]{\begin{array}{c}\rnode{#1}{\B}\\[-1ex]\rnode{#1#1}{\makebox[0pt]{$_{#2}$}}\end{array}}

\definecolor{lightgray}{gray}{.85}

\begin{document}


\title[Dual equivalence graphs II]{Dual equivalence graphs II: \\ 
  Transformations on locally Schur positive graphs}  

\author[S. Assaf]{Sami H. Assaf}
\address{Department of Mathematics, University of Southern California, Los Angeles, CA 90089-2532}
\email{shassaf@usc.edu}





\begin{abstract}
  Dual equivalence graphs are a powerful tool in symmetric function theory that provide a general framework for proving that a given quasisymmetric function is symmetric and Schur positive. In this paper, we study a larger family of graphs that includes dual equivalence graphs and define maps that, in certain cases, transform graphs in this larger family into dual equivalence graphs. This allows us to broaden the applications of dual equivalence graphs and points the way toward a broader theory that could solve many important, long-standing Schur positivity problems.
\end{abstract}

\maketitle
\tableofcontents

%
\section{Introduction}
%
\label{sec:introduction}

A quintessential problem in symmetric function theory that often arises in connections with representation theory or geometry is to prove that a given function has nonnegative integer coefficients when expressed in the Schur basis. For example, the product of two Schur functions is Schur positive, and these Littlewood--Richardson coefficients that arise as structure constants for Schur functions give multiplicities of irreducible representations in tensor products of representations for the general linear group as well as intersection numbers for Grassmannian subvarieties of the flag manifold. More generally, the LLT polynomials introduced by Lascoux, Leclerc, and Thibon \cite{LLT97} give a $q$-analog of these coefficients that arise from Fock space representations of quantum affine Lie algebras. In certain cases, the Schur coefficients of LLT polynomials are known to be nonnegative \cite{LT00,vLe05,Ass15}, but the general case remains open. Many important symmetric functions can be expressed naturally as a nonnegative sum of LLT polynomials, adding emphasis to this important problem. For example, Macdonald polynomials \cite{Mac88} have a monomial expansion due to Haglund, Haiman and Loehr \cite{Hag04,HHL05} that easily gives an LLT expansion as well. While Macdonald polynomials are known to expand non-negatively into the Schur basis by the representation theory developed by Garsia and Haiman \cite{GH96} and the algebraic geometry developed by Haiman \cite{Hai01}, it remains an important open problem to give a combinatorial proof of Schur positivity.

In this paper, we study a combinatorial construction called a \emph{dual equivalence graph} \cite{Ass07,Ass15} by which one can establish the symmetry and Schur positivity of a function expressed in Gessel's fundamental basis for quasisymmetric functions. This powerful tool has been studied from many perspectives\footnote{See, for example \cite{MR1,MR2,Bla16,MR4,MR5,MR6,MR7,MR8,MR9,MR10,MR11,MR12,MR13,MR14,MR15,MR16,MR17,MR18,AB12}, only the last of which includes the original author.} since its introduction in \cite{Ass07}, though the initial definitions did not appear until \cite{Ass15}. Among other applications, dual equivalence graphs give a simple proof of positivity of the Littlewood--Richardson coefficients in general as well as the LLT $q$-analogs indexed by two shapes \cite{Ass15}, the latter of which implies a combinatorial proof of Macdonald positivity for two column shapes. Moreover, the approach of this paper has led to a simple combinatorial formula for the Schur expansion of Macdonald polynomials indexed by a partition with second part at most $2$ \cite{Ass-B}, a case that has not been proved by any other means.

While this paper is self-contained, the reader is encouraged to consult its predecessor \cite{Ass15} for a basic introduction to dual equivalence graphs. Nevertheless, we begin in Section~\ref{sec:schur} by presenting the essentials of quasisymmetric functions as they relate to Schur functions and, in particular, the approach for which we advocate. We recall the definition of dual equivalence graphs from \cite{Ass15} and recall the main theorem from \cite{Ass15}, that putting a dual equivalence graph structure on a combinatorial set establishes the Schur positivity of the generating function. We refer to reader to \cite{Ass15} for explicit examples of the standard dual equivalence graphs on tableaux as well as for the explicit structure defined for LLT and Macdonald polynomials. The examples throughout this paper will be of abstract graphs with vertices denoted by $\B$.

Picking up from this point, in Section~\ref{sec:local}, we define a larger class of \emph{locally Schur positive graphs}. The motivation for this definition comes from the re-characterization of the axioms for dual equivalence graphs in \cite{Ass15} that states that a function that is locally a Schur function is globally a Schur function provided there exists a family of commuting involutions with prescribed fixed points. Note that this commutativity condition is necessary and, as such, will never be omitted.\footnote{The examples in \cite{BF17}(Proposition~3.8) are counter-examples to the false statement the authors considered: that dual equivalence axiom $5$ can be omitted without consequence. This was a natural consideration from their algebraic reformulation of dual equivalence graphs, though from the graph perspective axiom $5$ is clearly necessary. Those examples have no relevance to dual equivalence graphs nor to the generalizations considered in this paper since we \emph{always} assume axiom $5$ holds.} A locally Schur positive graph is one whose degree $4,5,6$ restricted generating functions are Schur positive. This family includes the graphs constructed for LLT and Macdonald polynomials \cite{Ass15} as well as those for $k$-Schur functions \cite{AB12}. It is not known whether this family includes graphs whose full generating functions are not Schur positive.\footnote{The $4950$ vertex graph $\G^{*}$ of degree $8$ in \cite{Bla16}(Section~7.6) is not locally Schur positive.} We prove basic results for low degree graphs both to motivate the transformations to come, one for each of degrees $4,5,6$, and to give a sense of the style of diagram chasing arguments to follow. 

The three transformations are discussed in Section~\ref{sec:LSP6} for degree $6$ (equivalently, axiom $6$), Section~\ref{sec:LSP4} for degree $4$, and Section~\ref{sec:LSP5} for degree $5$ (these latter two together are equivalent to axiom $4$). For each case, we define and investigate in detail a transformation on graphs that focuses on the restricted generating functions of the given degree. We give necessary and sufficient conditions for the maps to be applied to a graph, and prove that the maps maintain the lower structure of the graph. In Section~\ref{sec:LSP}, we introduce a third transformation designed to facilitate the application of the others, and we prove that together these maps can transform a locally Schur positive graph into a dual equivalence graph if only the highest color edges need adjustment. While theoretically necessary, this additional transformation is not needed to carry out the described transformation of the graphs for LLT polynomials, Macdonald polynomials, or $k$-Schur functions up to degree $12$.

Finally, in Section~\ref{sec:axiom4p}, we explore the shortfalls of these transformations, presenting necessary additional axioms that must be maintained in order for the maps to maintain the upper structure of the graph. Not only do the graphs for LLT polynomials, Macdonald polynomials, and $k$-Schur functions all satisfy these additional axioms, but we prove that the transformations cannot create new violations of them. The question remains as to whether they are implied by degree $6$ local Schur positivity as well as whether they are sufficient to ensure the transformation to a dual equivalence graph can always be carried out.

%
\section{Dual equivalence graphs}
%
\label{sec:schur}

We begin with functions expressed in terms of Gessel's fundamental quasisymmetric functions \cite{Ges84}. For $\sigma \in \{\pm 1\}^{\nmo}$, the \emph{fundamental quasisymmetric function} $Q_{\sigma}(X)$ is defined by
\begin{equation}
  Q_{\sigma}(X) = \sum_{\substack{i_1 \leq \cdots \leq i_n \\ \sigma_j = -1 \Rightarrow i_j < i_{j+1}}} x_{i_1} \cdots x_{i_n} .
\label{eqn:quasisym}
\end{equation}
We have indexed quasisymmetric functions by sequences of $+1$'s and $-1$'s to facilitate the study of signed graphs. By setting $D(\sigma) = \{ i \mid \sigma_i = -1\}$, we may change the indexing to the more familiar subsets of $[\nmo] =\{1,2,\ldots,n-1\}$.

A \emph{standard Young tableau} is  bijection $T$ from the cells of a partition $\lambda$ of $n$ to $[n]$ such that entries increase along rows and columns. To connect quasisymmetric functions with Schur functions, for $T$ a standard Young tableau on $[n]$, define the \emph{descent signature} $\sigma(T) \in \{\pm1\}^{\nmo}$ by
\begin{equation}
  \sigma(T)_{i} \; = \; \left\{ 
    \begin{array}{ll}
      +1 & \; \mbox{if $i$ appears weakly above $\ipo$ in $T$} \\
      -1 & \; \mbox{if $\ipo$ appears strictly above $i$ in $T$}
    \end{array} \right. .
\label{e:sigma}
\end{equation}

We can now define a Schur function as the generating function for standard Young tableaux.

\begin{definition}[\cite{Ges84}]
  The Schur function $s_{\lambda}$ is given by
  \begin{equation}
    s_{\lambda}(X) = \sum_{T \in \mathrm{SYT}(\lambda)} Q_{\sigma(T)}(X) .
  \label{eqn:quasi-s}
  \end{equation}
\label{def:quasisym}
\end{definition}

Dual equivalence graphs provide a general framework for proving that a given function, expressed in terms of fundamental quasisymmetric functions, is symmetric and Schur positive. The motivating idea is to collect together terms in the fundamental quasisymmetric expansion of a function into equivalence classes so that each class is a single Schur function, or, more generally, is Schur positive. 

Given a set $\mathcal{A}$ of combinatorial objects and a signature function $\sigma:\mathcal{A}\rightarrow\{\pm 1\}^{n-1}$, we can form the quasisymmetric generating function for $\mathcal{A}$ with respect to $\sigma$ by
\begin{displaymath}
  f(X;q) \ = \ \sum_{T \in \mathcal{A}} q^{\mathrm{stat}(T)} Q_{\sigma(T)}(X),
\end{displaymath}
where $\mathrm{stat}$ is some positive integer statistic on $\mathcal{A}$. 

In order to prove that $f(X;q)$ is Schur positive, we group together terms of $\mathcal{A}$ into equivalence classes so that each class has a constant $\mathrm{stat}$ value and the quasisymmetric generating function for a class is a Schur function. To put added structure, we define not only equivalence classes, but elementary equivalence relations between objects of $\mathcal{A}$, and we record this structure with a signed, colored graph.

\begin{definition}[\cite{Ass15}]
  A \emph{signed, colored graph of type $(n,N)$} consists of the following:
  \begin{itemize}
    \item a finite vertex set $V$;
    \item a signature function $\sigma : V \rightarrow \{\pm 1\}^{N-1}$;
    \item for each $1 < i < n$, a collection $E_i$ of pairs of distinct vertices of $V$.
  \end{itemize}
  We denote such a graph by $\G = (V,\sigma,E_{2} \cup\cdots\cup E_{\nmo})$ or simply by $(V,\sigma,E)$.
\label{def:scg}
\end{definition}

Here the vertex set $V$ corresponds to the combinatorial set $\mathcal{A}$ and the edges correspond to elementary equivalence relations. For example, if we take $\mathcal{A}$ to be standard Young tableaux of shape $(3,2)$ with $\sigma$ as in \eqref{e:sigma}, and elementary equivalence relations that swap $i$ with whichever of $i \pm 1$ lies further away, we obtain the \emph{standard dual equivalence graph} shown in Figure~\ref{fig:G5}. Its generating function is $s_{(3,2)}(X)$.

\begin{figure}[ht]
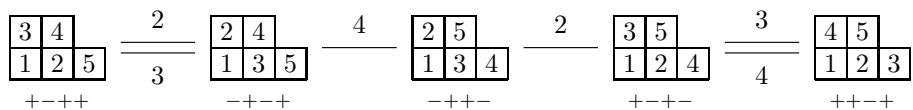

  \begin{center}
    \begin{displaymath}
      \begin{array}{\cs{7} \cs{7} \cs{7} \cs{7} c}
        \stab{a}{3 & 4 \\ 1 & 2 & 5}{+-++} &
        \stab{b}{2 & 4 \\ 1 & 3 & 5}{-+-+} &
        \stab{c}{2 & 5 \\ 1 & 3 & 4}{-++-} &
        \stab{d}{3 & 5 \\ 1 & 2 & 4}{+-+-} &
        \stab{e}{4 & 5 \\ 1 & 2 & 3}{++-+}
      \end{array}
      \psset{nodesep=6pt,linewidth=.1ex}
      \ncline[offset=2pt] {a}{b} \naput{2}
      \ncline[offset=2pt] {b}{a} \naput{3}
      \ncline             {b}{c} \naput{4}
      \ncline             {c}{d} \naput{2}
      \ncline[offset=2pt] {d}{e} \naput{3}
      \ncline[offset=2pt] {e}{d} \naput{4}
    \end{displaymath}
    \caption{\label{fig:G5}The standard dual equivalence graph  $\G_{(3,2)}$.}
  \end{center}
\end{figure}

In general terms, the goal is for the connected components, which correspond to equivalence classes, to be Schur functions. This notion is captured by the following definition, Definition~3.2 from \cite{Ass15}.

\begin{definition}[\cite{Ass15}]
  A signed, colored graph $\G = (V,\sigma,E)$ of type $(n,N)$ is a \emph{dual equivalence graph of type $(n,N)$} if $n \leq N$ and the following hold:
  \begin{itemize}
    
  \item[(ax$1$)] For $w \in V$ and $1<i<n$, $\sigma(w)_{\imo} = -\sigma(w)_{i}$ if and only if there exists $x \in V$ such that $\{w,x\} \in E_{i}$. Moreover, $x$ is unique when it exists.

  \item[(ax$2$)] For $\{w,x\} \in E_{i}$, $\sigma(w)_j = -\sigma(x)_j$ for $j=\imo,i$, and $\sigma(w)_h = \hspace{1ex}\sigma(x)_h$ for $h < \imt$ and $h > \ipo$.
      
  \item[(ax$3$)] For $\{w,x\} \in E_{i}$, if $\sigma(w)_{\imt} = -\sigma(x)_{\imt}$, then $\sigma(w)_{\imt} = -\sigma(w)_{\imo}$, and if $\sigma(w)_{\ipo} = -\sigma(x)_{\ipo}$, then $\sigma(w)_{\ipo} = -\sigma(w)_{i}$.
    
  \item[(ax$4$)] Every connected component of $(V,\sigma, E_{\imo} \cup E_{i})$ appears in Figure~\ref{fig:lambda4} and every connected component of $(V,\sigma,E_{\imt} \cup E_{\imo} \cup E_{i})$ appears in Figure~\ref{fig:lambda5}.

  \item[(ax$5$)] If $\{w,x\} \in E_i$ and $\{x,y\} \in E_j$ for $|i-j|\geq 3$, then $\{w,v\} \in E_j$ and $\{v,y\} \in E_i$ for some $v \in V$.

  \item[(ax$6$)] Any two vertices of a connected component of $(V,\sigma,E_2 \cup\cdots\cup E_i)$ may be connected by a path crossing at most one $E_i$ edge.
    
  \end{itemize}
\label{def:deg}
\end{definition}

\begin{figure}[ht]
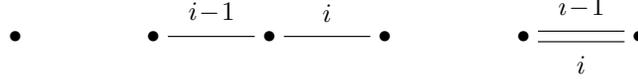

  \begin{displaymath}
    \begin{array}{\cs{11}\cs{9}\cs{9}\cs{11}\cs{9}c}
      \B & \rnode{a}{\B} & \rnode{b}{\B} & \rnode{c}{\B} & \rnode{d}{\B} & \rnode{e}{\B}
    \end{array}
    \psset{nodesep=3pt,linewidth=.1ex}
    \ncline            {a}{b} \naput{\imo}
    \ncline            {b}{c} \naput{i}
    \ncline[offset=2pt]{d}{e} \naput{\imo}
    \ncline[offset=2pt]{e}{d} \naput{i}
  \end{displaymath}
  \caption{\label{fig:lambda4} Allowed $2$-color connected components of a dual equivalence graph.}
\end{figure}

\begin{figure}[ht]
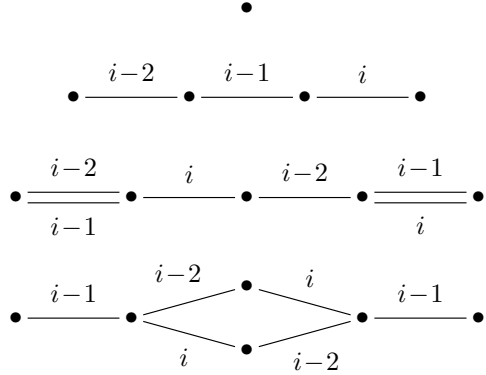

  \begin{displaymath}
    \begin{array}{c}
        \begin{array}{c}
          \B
        \end{array}\\[5ex]
        \begin{array}{\cs{9} \cs{9} \cs{9} c}
          \rnode{h}{\B} & \rnode{i}{\B} & \rnode{j}{\B} & \rnode{k}{\B}
        \end{array} \\[6ex]
        \begin{array}{\cs{9} \cs{9} \cs{9} \cs{9} c}
          \rnode{a}{\B} & \rnode{b}{\B} & \rnode{c}{\B} & \rnode{d}{\B} & \rnode{e}{\B}
        \end{array} \\[5ex]
        \begin{array}{\cs{9} \cs{9} \cs{9} \cs{9} c}
          & & \rnode{w}{\B} & & \\
          \rnode{u}{\B} & \rnode{v}{\B} & & \rnode{y}{\B} & \rnode{z}{\B} \\
          & & \rnode{x}{\B} & &
        \end{array}
        \psset{nodesep=2pt,linewidth=.1ex}
        \ncline  {h}{i} \naput{\imt}
        \ncline  {i}{j} \naput{\imo}
        \ncline  {j}{k} \naput{i}
        \ncline[offset=2pt] {a}{b} \naput{\imt}
        \ncline[offset=2pt] {b}{a} \naput{\imo}
        \ncline             {b}{c} \naput{i}
        \ncline             {c}{d} \naput{\imt}
        \ncline[offset=2pt] {d}{e} \naput{\imo}
        \ncline[offset=2pt] {e}{d} \naput{i}
        \ncline {u}{v}  \naput{\imo}
        \ncline {v}{w}  \naput{\imt}
        \ncline {v}{x}  \nbput{i}
        \ncline {w}{y}  \naput{i}
        \ncline {x}{y}  \nbput{\imt}
        \ncline {y}{z}  \naput{\imo}
      \end{array}
    \end{displaymath}
  \caption{\label{fig:lambda5} Allowed $3$-color connected components of a dual equivalence graph.}
\end{figure}

Note that Figure~\ref{fig:lambda5} implies Figure~\ref{fig:lambda4}. Therefore in graph of type $(n,N)$ with $n \geq 5$, axiom $4$ is equivalent to all connected component of $(V,\sigma,E_{\imt} \cup E_{\imo} \cup E_{i})$ appearing in Figure~\ref{fig:lambda5}. Moreover, if $n\leq 5$, then inspecting Figures~\ref{fig:lambda4} and \ref{fig:lambda5}, we see that axiom $4$ implies axiom $6$. In \cite{Ass15}, we prove that both axioms $4$ and $6$ may be replaced with local Schur positivity conditions, though for our current purposes we prefer to study the explicit structure of the graph.

As a first example, we note that the graph structure on standard Young tableaux of shape $\lambda$ given by elementary dual equivalence moves is a dual equivalence graph. This is proved in Proposition~3.3 of \cite{Ass15}. For this reason, we call these the \emph{standard dual equivalence graphs} and denote them by $\G_{\lambda}$.

When two graphs satisfy axiom 1, as all graphs in this paper do, we define an isomorphism between them, Definition~3.5 from \cite{Ass15}, to be a sign-preserving bijection on vertex sets that respects color-adjacency. 

\begin{definition}[\cite{Ass15}]
  A \emph{morphism} between two signed, colored graphs of type $(n,N)$ satisfying dual equivalence graph axiom $1$, say $\G=(V,\sigma,E)$ and $\mathcal{H}=(W,\tau,F)$, is a map $\phi : V \rightarrow W$ such that 
  \begin{itemize}
  \item for every $1 \leq i < N$, we have $\sigma(v)_i = \tau(\phi(v))_i$, and
  \item for every $1 < i < n$, if $\{u,v\} \in E_i$, then $\{\phi(u),\phi(v)\} \in F_i$.
  \end{itemize}
  A morphism is an \emph{isomorphism} if it is a bijection on vertex sets.
\label{def:isomorphism}
\end{definition}

One of the main results from \cite{Ass15}, Theorem~3.7, is the following structure theorem.

\begin{theorem}[\cite{Ass15}]
  Every connected component of a dual equivalence graph of type $(n,n)$ is isomorphic to $\G_{\lambda}$ for a unique partition $\lambda$ of $n$.
\label{thm:isomorphic}
\end{theorem}

In \cite{Ass15}, and earlier in \cite{Ass07}, we define a graph structure for the combinatorial objects that generate Macdonald polynomials and, more generally, LLT polynomials  \cite{Ass15,Ass-B} as well as $k$-Schur functions \cite{AB12}. In certain cases, these are dual equivalence graphs, and so the following combinatorial formula, Corollary~3.8 from \cite{Ass15}, applies to give the Schur expansion.

\begin{corollary}[\cite{Ass15}]
  Let $\G$ be a dual equivalence graph of type $(n,n)$ such that every vertex is assigned some additional statistic $\alpha$ that is constant on connected components of $\G$. Then
  \begin{equation}
    \sum_{v \in V(\G)} q^{\alpha(v)} Q_{\sigma(v)}(X) =
    \sum_{\lambda} \left( \sum_{\C \cong \G_{\lambda}} q^{\alpha(\C)} \right) s_{\lambda}(X) .
    \label{eqn:schurpos}
  \end{equation}
  where the inner sum is over connected components of $\G$ that are isomorphic to $\G_{\lambda}$. In particular, the generating function for $\G$ so defined is symmetric and Schur positive.
\label{cor:schurpos}
\end{corollary}

These graphs are not, in general, dual equivalence graphs. In particular, for Macdonald polynomials this fails for $\mu_1>3$ or $\mu_2>2$. However, we conjecture that the graphs for LLT polynomials, Macdonald polynomials, and $k$-Schur functions have Schur positive connected components. This has been tested by computer up to $n=12$. Moreover, using the algorithms described in this paper, we obtain explicit formulas for the Schur expansion in certain cases \cite{Ass-B}.

%
\section{Local Schur positivity}
%
\label{sec:local}

The graphs for LLT polynomials, Macdonald polynomials, and $k$-Schur functions all have similar local structure. In order to study these graphs and expand the applicability of dual equivalence, we introduce terminology to characterize certain local properties.

\begin{definition}
  A signed, colored graph $\G = (V,\sigma,E)$ of type $(n,N)$ is \emph{Schur multiplicity-free for degree $m$}, denoted by $\LSF_m$, if for every $m-2<i<n$ and every connected component $\C$ of $(V,\sigma,E_{i-(m-3)} \cup \cdots \cup E_i)$, there exists some $\lambda$ such that 
  \begin{equation}
    \sum_{v \in \C} Q_{\sigma(v)_{i-(m-2),\ldots,i}}(X) = s_{\lambda}(X).
    \label{eqn:LSF}
  \end{equation}
  A graph is \emph{locally Schur multiplicity-free} if it is Schur multiplicity-free for degrees up to $6$.
  \label{def:LSF}
\end{definition}

In \cite{Ass15}, we show that for a connected graph, if axiom $5$ holds, then locally Schur multiplicity-free implies $\LSF_m$ for all $m$. That is, locally a Schur function implies globally a Schur function. Translating the axioms to local positivity conditions, axiom $1$ is equivalent to $\LSF_3$, Figure~\ref{fig:lambda4} is equivalent to $\LSF_4$, Figure~\ref{fig:lambda5} is equivalent to $\LSF_5$, and axiom $6$ is equivalent to $\LSF_6$. Axioms $2$ and $5$ are the commutativity axioms.

\begin{definition}
  A signed, colored graph $\G = (V,\sigma,E)$ of type $(n,N)$ is \emph{Schur positive for degree $m$}, denoted by $\LSP_m$, if for every $m-2<i<n$ and every connected component $\C$ of $(V,\sigma,E_{i-(m-3)} \cup \cdots \cup E_i)$, the restricted degree $m$ generating function
  \begin{equation}
    \sum_{v \in \C} Q_{\sigma(v)_{i-(m-2),\ldots,i}}(X)
    \label{eqn:LSP}
  \end{equation}
  is symmetric and Schur positive. 
  \label{def:LSP}
\end{definition}

We can now see how axiom $3$ fits into this picture. In some sense, it is redundant since it is implied by axiom $4$. However, since we aim to generalize dual equivalence graphs, it serves as a useful guide.

\begin{proposition}
  A signed, colored graph of type $(n,n)$ with $\LSF_3$ and $\LSP_4$ satisfies axiom $3$.
\end{proposition}

\begin{proof}
  We claim that axiom $3$ is equivalent to the statement that for all $\{w,x\} \in E_{i}$, at least one of $w$ or $x$ has an $i\pm 1$-neighbor. To see the equivalence, note that by axiom $1$ (equivalently, $\LSF_3$), neither $w$ nor $x$ will admit an $\imo$-neighbor if and only if $\sigma(w)_{\imt} = \sigma(w)_{\imo}$ and $\sigma(x)_{\imt} = \sigma(x)_{\imo}$. By axioms 1 and 2, this implies $\sigma(w)_{\imt} = \sigma(w)_{\imo} = -\sigma(x)_{\imo} = - \sigma(x)_{\imt}$. The analogous argument holds for $\ipo$.
  
  If neither $w$ nor $E_i(w)$ has an $\imo$-neighbor (resp. $\ipo$-neighbor) then the connected component of $E_{\imo} \cup E_{i}$ (resp. $E_{i} \cup E_{\ipo}$) containing $w$ consists solely of $w$ and $E_i(w)$ forcing the restricted degree 4 generating function to be $Q_{++-}+Q_{--+}$, which is not Schur positive. Therefore, assuming $\LSP_4$, if $w$ has an $i$-neighbor, then at least one of $w$ and $E_i(w)$ has an $\imo$, and so axiom $3$ holds.
\end{proof}

The requirement that the graph be of type $(n,n)$ is necessary in order to ensure that $E_{\ipo}$ edges exist in the graph. For a graph is of type $(n,N)$ with $n<N$, neither local Schur positivity nor even axiom $4$ is enough to ensure axiom $3$.

The graphs mentioned above for LLT polynomials, Macdonald polynomials, and $k$-Schur functions all satisfy axioms $1, 2, 3$ and $5$ and are locally Schur positive. Therefore these conditions provide a good starting point for a more general study of signed, colored graphs.

\begin{definition}
  A signed, colored graph $\G = (V,\sigma,E)$ of type $(n,N)$ is \emph{locally Schur positive} if it satisfies dual equivalence graphs axioms $1$, $2$, $3$, and $5$ and has $\LSP_m$ for $m = 4,5,6$.
  \label{def:LSPG}
\end{definition}

The idea for transforming a locally Schur positive graph $\G$ into a dual equivalence graph $\widetilde{\G}$ with the same quasisymmetric generating function is to construct a sequence of signed, colored graphs $\G = \G_2 , \ldots, \G_{\nmo} = \widetilde{\G}$ on the same vertex set such that $\G_{\imo}$ is a locally Schur positive and the $(i,N)$-restriction of $\G_{\imo}$ is a dual equivalence graph. We wish to define a transformation by identifying two $i$-edges on the same connected component of $E_{2} \cup\cdots\cup E_{i}$ and swapping the connections in the unique way that maintains the reversal of $\sigma_{\imo}$ and $\sigma_i$ required by dual equivalence axiom $1$, as indicated in Figure~\ref{fig:swap}.

\begin{figure}[ht]
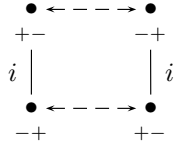

  \begin{displaymath}
    \begin{array}{\cs{7}c}
      \sbull{a}{+-} & \sbull{b}{-+} \\[5ex]
      \sbull{c}{-+} & \sbull{d}{+-}
    \end{array}
    \psset{nodesep=3pt,linewidth=.1ex}
    \ncline {aa}{c} \nbput{i}
    \ncline {bb}{d} \naput{i}
    \ncline[linestyle=dashed] {<->}{a}{b} 
    \ncline[linestyle=dashed] {<->}{c}{d} 
  \end{displaymath}
  \caption{\label{fig:swap} An illustration of how two $i$-edges are swapped in the transformation process.}
\end{figure}

Since axiom $1$ implies $\LSF_3$, this degree is resolved, so we focus on three transformations to address $\LSF_4$, $\LSF_5$, and $\LSF_6$, respectively. These transformations are depicted in Figure~\ref{fig:involutions}. To see why these transformations are natural, let us consider graphs is increasing sizes that satisfy $\LSF_i$ for $i<n$ and $\LSP_n$.

\begin{figure}[ht]
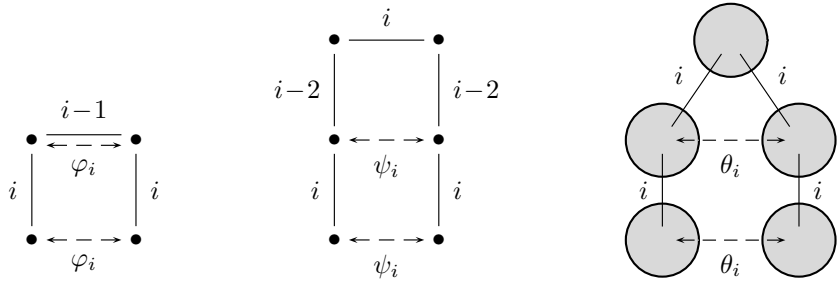

  \begin{displaymath}
    \begin{array}{\cs{8}c}
      & \\[6ex]
      \rnode{C}{\B} & \rnode{D}{\B}  \\[6ex] 
      \rnode{B}{\B} & \rnode{E}{\B}
    \end{array}
    \hspace{6em}
    \begin{array}{\cs{8}c}
      \rnode{a2}{\B} & \rnode{b2}{\B} \\[6ex] 
      \rnode{e2}{\B} & \rnode{f2}{\B} \\[6ex] 
      \rnode{h2}{\B} & \rnode{i2}{\B} 
    \end{array}
    \hspace{7em}
    \begin{array}{\cs{5}\cs{5}c}
      & \rnode{X}{%
        \psset{xunit=1ex}
        \psset{yunit=1ex}
        \pspicture(0,0)(1,1)
        \pscircle[fillstyle=solid,fillcolor=lightgray](0.5,0.5){0.5}
        \endpspicture} & \\[6ex]
      \rnode{Y1}{%
        \psset{xunit=1ex}
        \psset{yunit=1ex}
        \pspicture(0,0)(1,1)
        \pscircle[fillstyle=solid,fillcolor=lightgray](0.5,0.5){0.5}
        \endpspicture} & & \rnode{Y2}{%
        \psset{xunit=1ex}
        \psset{yunit=1ex}
        \pspicture(0,0)(1,1)
        \pscircle[fillstyle=solid,fillcolor=lightgray](0.5,0.5){0.5}
        \endpspicture} \\[6ex]
      \rnode{Z1}{%
        \psset{xunit=1ex}
        \psset{yunit=1ex}
        \pspicture(0,0)(1,1)
        \pscircle[fillstyle=solid,fillcolor=lightgray](0.5,0.5){0.5}
        \endpspicture} & & \rnode{Z2}{%
        \psset{xunit=1ex}
        \psset{yunit=1ex}
        \pspicture(0,0)(1,1)
        \pscircle[fillstyle=solid,fillcolor=lightgray](0.5,0.5){0.5}
        \endpspicture}
    \end{array}
    \psset{nodesep=3pt,linewidth=.1ex}
    \ncline[offset=2pt] {C}{D} \naput{\imo}
    \ncline {B}{C} \naput{i}
    \ncline {D}{E} \naput{i}
    \ncline[offset=-2pt,linestyle=dashed] {<->}{C}{D} \nbput{\varphi_i}
    \ncline[linestyle=dashed] {<->}{B}{E} \nbput{\varphi_i}
    \ncline {a2}{b2} \naput{i}
    \ncline {a2}{e2} \nbput{\imt}
    \ncline {b2}{f2} \naput{\imt}
    \ncline {e2}{h2} \nbput{i}
    \ncline {f2}{i2} \naput{i}
    \ncline[linestyle=dashed] {<->}{e2}{f2} \nbput{\psi_i}
    \ncline[linestyle=dashed] {<->}{h2}{i2} \nbput{\psi_i}
    \ncline {Y1}{X} \naput{i}
    \ncline {X}{Y2} \naput{i}
    \ncline {Y1}{Z1} \nbput{i}
    \ncline {Y2}{Z2} \naput{i}
    \ncline[linestyle=dashed]{<->} {Y1}{Y2} \nbput{\theta_i}
    \ncline[linestyle=dashed]{<->} {Z1}{Z2} \nbput{\theta_i}
  \end{displaymath}
  \caption{\label{fig:involutions} Illustrations of the involutions $\varphi_i$, $\psi_i$, and $\theta_i$ used to redefine $E_i$.}
\end{figure}

Given a connected signed, colored graph of degree $4$ having $\LSF_3$, if the graph is locally Schur positive (in this case, meaning has $\LSP_4$) but the generating function is not a single Schur function, then there are only a few possibilities for the generating function, similar to Proposition~\ref{prop:LSP6}.

\begin{proposition}
  Let $\G$ be a connected, locally Schur positive graph of degree $4$. Then the quasisymmetric generating function for $\G$ is either $s_{\lambda}$ for some partition $\lambda$ of $4$ or has the form $s_{(3,1)} + k s_{(2,2)}, k s_{(2,2)}, s_{(2,1,1)} + k s_{(2,2)}$, for some positive integer $k$.
  \label{prop:LSP4}
\end{proposition}

\begin{proof}
  By axiom $1$ (equivalently $\LSF_3$), every vertex has at most one $2$-edge and at most one $3$-edge. Therefore edges must alternate between $2$ and $3$, so $\G$ is either a path or a closed loop. Using axiom $1$ again, a vertex has no edges if and only if it has signature $+++$ or $---$, in which case the generating function is $s_{(4)}$ or $s_{(1,1,1,1)}$, respectively. Moreover, a vertex has a $2$-edge but not a $3$-edge if and only if it has signature $-++$ or $+--$, and similarly, a vertex has a $3$-edge but not a $2$-edge if and only if it has signature $++-$ or $--+$. Therefore $\G$ is a closed loop if and only if all vertices admit both a $2$ and a $3$ edge if and only if all vertices have signature $+-+$ or $-+-$. In this case, axiom $1$ ensures that these signatures alternate across edges, and so there must be an even number of vertices. Therefore the generating function is $kQ_{+-+} + kQ_{-+-} = ks_{(2,2)}$ for some positive integer $k$. If $\G$ is a path, then there are exactly two vertices that admit one color edge and not the other, and all other vertices admit both a $2$ and a $3$ edge. The two vertices admitting a single color edge contribute only to $s_{(3,1)}$ or $s_{(2,1,1)}$, and so both must contribute to the same function to ensure local Schur positivity. By axioms $1$ and $3$, if $\sigma(w)=++-$, then $w$ has a $3$-edge and not a $2$-edge, and $\sigma(E_3(w)) = +-+$, and similarly if $\sigma(v)=-++$, then $v$ has a $2$-edge and not a $3$-edge, and $\sigma(E_2(v)) = +-+$. Therefore these either agree, meaning $E_3(w)=E_2(v)$ giving generating function $s_{(3,1)}$, or we may toggle by pairs of $2$ and $3$ edges $k$ times, with vertices having signatures $+-+$, $-+-$, $+-+$, giving generating function $s_{(3,1)} + k s_{(2,2)}$.
\end{proof}

For examples when $k>0$, see Figure~\ref{fig:degree4}. The top left graph has generating function $s_{(3,1)}+s_{(2,2)}$, the bottom left graph has generating function $s_{(2,1,1)} + s_{(2,2)}$ and the right graph has generating function $2s_{(2,2)}$.

\begin{figure}[ht]
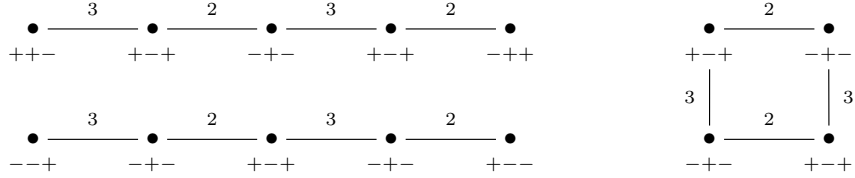

  \begin{displaymath}
    \begin{array}{\cs{7}\cs{7}\cs{7}\cs{7}\cs{7}\cs{7}\cs{7}c}
      \sbull{c1}{++-} & \sbull{c2}{+-+} & \sbull{c3}{-+-} & \sbull{c4}{+-+} & \sbull{c5}{-++} & & \sbull{e1}{+-+} & \sbull{e2}{-+-} \\[6ex]
      \sbull{d1}{--+} & \sbull{d2}{-+-} & \sbull{d3}{+-+} & \sbull{d4}{-+-} & \sbull{d5}{+--} & & \sbull{f1}{-+-} & \sbull{f2}{+-+} 
    \end{array}
    \psset{nodesep=3pt,linewidth=.1ex}
    \everypsbox{\scriptstyle}
    \ncline {c1}{c2} \naput{3}
    \ncline {c2}{c3} \naput{2}
    \ncline {c3}{c4} \naput{3}
    \ncline {c4}{c5} \naput{2}
    \ncline {d1}{d2} \naput{3}
    \ncline {d2}{d3} \naput{2}
    \ncline {d3}{d4} \naput{3}
    \ncline {d4}{d5} \naput{2}
    \ncline {e1}{e2} \naput{2}
    \ncline {e1e1}{f1} \nbput{3}
    \ncline {e2e2}{f2} \naput{3}
    \ncline {f1}{f2} \naput{2}
  \end{displaymath}
  \caption{\label{fig:degree4} Possible locally Schur positive graphs of degree $4$.}
\end{figure}

Having $\LSF_4$ restricts the lengths of $2$-color strings by forcing the number of edges of a nontrivial connected component of $E_{\imo} \cup E_{i}$ to be two, either with three distinct vertices (corresponding to $s_{(3,1)}$ or $s_{(2,1,1)}$) or forming a cycle with two vertices (corresponding to $s_{(2,2)}$). The map $\varphi_i$ swaps $i$-edges on connected components of $E_{\imo} \cup E_i$ with more than two edges. This map is studied in Section~\ref{sec:LSP4}.

Given a connected signed, colored graph of degree $5$ having $\LSF_3$ and $\LSF_4$, if the graph is locally Schur positive (in this case, meaning has $\LSP_5$) but the generating function is not a single Schur function, then there are again only a few possibilities for the generating function.

\begin{proposition}
  Let $\G$ be a connected, locally Schur positive graph of degree $5$ having $\LSF_4$. Then the quasisymmetric generating function for $\G$ is either $s_{\lambda}$ for some partition $\lambda$ of $5$ or has the form $s_{3,2} + k s_{3,1,1}, k s_{3,1,1}, s_{2,2,1} + k s_{3,1,1}$, for some positive integer $k$.
  \label{prop:LSP5}
\end{proposition}

\begin{proof}
  If $\G$ is a single vertex with no edges, then by axiom $1$, the vertex has signature $++++$ or $----$, in which case the generating function is $s_{(5)}$ or $s_{(1,1,1,1,1)}$. If $\G$ has a vertex, say $u$, that admits a $2$-edge but neither a $3$-edge nor a $4$-edge, the by axiom $1$, $u$ must have signature $-+++$ or $+---$. Suppose the former is the case. By $\LSF_4$ for $E_2 \cup E_3$, there must be vertices $v = E_2(u)$ and $w = E_3(v)$ such that $w$ does not have a $2$-edge. By axioms $1$ and $3$, we have $\sigma(v)=+-++$, and by those same axioms again we have $\sigma(w) = ++-+$. Finally, by $\LSF_4$ for $E_3 \cup E_4$, we must have $x = E_4(w)$ with $\sigma(v)=+++-$. In particular, the generating function of $\G$ is $s_{(4,1)}$. The same argument applies when $u$ has signature $+---$ by multiplying the signature component wise by $-1$ to conclude the $\G$ has generating function $s_{(2,1,1,1)}$. From these two cases, reversing signature shows that if $\G$ has a vertex admitting a $4$-edge but neither a $2$-edge nor a $3$-edge then it has generating function $s_{(4,1)}$ or $s_{(2,1,1,1)}$. Therefore we may assume every vertex of $\G$ has either a $3$-edge (possibly with other edges as well) or both a $2$-edge and a $4$-edge.

  By axiom $1$, a vertex $u$ has a $2$-edge and a $3$-edge but no $4$-edge if and only if it has signature $+-++$ or $-+--$. If $E_2(u) \neq E_3(u)$, then $E_2(u)$ has signature $-+++$ or $+---$, which, by axiom $1$, implies it has no $3$-edge or $4$-edge contradicting the assumption on $\G$. Therefore we must have $E_2(u) = E_3(u)$ and this vertex, say $v$, has signature $-+-+$ or $+-+-$, respectively. By axiom $1$, $v$ has a $4$-edge, say $w = E_4(v)$, and by $\LSP_4$, $\sigma(w) = -++-$ or $+--+$, respectively. Therefore, by axiom $1$, $w$ has a $2$-edge but no $3$-edge, and $x = E_{2}(w)$ has signature $+-+-$ or $-+-+$, respectively. Now, by axiom $1$, $x$ has a $3$-edge and a $4$-edge. If $E_3(x) = E_4(x)$, then these five vertices comprise all of $\G$ which has generating function $s_{(3,2)}$ or $s_{(2,2,1)}$, respectively. Otherwise, by $\LSP_4$, $E_3(x)$ has signature $++--$ or $--++$, forcing the generating function to include $s_{(3,1,1)}$. We may continue following the $2$-edges and $4$-edges, alternating along and sprouting a $3$-edge every other step, and this process can only terminate when the $3$-edge and $4$-edge coincide. Moreover, it must do this after an even number, say $2k$, of lone $3$-edges sprouting up, and the resulting generating function will be $s_{(3,2)} + ks_{(3,1,1)}$ or $s_{(2,2,1)} + ks_{(3,1,1)}$, respectively. Reversing signatures, the same argument shows that beginning with a vertex that has a $4$-edge and a $3$-edge but no $2$-edge results in the same possibilities. If no vertex on $\G$ has a $3$-edge and exactly one other edge, then the same case analysis shows that the $2$-edges and $4$-edge must alternate, sprouting an even number, say $2k$, of lone $3$-edges at every other vertex, and the resulting graph has generating function $k s_{(3,1,1)}$.
\end{proof}

For examples, see Figure~\ref{fig:degree5}. The top left graph has generating function $s_{(3,2)}+s_{(3,1,1)}$, the bottom left graph has generating function $s_{(2,2,1)} + s_{(3,1,1)}$ and the right graph has generating function $2s_{(3,1,1)}$.

\begin{figure}[ht]
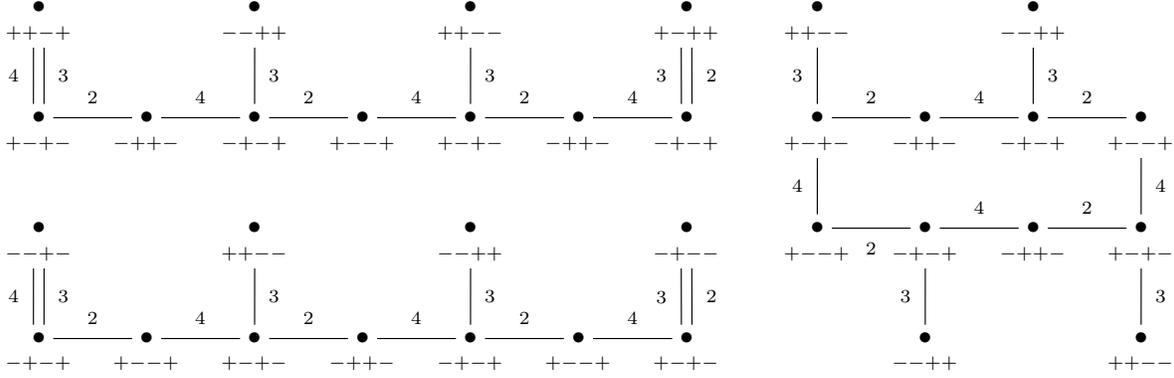

  \begin{displaymath}
    \begin{array}{\cs{6}\cs{6}\cs{6}\cs{6}\cs{6}\cs{6}\cs{8}\cs{6}\cs{6}\cs{6}c}
      \sbull{b1}{++-+} &                  & \sbull{b3}{--++} &                  & \sbull{b5}{++--} &                  & \sbull{b7}{+-++} & \sbull{a1}{++--} &                  & \sbull{A1}{--++} &                  \\[6ex]
      \sbull{c1}{+-+-} & \sbull{c2}{-++-} & \sbull{c3}{-+-+} & \sbull{c4}{+--+} & \sbull{c5}{+-+-} & \sbull{c6}{-++-} & \sbull{c7}{-+-+} & \sbull{a2}{+-+-} & \sbull{z2}{-++-} & \sbull{A2}{-+-+} & \sbull{A3}{+--+} \\[6ex]
      \sbull{B1}{--+-} &                  & \sbull{B3}{++--} &                  & \sbull{B5}{--++} &                  & \sbull{B7}{-+--} & \sbull{a3}{+--+} & \sbull{a4}{-+-+} & \sbull{z4}{-++-} & \sbull{A4}{+-+-} \\[6ex]
      \sbull{C1}{-+-+} & \sbull{C2}{+--+} & \sbull{C3}{+-+-} & \sbull{C4}{-++-} & \sbull{C5}{-+-+} & \sbull{C6}{+--+} & \sbull{C7}{+-+-} &                  & \sbull{a5}{--++} &                  & \sbull{A5}{++--} 
    \end{array}
    \psset{nodesep=3pt,linewidth=.1ex}
    \everypsbox{\scriptstyle}
    \ncline[offset=2pt] {b1b1}{c1} \naput{3}
    \ncline[offset=2pt] {c1}{b1b1} \naput{4}
    \ncline {b3b3}{c3} \naput{3}
    \ncline {b5b5}{c5} \naput{3}
    \ncline[offset=2pt] {b7b7}{c7} \naput{2}
    \ncline[offset=2pt] {c7}{b7b7} \naput{3}
    \ncline {c1}{c2} \naput{2}
    \ncline {c2}{c3} \naput{4}
    \ncline {c3}{c4} \naput{2}
    \ncline {c4}{c5} \naput{4}
    \ncline {c5}{c6} \naput{2}
    \ncline {c6}{c7} \naput{4}
    \ncline[offset=2pt] {B1B1}{C1} \naput{3}
    \ncline[offset=2pt] {C1}{B1B1} \naput{4}
    \ncline {B3B3}{C3} \naput{3}
    \ncline {B5B5}{C5} \naput{3}
    \ncline[offset=2pt] {B7B7}{C7} \naput{2}
    \ncline[offset=2pt] {C7}{B7B7} \naput{3}
    \ncline {C1}{C2} \naput{2}
    \ncline {C2}{C3} \naput{4}
    \ncline {C3}{C4} \naput{2}
    \ncline {C4}{C5} \naput{4}
    \ncline {C5}{C6} \naput{2}
    \ncline {C6}{C7} \naput{4}
    \ncline {A1A1}{A2} \naput{3}    
    \ncline {A2}{A3} \naput{2}    
    \ncline {A3A3}{A4} \naput{4}    
    \ncline {A4A4}{A5} \naput{3}    
    \ncline {a1a1}{a2} \nbput{3}    
    \ncline {a2a2}{a3} \nbput{4}    
    \ncline {a3}{a4} \nbput{2}    
    \ncline {a4a4}{a5} \nbput{3}    
    \ncline {a2}{z2} \naput{2}    
    \ncline {z2}{A2} \naput{4}    
    \ncline {a4}{z4} \naput{4}    
    \ncline {z4}{A4} \naput{2}    
  \end{displaymath}
  \caption{\label{fig:degree5} Possible locally Schur positive graphs with $\LSF_4$ of degree $5$.}
\end{figure}

Having $\LSF_5$ forces the number of edges of a nontrivial connected component of $E_{\imt} \cup E_{i}$ to be one (corresponding to $s_{(4,1)}$ or $s_{(2,1,1,1)}$) or four, where there are either five distinct vertices (corresponding to $s_{(3,2)}$ or $s_{(2,2,1)}$) or four vertices forming a cycle (corresponding to $s_{(3,1,1)}$. The map $\psi_i$ swaps $i$-edges on connected components of $E_{\imt} \cup E_i$ with more than four edges. This map is studied in Section~\ref{sec:LSP5}.

Given a connected signed, colored graph of degree $6$ having $\LSF_3, \LSF_4$ and $\LSF_5$, a stronger statement holds. This follows from Theorem~3.11 and Corollary~3.14 of \cite{Ass15}, which we recall below.

\begin{theorem}[\cite{Ass15}]
  Let $\G$ be a connected signed, colored graph of type $(\npo,\npo)$ satisfying axioms $1$ through $5$ such that each connected component of the $(n,n)$-restriction of $\G$ is isomorphic to a standard dual equivalence graph. Then there exists a morphism $\phi$ from $\G$ to $\G_{\lambda}$ for some unique partition $\lambda$ of $\npo$. 
\label{thm:cover}
\end{theorem}

The morphism of Theorem~\ref{thm:cover} is necessarily surjective, though in general it need not be injective. For example, Figure~\ref{fig:gregg} gives such a graph that is a two-fold cover of $\G_{(3,2,1)}$. 

\begin{corollary}[\cite{Ass15}]
  Let $\G$ be a signed, colored graph of degree $n$ satisfying dual equivalence axioms $1$ through $5$ such that the $(n-1)$-restriction satisfies dual equivalence axiom $6$ as well. Then the quasisymmetric generating function for $\G$ is given by
  \begin{equation}
    \sum_{v \in \G} Q_{\sigma(v)}(X) = k s_{\lambda}(X)
  \end{equation}
  for some partition $\lambda$ of $n$ and some positive integer $k$.
  \label{cor:fibers}
\end{corollary}

Applying Corollary~\ref{cor:fibers} to a graph of degree $6$, we have the following.

\begin{proposition}
  Let $\G$ be a connected graph of degree $6$ having $\LSF_3, \LSF_4$ and $\LSF_5$. The quasisymmetric generating function for $\G$ is either $s_{\lambda}$ for some partition $\lambda$ of $6$ or $k s_{(3,2,1)}$ for some positive integer $k$.
  \label{prop:LSP6}
\end{proposition}

\begin{proof}
  If $\G$ satisfies dual equivalence axiom $6$, then by Theorem~\ref{thm:isomorphic}, $\G$ is isomorphic to $\G_{\lambda}$ for some partition $\lambda$ of $6$, and so its quasisymmetric generating function is $s_{\lambda}$. Otherwise, if $\G$ fails dual equivalence axiom $6$, then by Corollary~\ref{cor:fibers}, its quasisymmetric generating function is $k s_{\lambda}$ for some positive integer $k$. However, if $\lambda$ is a partition of $6$ other than $(3,2,1)$, then $\lambda$ has at most $2$ removable corners, so there are at most two possible isomorphism classes for the $(5,6)$-restriction of $\G$. Therefore each can occur at most once, so $\G = \G_{\lambda}$.
\end{proof}

For example, Figure~\ref{fig:gregg} gives such a graph with generating function $2s_{(3,2,1)}$.

\begin{figure}[ht]
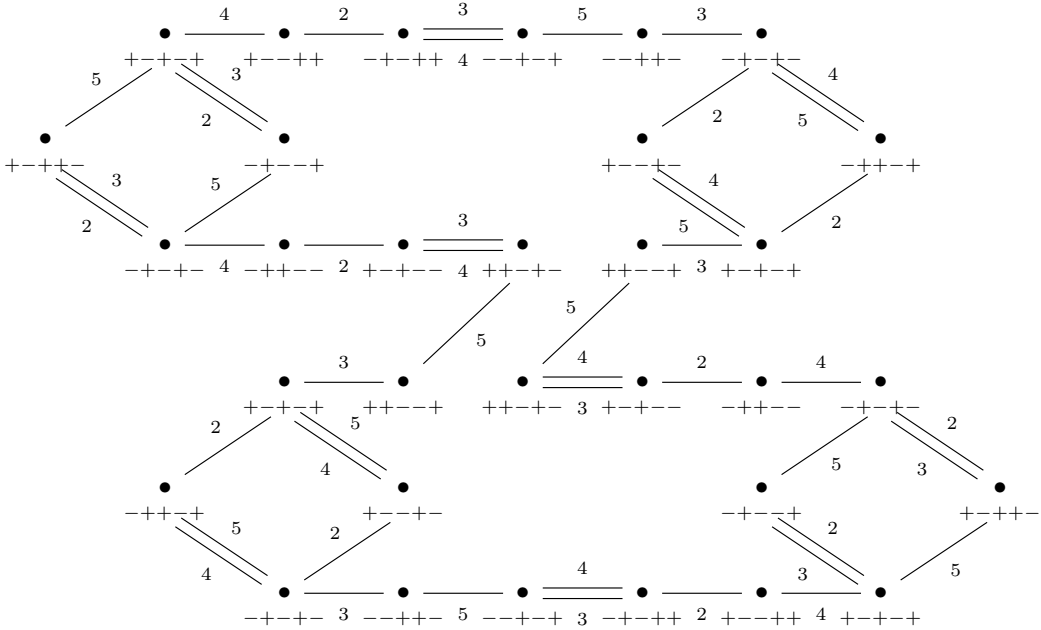

  \begin{displaymath}
    \begin{array}{\cs{7} \cs{7} \cs{7} \cs{7} \cs{7} \cs{7} \cs{7} \cs{7} c}
      &
      \sbull{c2}{+-+-+} &
      \sbull{b2}{+--++} &
      \sbull{a3}{-+-++} &
      \sbull{a4}{--+-+} &
      \sbull{b5}{--++-} &
      \sbull{c5}{-+-+-} &
      & \\[2\cellsize]
      \sbull{d3}{+-++-} & &
      \sbull{d1}{-+--+} & & &
      \sbull{d4}{+--+-} & &
      \sbull{d6}{-++-+} 
      & \\[2\cellsize]
      &
      \sbull{e2}{-+-+-} &
      \sbull{f2}{-++--} &
      \sbull{g3}{+-+--} &
      \sbull{g4}{++-+-} &
      \sbull{f5}{++--+} &
      \sbull{e5}{+-+-+} & 
      & \\[3\cellsize]
      & &
      \sbull{xe5}{+-+-+} & 
      \sbull{xf5}{++--+} &
      \sbull{xg4}{++-+-} &
      \sbull{xg3}{+-+--} &
      \sbull{xf2}{-++--} &
      \sbull{xe2}{-+-+-} & 
      \\[2\cellsize]
      &
      \sbull{xd6}{-++-+} & &
      \sbull{xd4}{+--+-} & & &
      \sbull{xd1}{-+--+}  & &
      \sbull{xd3}{+-++-}
      \\ [2\cellsize]
      & &
      \sbull{xc5}{-+-+-} &
      \sbull{xb5}{--++-} &
      \sbull{xa4}{--+-+} &
      \sbull{xa3}{-+-++} &
      \sbull{xb2}{+--++} &
      \sbull{xc2}{+-+-+} & 
    \end{array}
    \psset{nodesep=5pt,linewidth=.1ex}
    \everypsbox{\scriptstyle}
    \ncline[offset=2pt]{a3}{a4} \naput{3}
    \ncline[offset=2pt]{a4}{a3} \naput{4}
    \ncline            {b2}{a3} \naput{2}
    \ncline            {a4}{b5} \naput{5}
    \ncline            {c2}{b2} \naput{4}
    \ncline            {b5}{c5} \naput{3}
    \ncline[offset=2pt]{d1}{c2c2} \naput{2}
    \ncline[offset=2pt]{c2c2}{d1} \naput{3}
    \ncline            {c2c2}{d3} \nbput{5}
    \ncline            {d4}{c5c5} \nbput{2}
    \ncline[offset=2pt]{c5c5}{d6} \naput{4}
    \ncline[offset=2pt]{d6}{c5c5} \naput{5}
    \ncline            {d1d1}{e2} \nbput{5}
    \ncline[offset=2pt]{e2}{d3d3} \naput{2}
    \ncline[offset=2pt]{d3d3}{e2} \naput{3}
    \ncline[offset=2pt]{d4d4}{e5} \naput{4}
    \ncline[offset=2pt]{e5}{d4d4} \naput{5}
    \ncline            {e5}{d6d6} \nbput{2}
    \ncline            {e2}{f2} \nbput{4}
    \ncline            {f5}{e5} \nbput{3}
    \ncline            {f2}{g3} \nbput{2}
    \ncline            {g4g4}{xf5} \naput{5}
    \ncline[offset=2pt]{g3}{g4} \naput{3}
    \ncline[offset=2pt]{g4}{g3} \naput{4}
    \ncline[offset=2pt]{xa3}{xa4} \naput{3}
    \ncline[offset=2pt]{xa4}{xa3} \naput{4}
    \ncline            {xb2}{xa3} \naput{2}
    \ncline            {xa4}{xb5} \naput{5}
    \ncline            {xc2}{xb2} \naput{4}
    \ncline            {xb5}{xc5} \naput{3}
    \ncline[offset=2pt]{xd1xd1}{xc2} \naput{2}
    \ncline[offset=2pt]{xc2}{xd1xd1} \naput{3}
    \ncline            {xc2}{xd3xd3} \nbput{5}
    \ncline            {xd4xd4}{xc5} \nbput{2}
    \ncline[offset=2pt]{xc5}{xd6xd6} \naput{4}
    \ncline[offset=2pt]{xd6xd6}{xc5} \naput{5}
    \ncline            {xd1}{xe2xe2} \nbput{5}
    \ncline[offset=2pt]{xe2xe2}{xd3} \naput{2}
    \ncline[offset=2pt]{xd3}{xe2xe2} \naput{3}
    \ncline[offset=2pt]{xd4}{xe5xe5} \naput{4}
    \ncline[offset=2pt]{xe5xe5}{xd4} \naput{5}
    \ncline            {xe5xe5}{xd6} \nbput{2}
    \ncline            {xe2}{xf2} \nbput{4}
    \ncline            {xf5}{xe5} \nbput{3}
    \ncline            {xf2}{xg3} \nbput{2}
    \ncline            {xg4}{f5f5} \naput{5}
    \ncline[offset=2pt]{xg3}{xg4} \naput{3}
    \ncline[offset=2pt]{xg4}{xg3} \naput{4}
  \end{displaymath}
  \caption{\label{fig:gregg}The smallest graph satisfying dual equivalence graph axioms $1-5$ but not $6$.}
\end{figure}

The hypotheses for Corollary~\ref{cor:fibers} are quite strong. One can construct a signed, colored graph of degree $7$ satisfying dual equivalence axioms $1$ through $5$ such that the generating function has distinct Schur functions appearing with positive multiplicity. Nevertheless, computer exploration for graphs up to degree $12$ suggests the following conjecture, which might be proved using the transformations described in this paper.

\begin{conjecture}
  For any signed, colored graph $\G$ of degree $n$ satisfying dual equivalence axioms $1$ through $5$, the quasisymmetric generating function for $\G$ is symmetric and Schur positive.
  \label{conj:cover}
\end{conjecture}

%
\section{Resolving axiom $6$}
%
\label{sec:LSP6}

We begin our study of locally Schur positive graphs by considering graphs that satisfy dual equivalence axioms $1$ through $5$ but not necessarily satisfying axiom $6$. Precisely, let $\G$ be a locally Schur positive graph of type $(n,N)$ such that the $(i,N)$-restriction of $\G$ is a dual equivalence graph and the $(\ipo,N)$-restriction of $\G$ satisfies dual equivalence axiom $4$ (equivalently, has $\LSF_4$ and $\LSF_5$). Having $\LSF_6$, which is equivalent to axiom $6$, restricts the size of $E_2 \cup \cdots \cup E_{\imo}$ isomorphism classes of a connected component of $E_2 \cup \cdots \cup E_i$ to be one. The map $\theta_i$ swaps $i$-edges on connected components of $E_2 \cup \cdots \cup E_i$ with more than one member of a given $E_2 \cup \cdots \cup E_{\imo}$ isomorphism class. 

By Theorem~\ref{thm:cover}, for each connected component $\mathcal{H}$ of the $(\ipo,\ipo)$-restriction of $\G$, there exists a morphism $\phi$ from $\mathcal{H}$ to $\G_{\lambda}$ for a unique partition $\lambda$ of $\ipo$, and the fiber over each vertex of $\G_{\lambda}$ has the same cardinality. By Theorem~\ref{thm:isomorphic}, $\mathcal{H}$ satisfies axiom $6$ if and only if $\phi$ is an isomorphism. Figure~\ref{fig:tab-gregg} highlights the isomorphism classes of the $(5,6)$-restriction of the graph in Figure~\ref{fig:gregg} that is a $2$-fold cover of $\G_{(3,2,1)}$.

\begin{figure}[ht]
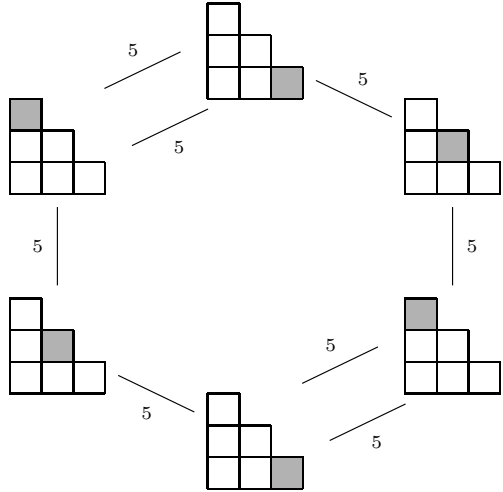

  \begin{displaymath}
    \begin{array}{\cs{9}\cs{9}c}
      & \rnode{a}{\tableau{\e \\ \e & \e \\ \e & \e & \cb}} & \\[\cellsize]
      \rnode{b}{\tableau{\cb \\ \e & \e \\ \e & \e & \e}} & 
      & \rnode{c}{\tableau{\e \\ \e & \cb \\ \e & \e & \e}} \\[5\cellsize]
      \rnode{C}{\tableau{\e \\ \e & \cb \\ \e & \e & \e}} & 
      & \rnode{B}{\tableau{\cb \\ \e & \e \\ \e & \e & \e}} \\[\cellsize]
      & \rnode{A}{\tableau{\e \\ \e & \e \\ \e & \e & \cb}} & 
    \end{array}
    \psset{nodesep=5pt,linewidth=.1ex}
    \everypsbox{\scriptstyle}
    \ncline[offset=12pt] {a}{b} \naput{5}
    \ncline[offset=12pt] {b}{a} \naput{5}
    \ncline {a}{c} \naput{5}
    \ncline {b}{C} \nbput{5}
    \ncline {c}{B} \naput{5}
    \ncline {C}{A} \nbput{5}
    \ncline[offset=12pt] {B}{A} \naput{5}
    \ncline[offset=12pt] {A}{B} \naput{5}
  \end{displaymath}
  \caption{\label{fig:tab-gregg}The $(5,6)$-restriction of Figure~\ref{fig:gregg} highlighting the two-fold cover of $\G_{(3,2,1)}$.}
\end{figure}

We define an involution $\theta_i$ on vertices of $\mathcal{H}$ admitting an $i$-neighbor as indicated in Figure~\ref{fig:theta} and use it to redefine $i$-edges that are in violation of axiom $6$.

\begin{figure}[ht]
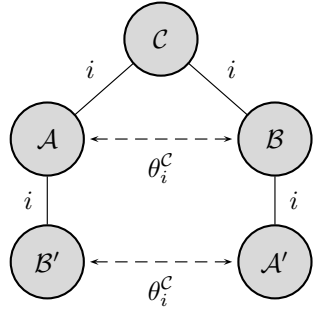

  \begin{displaymath}
    \begin{array}{\cs{9}\cs{9}c}
      & \rnode{X}{%
        \psset{xunit=1ex}
        \psset{yunit=1ex}
        \pspicture(0,0)(1,1)
        \pscircle[fillstyle=solid,fillcolor=lightgray](0.5,0.5){0.5}
        \rput(0.5,0.5){$\C$}
        \endpspicture} & \\[6ex]
      \rnode{Y1}{%
        \psset{xunit=1ex}
        \psset{yunit=1ex}
        \pspicture(0,0)(1,1)
        \pscircle[fillstyle=solid,fillcolor=lightgray](0.5,0.5){0.5}
        \rput(0.5,0.5){$\mathcal{A}$}
        \endpspicture} & & \rnode{Y2}{%
        \psset{xunit=1ex}
        \psset{yunit=1ex}
        \pspicture(0,0)(1,1)
        \pscircle[fillstyle=solid,fillcolor=lightgray](0.5,0.5){0.5}
        \rput(0.5,0.5){$\mathcal{B}$}
        \endpspicture} \\[8ex]
      \rnode{Z1}{%
        \psset{xunit=1ex}
        \psset{yunit=1ex}
        \pspicture(0,0)(1,1)
        \pscircle[fillstyle=solid,fillcolor=lightgray](0.5,0.5){0.5}
        \rput(0.5,0.5){$\mathcal{B}^{\prime}$}
        \endpspicture} & & \rnode{Z2}{%
        \psset{xunit=1ex}
        \psset{yunit=1ex}
        \pspicture(0,0)(1,1)
        \pscircle[fillstyle=solid,fillcolor=lightgray](0.5,0.5){0.5}
        \rput(0.5,0.5){$\mathcal{A}^{\prime}$}
        \endpspicture}
    \end{array}
    \psset{nodesep=12pt,linewidth=.1ex}
    \ncline[nodesep=8pt] {Y1}{X} \naput{i}
    \ncline[nodesep=8pt] {X}{Y2} \naput{i}
    \ncline {Y1}{Z1} \nbput{i}
    \ncline {Y2}{Z2} \naput{i}
    \ncline[nodesep=14pt,linestyle=dashed]{<->} {Y1}{Y2} \nbput{\theta_i^{\C}}
    \ncline[nodesep=14pt,linestyle=dashed]{<->} {Z1}{Z2} \nbput{\theta_i^{\C}}
  \end{displaymath}
  \caption{\label{fig:theta} An illustration of the involution $\theta_i^{\C}$ where $\mathcal{A} \cong \mathcal{A}^{\prime}$ and $\mathcal{B} \cong \mathcal{B}^{\prime}$.}
\end{figure}

\begin{definition}
  Let $\mathcal{H}$ be a connected component of the $(\ipo,\ipo)$-restriction of $\G$ and let $\C$ be a connected component of the $(i,i)$-restriction of $\mathcal{H}$. Let $E_i(\C)$ be the union of all connected components $\mathcal{B}$ of the $(i,i)$-restriction of $\mathcal{H}$ such that $\mathcal{B} \neq \C$ and $\{w,u\} \in E_{i}$ for some $w \in \C$ and some $u \in \mathcal{B}$. For each connected component $\mathcal{B}^{\prime}$ of the $(i,i)$-restriction of $\mathcal{H}$, let $\phi_{\mathcal{B}^{\prime}}$ be the (unique) isomorphism from $\mathcal{B}^{\prime}$ to some (unique) $\mathcal{B} \subset E_i(\C)$. Define the involution $\theta_i^{\C}$ by
  \begin{equation}
    \theta_i^{\C}(u) = \left\{ \begin{array}{rl}
        \phi_{\mathcal{B}^{\prime}}(E_{i}(u))
        & \mbox{if} \ u \in E_i(\C) \ \mbox{and} \ E_i(u) \in \mathcal{B}^{\prime}, \\  
        E_i(\phi_{\mathcal{B}^{\prime}}(u))
        & \mbox{if} \ E_i(u) \in E_i(\C) \ \mbox{and} \ u \in \mathcal{B}^{\prime}, \\  
        E_{i}(u) & \mbox{otherwise.}
      \end{array} \right.
    \label{eqn:theta}
  \end{equation}
  Define $E_i'$ to be the set of pairs $\{v,\theta_i^{\C}(v)\}$ for all vertices $v$ admitting an $i$-neighbor. Define a signed, colored graph $\theta_i^{\C}(\G)$ by
  \begin{equation}
    \theta_i^{\C}(\G) = (V, \sigma, E_2 \cup\cdots\cup E_{\imo} \cup E'_i \cup E_{\ipo} \cup\cdots\cup E_{\nmo}). 
  \end{equation}
  \label{defn:theta}
\end{definition}

In order to ensure that axiom $3$ is maintained, one must be careful in the choice of $\C$.

\begin{definition}
  Let $\mathcal{H}$ be a connected component of the $(\ipo,\ipo)$-restriction of $\G$, and let $\lambda$ be the unique partition of $\ipo$ such that there is a surjective morphism from $\mathcal{H}$ to $\G_{\lambda}$. A connected component $\C$ of the $(i,i)$-restriction of $\mathcal{H}$ is \emph{negatively dominant} if one of the following holds:
  \begin{itemize}
  \item $\sigma_{\ipo}(\C) \equiv -1$ and for every connected component $\mathcal{B}$ of the $(i,i)$-restriction of $\mathcal{H}$ such that $\sigma_{\ipo}(\mathcal{B}) \equiv -1$, if $\C \cong \G_{\mu}$ and $\mathcal{B} \cong \G_{\nu}$ for $\mu,\nu \subset \lambda$, then $\mu \geq \nu$ in dominance order;
  \item $\sigma_{\ipo}(\mathcal{B}) \equiv +1$ for every connected component $\mathcal{B}$ of the $(i,i)$-restriction of $\mathcal{H}$, and if $\C \cong \G_{\mu}$ and $\mathcal{B} \cong \G_{\nu}$ for $\mu,\nu \subset \lambda$, then $\mu \geq \nu$ in dominance order.
  \end{itemize}
  \label{defn:hindge}
\end{definition}

We now show that for a suitably chosen $\C$, the map $\theta_i^{\C}$ is well-defined and brings the graph closer to satisfying dual equivalence axiom $6$.

\begin{proposition}
  Let $\G$ be a locally Schur positive graph of type $(n,N)$ satisfying dual equivalence axioms $1,2,3$ and $5$ such that the $(i,N)$-restriction is a dual equivalence graph and the $(\ipo,N)$-restriction satisfies dual equivalence axiom $4$. For $\C$ a negatively dominant $(i,i)$-restricted component of $\G$, the graph $\theta_i^{\C}(\G)$ also satisfies dual equivalence axioms $1,2,3$ and $5$ and the $(\ipo,N)$-restriction of $\theta_i^{\C}(\G)$ also satisfies dual equivalence axiom $4$. Moreover, if $\mathcal{H}$ is the connected component of the $(\ipo,N)$-restriction of $\G$ containing $\C$, then $\theta_i^{\C}(\mathcal{H})$ has two connected components.
  \label{prop:theta-reasonable}
\end{proposition}

\begin{proof}
  The assertion that $\mathcal{H}$ has two connected components is obvious from the definition of $\theta_i^{\C}$. Axioms $1,2$ and $5$ follow from the definition of $\theta_i^{\C}$, and axiom $4$ follows from the fact that edges are swapped only between isomorphic components, so the local structure of the $E_{\imt} \cup E_{\imo} \cup E_i$ remains unchanged by $\theta_i^{\C}$. Therefore we need only address axiom $3$.

  Let $\mathcal{A}$ and $\mathcal{B}$ be two connected components of the $(i,i)$-restriction of $\mathcal{H}$, and suppose $\mathcal{A}  \cong \G_{\alpha}$ and $\mathcal{B} \cong \G_{\beta}$ with $\alpha > \beta$ in dominance order. Let $a \in \mathcal{A}$ and $b \in \mathcal{B}$ and suppose $\{a,b\} \in E_i$. Similar to the proof of \cite{Ass15}(Lemma~3.11), we have $\sigma(w)_{\imo,i} =+-$ and $\sigma(v)_{\imo,i} = -+$. Therefore axiom $3$ fails for this edge if and only if $\sigma(w)_{\imo,i,\ipo} =+--$ and $\sigma(v)_{\imo,i,\ipo} = -++$ if and only if $\sigma(\mathcal{A})_{\ipo} = -1$ and $\sigma(\mathcal{B})_{\ipo} = +1$. With this characterization in mind, suppose now that $\mathcal{A},\mathcal{B}$ and $\mathcal{B}'$ are restricted components of $\mathcal{H}$, with $\mathcal{A},\mathcal{B} \in E_i(\C)$, $\mathcal{B}' \cong \mathcal{B}$, and $a \in \mathcal{A}, b \in \mathcal{B}$ and $b' \in \mathcal{B}'$ such that $\{a,b'\} \in E_i$ and $b = \theta_i^{\C}(a)$. As before, let $\mathcal{A} \cong \G_{\alpha}$ and $\mathcal{B}' \cong \mathcal{B} \cong \G_{\beta}$. Suppose $\sigma(\mathcal{A})_{\ipo} = -1$ and $\sigma(\mathcal{B})_{\ipo} = +1$. Let $\C \cong \G_{\mu}$.  Then the choice of $\C$ as negatively dominant ensures that $\sigma(\C)_{\ipo} = -1$ and that $\mu > \alpha$. Further, since axiom $3$ holds for $\G$, the preceding characterization ensures that $\beta > \mu$. Therefore $\beta > \alpha$ and axiom $3$ holds for $\theta_i^{\C}(\G)$ as well.
\end{proof}

We must consider the local Schur positivity, and show that it is maintained when applying $\theta_i^{\C}$. This is not difficult to show when the $(\iph,N)$-restriction of $\G$ satisfies dual equivalence axiom $4$.

\begin{theorem}
  Let $\G$ be a locally Schur positive graph of type $(n,N)$ such that the $(i,N)$-restriction is a dual equivalence graph and the $(\iph,N)$-restriction satisfies dual equivalence axiom $4$. For $\C$ a negatively dominant restricted component, $\theta_i^{\C}(\G)$ is locally Schur positive.
  \label{thm:theta-LSP}
\end{theorem}

\begin{proof}
  Let $\{w,x\}, \{u,v\} \in E_i(\G)$ with $\theta_i^{\C}(w) = u$ and $\theta_i^{\C}(x) = v$. By the definition of $\theta_i^{\C}$, there is an isomorphism from the connected component of the $(i,N)$-restriction of $\G$ containing $w$ to the connected component of the $(i,N)$-restriction of $\G$ containing $v$ that sends $w$ to $v$, and similarly for the pair $x$ and $u$. Furthermore, by axiom $2$, $\sigma(u)_{\ipt} = \sigma(w)_{\ipt} = \sigma(x)_{\ipt} = \sigma(v)_{\ipt}$.

  By Proposition~\ref{prop:theta-reasonable}, if local Schur positivity fails, then it must do so for some connected component of $E_i \cup E_{\ipo}$, $E_{\imo} \cup E_i \cup E_{\ipo}$, or $E_i \cup E_{\ipo} \cup E_{\ipt}$. If either $\{w,x\} \in E_{\ipo}(\G)$ or $\{u,v\} \in E_{\ipo}(\G)$, then applying $\theta_i^{\C}$ results in the two components being joined for all three cases since there $\ipo$-edges do not change. In this case, local Schur positivity is preserved since components are combined. Therefore we may assume this is not the case. For $E_i \cup E_{\ipo}$, since these components appear in Figure~\ref{fig:lambda4}, each chain must have one $E_i$ edge and one $E_{\ipo}$ edge. Therefore if some component is not locally Schur positive in $\theta_i^{\C}(\G)$, then it must be that one chain has two $E_{\ipo}$ edges while the other has none, violating axiom $3$ contradicting Proposition~\ref{prop:theta-reasonable}. Thus $\theta_i^{\C}(\G)$ maintains $\LSP_4$. By symmetry, we may assume $u$ and $x$ admit $\ipo$-neighbors, and $w$ and $v$ do not as shown in Figure~\ref{fig:theta-swap}. Note that by axioms $3$ and $4$, we have $\sigma(w)_{\ipo} = \sigma(x)_{\ipo} = \sigma(u)_{\ipo} = \sigma(v)_{\ipo}$.

  \begin{figure}[ht]
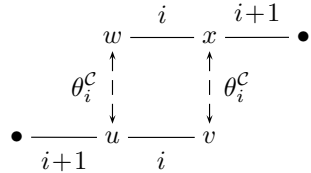

    \begin{displaymath}
      \begin{array}{\cs{7}\cs{7}\cs{7}c}
        & \rnode{a}{w} & \rnode{b}{x} & \rnode{B}{\B} \\[6ex]
        \rnode{C}{\B} & \rnode{c}{u} & \rnode{d}{v} &
      \end{array}
      \psset{nodesep=3pt,linewidth=.1ex}
      \ncline {a}{b} \naput{i}
      \ncline {b}{B} \naput{\ipo}
      \ncline {c}{d} \nbput{i}
      \ncline {C}{c} \nbput{\ipo}
      \ncline[linestyle=dashed] {<->}{a}{c} \nbput{\theta_i^{\C}}
      \ncline[linestyle=dashed] {<->}{b}{d} \naput{\theta_i^{\C}}
    \end{displaymath}
    \caption{\label{fig:theta-swap} Connected components of $E_i \cup E_{\ipo}$ that $\theta_i^{\C}$ might alter.}
  \end{figure}

  Consider now components of $E_{\imo} \cup E_i \cup E_{\ipo}$. Neither of the $i$-edges between $w$ and $x$ nor between $u$ and $v$ can be double edges with $\imo$ since the vertices lie on difference components of the $\ipo$-restriction. Therefore the restricted degree $5$ generating functions cannot be $s_{(3,2)}$ or $s_{(2,2,1)}$. Therefore the restricted degree $5$ generating function for the top component is $s_{(3,1,1)}$ if and only if $x$ has an $\imo$-neighbor, and restricted degree $5$ generating function for the bottom component is $s_{(3,1,1)}$ if and only if $u$ has an $\imo$-neighbor. By axiom $3$, exactly one of $w$ and $x$ has an $\imo$-neighbor and exactly one of $u$ and $v$ has an $\imo$-neighbor, and by Proposition~\ref{prop:theta-reasonable}, $\theta_i^{\C}$ preserves axiom $3$, and so exactly one of $w$ and $u$ has an $\imo$-neighbor. Combining these, the restricted degree $5$ generating function for the top component is $s_{(3,1,1)}$ if and only if the restricted degree $5$ generating function for the bottom component is $s_{(3,1,1)}$. When this is the case, both resulting components after applying $\theta_i^{\C}$ will have generating function $s_{(3,1,1)}$. When this is not the case, by Proposition~\ref{prop:theta-reasonable}, $\theta_i^{\C}$ preserves axiom $2$, and so both components must have generating function $s_{(4,1)}$ or $s_{(2,1,1,1)}$, and so both resulting components after applying $\theta_i^{\C}$ will again have the same generating function. Thus connected components of $E_{\imo} \cup E_i \cup E_{\ipo}$ remain locally Schur positive.
 
  Consider now components of $E_i \cup E_{\ipo} \cup E_{\ipt}$. All or none of $u,v,w,x$ admit an $\ipt$-neighbor. None do if and only if both of the two components have restricted degree $5$ generating function $s_{(4,1)}$ or $s_{(2,1,1,1)}$, in which case both resulting components after applying $\theta_i^{\C}$ will still have that generating function. Therefore assume all four vertices have an $\ipt$-neighbor. If one of the components has restricted degree $5$ generating function $s_{(3,1,1)}$, then applying $\theta_i^{\C}$ results in the two components being joined, and so local Schur positivity is preserved since components are combined and generating functions are added. In the alternative case, both of the two components have restricted degree $5$ generating function $s_{(3,2)}$ or $s_{(2,2,1)}$, in which case both resulting components after applying $\theta_i^{\C}$ will still have the same generating function. Thus connected components of $E_i \cup E_{\ipo} \cup E_{\ipt}$ remain locally Schur positive, and so $\theta_i^{\C}(\G)$ has $\LSP_5$.
\end{proof}

The conclusion that is missing from Theorem~\ref{thm:theta-LSP} is that the $(\iph,N)$-restriction of the resulting graph still satisfies axiom $4$. As suggested by the proof for the two instances when components are combined, this is not, in general, the case. This fact is the main impediment to establishing Conjecture~\ref{conj:cover}. The transformations defined in the following two sections directly address this issue by resolving axiom $4$.

%
\section{Two color components}
%
\label{sec:LSP4}

In order to determine locally when axiom $4$ fails, we introduce the notion of the $i$-type of a vertex. 

\begin{definition}
  Let $\G$ be a signed, colored graph of type $(n,N)$ satisfying axioms $1, 2, 3$ and $5$. For $i \leq n$ with $i<N$, we say that a vertex $w$ has \emph{$i$-type W} if it has an $\imo$-neighbor and $\sigma(w)_{i} = -\sigma(E_{\imo}(w))_{i}$.
  \label{defn:type-W}
\end{definition}

For $i=2$, no vertex can have $i$-type W since there are no $\imo$-edges. If $\{w,x\} \in E_{\imo}$ and neither has an $i$-edge, then by axiom $1$ they must have signatures $+--$ and $-++$, in which case they violate axiom $3$. Therefore when we refer to the $i$-type of a vertex as W or not, we implicitly assume that the vertex has an $i$-edge and that $i \geq 3$.

In a dual equivalence graph, the vertices of $i$-type W are precisely those that are part of a double edge for $E_{\imo}$ and $E_i$. Vertices of $i$-type W are crucial to understanding how $\LSF_4$ can fail in a graph with $\LSP_4$. 

\begin{figure}[ht]
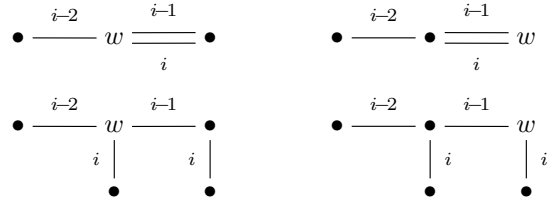

  \begin{displaymath}
    \begin{array}{\cs{7}\cs{7}\cs{10}\cs{7}\cs{7}\cs{7}}
      \rnode{d1}{\bullet} & \rnode{d2}{w} & \rnode{d3}{\bullet} &
      \rnode{d4}{\bullet} & \rnode{d5}{\bullet} & \rnode{d6}{w} \\[5ex]
      \rnode{D1}{\bullet} & \rnode{D2}{w} & \rnode{D3}{\bullet} &
      \rnode{D4}{\bullet} & \rnode{D5}{\bullet} & \rnode{D6}{w} \\[3ex]
      & \rnode{DD2}{\bullet} & \rnode{DD3}{\bullet} &
      & \rnode{DD5}{\bullet} & \rnode{DD6}{\bullet} 
    \end{array}
    \psset{nodesep=3pt,linewidth=.1ex}
    \everypsbox{\scriptstyle}
    \ncline {d1}{d2} \naput{\imt}
    \ncline[offset=2pt] {d2}{d3} \naput{\imo}
    \ncline[offset=2pt] {d3}{d2} \naput{i}
    \ncline {d4}{d5} \naput{\imt}
    \ncline[offset=2pt] {d5}{d6} \naput{\imo}
    \ncline[offset=2pt] {d6}{d5} \naput{i}
    \ncline {D1}{D2} \naput{\imt}
    \ncline {D2}{D3} \naput{\imo}
    \ncline {D2}{DD2} \nbput{i}
    \ncline {D3}{DD3} \nbput{i}
    \ncline {D4}{D5} \naput{\imt}
    \ncline {D5}{D6} \naput{\imo}
    \ncline {D5}{DD5} \naput{i}
    \ncline {D6}{DD6} \naput{i}
  \end{displaymath}
  \caption{\label{fig:type-W} An illustration of vertices $w$ of $i$-type W and neighboring $E_{\imt}$ and $E_{\imo}$ edges.}
\end{figure}

Figure~\ref{fig:type-W} shows the $E_{\imt}$, $E_{\imo}$ and $E_{i}$ edges neighboring a vertex with $i$-type W. The top row contains the possibilities in a dual equivalence graph, while the lower row gives the additional possibilities in the more general setting when axiom $4$ does not hold.

\begin{remark}
  By axioms $1,2$ and $5$, edges $E_j$ with $j\leq\imf$ or $j \geq i\!+\!2$ do not change whether or not a vertex is $i$-type W, i.e. $w$ is $i$-type W if and only if $E_j(w)$ is $i$-type W. In contrast, $E_{\imh}$ often changes whether or not a vertex is $i$-type W, as does $E_{\ipo}$, so these cases require some care. 
\end{remark}

Note that $w$ has $i$-type W if and only if $E_{\imo}(w)$ has $i$-type W, so in some sense $i$-type W is a property of $\imo$-edges rather than of vertices. It is helpful to have a dual property for $i$-edges.

\begin{definition}
  Let $\G$ be a signed, colored graph of type $(n,N)$ satisfying axioms $1, 2, 3$ and $5$. For $i < n$, a vertex $w$ has a \emph{flat $i$-edge} if $w$ has an $i$-edge and $\sigma(w)_{\imt} = \sigma(E_i(w))_{\imt}$.
  \label{defn:flat}
\end{definition}

In a signed, colored graph satisfying axioms $1, 2, 3$ and $5$, flat $i$-edges relate to $i$-type W in the following way: a vertex $w$ has a flat $i$-edge if and only if at most one of $w$ and $E_i(w)$ has $i$-type W. By axiom $1$, a vertex $w$ has a flat $i$-edge if and only if exactly one of $w$ and $E_i(w)$ has an $\imo$-neighbor. 

\begin{definition}
  Let $\G$ be a signed, colored graph of type $(n,N)$ satisfying axioms $1, 2, 3$ and $5$. For $i < n$, a \emph{non-flat $i$-chain} is a sequence $(w_1,w_2,\ldots,w_{2h-1},w_{2h})$ of distinct vertices such that
  \begin{displaymath}
    w_{2j-1} = E_i ( w_{2j} ) 
    \hspace{2em} \mbox{and} \hspace{2em}
    w_{2j+1} = E_{\imo} ( w_{2j} ) .
  \end{displaymath}
  \label{defn:nonflat-chain}
\end{definition}

\begin{figure}[ht]
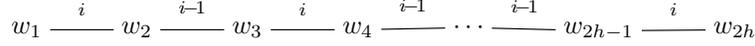

  \begin{displaymath}
    \begin{array}{\cs{7}\cs{7}\cs{7}\cs{7}\cs{7}\cs{7}c}
      \rnode{d1}{w_{1}} & \rnode{d2}{w_{2}} & \rnode{d3}{w_{3}} & \rnode{d4}{w_{4}} & \rnode{c}{\cdots} & \rnode{d5}{w_{2h-1}} & \rnode{d6}{w_{2h}} 
    \end{array}
    \psset{nodesep=3pt,linewidth=.1ex}
    \everypsbox{\scriptstyle}
    \ncline {d1}{d2} \naput{i}
    \ncline {d2}{d3} \naput{\imo}
    \ncline {d3}{d4} \naput{i}
    \ncline {d4}{c} \naput{\imo}
    \ncline {c}{d5} \naput{\imo}
    \ncline {d5}{d6} \naput{i}
  \end{displaymath}
  \caption{\label{fig:nonflat-chain} An illustration of a non-flat $i$-chain.}
\end{figure}

Implicitly, every vertex on a non-flat $i$-chain has an $i$-neighbor, and every vertex except, perhaps, the first and last has an $\imo$-neighbor. In particular, each $i$-edge of a non-flat $i$-chain is non-flat, except, perhaps, the first or last, and $w_j$ has $i$-type W for $1<j<2h$.

In a dual equivalence graph, every non-flat $i$-chain has length $2$. Define $W_i(\G)$ to be the set of vertices that lie on a non-flat $i$-chain of length greater than $2$, i.e.
\begin{equation}
  W_i(\G) \ = \ \left\{ w \in V \ | \ w = w_{j} \mbox{ on a non-flat $i$-chain of length $2h$ with } 1<j<2h \right\}.
\label{eqn:W}
\end{equation}
Equivalently, $W_i$ is the set of vertices $w$ for which $w$ has $i$-type W and $E_{\imo}(w) \neq E_i(w)$.

\begin{proposition}
  Let $\G$ be a signed, colored graph of type $(n,n)$ satisfying dual equivalence axioms $1,2,3$ and $5$ that has $\LSP_4$. Then $\G$ has $\LSF_4$ if and only if $W_i(\G)$ is empty for all $1 < i < n$.
  \label{prop:empty-W}
\end{proposition}

\begin{proof}
  When $\LSF_4$ holds for $\G$, non-flat $i$-chains have length at most $2$, so $W_i(\G)$ is empty. Conversely, suppose $W_i(\G)$ is empty for all $i$. By axioms $1$ and $2$, a vertex $u$ has $i$-type W if and only if both $u$ and $E_{\imo}(u)$ have an $i$-neighbor. In this case, $u$ lies on a non-flat $i$-chain of length at most $2$ if and only if $E_{\imo}(u) = E_i(u)$. Alternatively, if $u$ admits an $i$-edge but does not have $i$-type W, then by the previous analysis neither does $E_i(u)$, and so the $i$-edge is flat. In particular, the connected component of $E_{\imo} \cup E_{i}$ containing $u$ is a path with three vertices. Therefore all connected components of $E_{\imo} \cup E_i$ appear in Figure~\ref{fig:lambda4}. 
\end{proof}

By Proposition~\ref{prop:empty-W}, $\LSF_4$ holds if and only if $W_i$ is empty. We construct a map $\varphi_i^w$ with the goal of reducing the cardinality of $W_i$. This map, illustrated in Figure~\ref{fig:phi}, takes as input an element $w \in W_i(\G)$ and redefines $i$-edges so that $E_{\imo}(w) = E_i(w)$, thereby removing $w$ (and $E_{\imo}(w)$) from $W_i$.

\begin{figure}[ht]
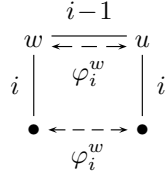

  \begin{displaymath}
    \begin{array}{\cs{8}c}
      \rnode{C}{w} & \rnode{D}{u}  \\[5ex] 
      \rnode{B}{\B} & \rnode{E}{\B}
    \end{array}
    \psset{nodesep=3pt,linewidth=.1ex}
    \ncline[offset=2pt] {C}{D} \naput{\imo}
    \ncline {B}{C} \naput{i}
    \ncline {D}{E} \naput{i}
    \ncline[offset=-2pt,linestyle=dashed] {<->}{C}{D} \nbput{\varphi_i^w}
    \ncline[linestyle=dashed] {<->}{B}{E} \nbput{\varphi_i^w}
  \end{displaymath}
  \caption{\label{fig:phi} An illustration of the involution $\varphi_i^w$, with $w \in W_i(\G)$ and $u = E_{\imo}(w)$.}
\end{figure}

How $\varphi_i^w$ acts on the connected component of $E_{\imo} \cup E_i$ is straightforward given that there is a unique choice that preserves axiom $1$. By axiom $5$, if $\{w,x\} \in E_i$ and $\{x,y\} \in E_j$ for $|i-j| \geq 3$, then $\{w,v\} \in E_j$ and $\{v,y\} \in E_i$ for some $v \in V$. Changing a single $i$-edge may result in a violation of this condition. Therefore when one $i$-edge is changed, all other $i$-edges that subsequently violate axiom $5$ must also be changed, as illustrated in Figure~\ref{fig:axiom5}.

\begin{figure}[ht]
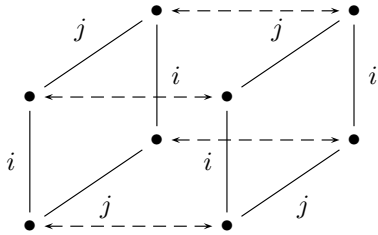

  \begin{displaymath}
    \begin{array}{\cs{5}\cs{5}\cs{5}\cs{5}\cs{5}c}
      & & \rnode{a3}{\bullet} & & & \rnode{a6}{\bullet} \\[1ex]
      & & & & & \\[1ex]
      \rnode{c1}{\bullet} & & & \rnode{c4}{\bullet} & & \\[1ex]
      & & \rnode{d3}{\bullet} & & & \rnode{d6}{\bullet} \\[1ex]
      & & & & & \\[1ex]
      \rnode{f1}{\bullet} & & & \rnode{f4}{\bullet} & & 
    \end{array}
    \psset{nodesep=3pt,linewidth=.1ex}
    \ncline {c1}{f1} \nbput{i}
    \ncline {c4}{f4} \nbput{i}
    \ncline {a3}{d3} \naput{i}
    \ncline {a6}{d6} \naput{i}
    \ncline {c1}{a3} \naput{j}
    \ncline {c4}{a6} \naput{j}
    \ncline {f1}{d3} \nbput{j}
    \ncline {f4}{d6} \nbput{j}
    \ncline[linestyle=dashed]{<->} {a3}{a6}
    \ncline[linestyle=dashed]{<->} {c1}{c4}
    \ncline[linestyle=dashed]{<->} {d3}{d6}
    \ncline[linestyle=dashed]{<->} {f1}{f4}
  \end{displaymath}
  \caption{\label{fig:axiom5} An illustration of how to maintain axiom $5$ when swapping $i$-edges.}
\end{figure}

\begin{definition}
  Let $(V,\sigma,E)$ be a signed, colored graph of type $(n,N)$ satisfying axioms $1,2$ and $5$. For a vertex $w$, the \emph{$i$-package of $w$} is the connected component of $E_2 \cup\cdots\cup E_{\imh} \cup E_{\iph} \cup\cdots\cup E_{\nmo}$ containing $w$.
  \label{defn:package}
\end{definition}

By axiom $2$, both $\sigma_{\imo}$ and $\sigma_{i}$ are constant on $i$-packages. Therefore $w$ admits an $i$-neighbor if and only if every vertex of the $i$-package of $w$ admits an $i$-neighbor. By axiom $5$, knowing $E_i(w)$ determines $E_i$ on the entire $i$-package of $w$. That is to say, $E_i$ may be regarded as an isomorphism between the $i$-packages of $w$ and $E_i(w)$ that preserves $\sigma_1, \ldots, \sigma_{\imh}, \sigma_{\ipt}, \ldots, \sigma_{N-1}$. If the four vertices in Figure~\ref{fig:swap} have isomorphic $i$-packages, we can swap all $i$-edges on the corresponding $i$-packages while maintaining axioms $2$ and $5$.

First, Lemma~3.11 from \cite{Ass15} shows that if ignoring the highest color edges of a graph results in a dual equivalence graph of degree $n-1$, then there is a unique dual equivalence graph of degree $n$ into which the restriction embeds. To be precise, for partitions $\lambda \subset \rho$, with $|\lambda| = n$ and $|\rho| = N$, choose a tableau $A$ of shape $\rho/\lambda$ with entries $n+1,\ldots,N$. Consider the set of standard Young tableaux $T \in \mathrm{SYT}(\rho)$ such that $T$ restricted to $\rho/\lambda$ is $A$. Let $\G_{\lambda,A}$ be the signed, colored graph of type $(n,N)$ on this set with $i$-edges given by elementary dual equivalences for $\triple$ with $i<n$. Then $\G_{\lambda,A}$ is a dual equivalence graph of type $(n,N)$, and the $(n,n)$-restriction of $\G_{\lambda,A}$ is $\G_{\lambda}$.

\begin{lemma}[\cite{Ass15}]
  Let $\G = (V,\sigma,E)$ be a connected dual equivalence graph of type $(n,N)$ with $n<N$, and let $\phi$ be an isomorphism from the $(n,n)$-restriction of $\G$ to $\G_{\lambda}$ for some partition $\lambda$ of $n$.  Then there exists a semi-standard tableau $A$ of shape $\rho/\lambda$, $|\rho| = N$, with entries $n+1,\ldots,N$ such that $\phi$ gives an isomorphism from $\G$ to $\G_{\lambda,A}$. Moreover, the position of the cell of $A$ containing $n+1$ is unique.
\label{lem:extend-signs}
\end{lemma}

\begin{lemma}
  Let $\G$ be a signed, colored graph of type $(n,N)$ satisfying dual equivalence axioms $1,2,3$ and $5$, and suppose that the $(\imt,N)$-restriction of $\G$ is a dual equivalence graph. Let $w$ be a vertex of $i$-type W such that every vertex on the $\imo$-package of $w$ has a flat $\imo$-edge. Then there exists an isomorphism between the $i$-packages of $w$ and $E_{\imo}(w)$.
\label{lem:phi-compatible}
\end{lemma}

\begin{proof}
  If $E_{\imo}(w) = E_i(w)$, then the result follows immediately from axioms $2$ and $5$. Suppose then that $w \in W_i(\G)$, and set $u = E_{\imo}(w)$. Recall that $E_{\imo}$ may be regarded as an involution on vertices that admit an $\imo$-neighbor. Regarded as such, by axioms $1,2$ and $5$, $E_{\imo}$ gives an involution between $\imo$-packages of $w$ and $u$. Therefore we need only show that this isomorphism restricted to $E_{2} \cup \cdots \cup E_{\imf}$ extends to an isomorphism for $E_{2} \cup \cdots \cup E_{\imh}$, since the isomorphism for $E_{\iph} \cup \cdots \cup E_{\nmo}$ is already established.

  By the assumption that all vertices $v$ on the $\imo$-package of $w$ have flat $\imo$-edges, we know $\sigma(v)_{\imh} = \sigma(E_{\imo}(v))_{\imh}$. Therefore $E_{\imo}$ gives an involution between the $(\imh,\imt)$-restrictions of the $i$-packages of $w$ and $u$. We extend this isomorphism as illustrated in Figure~\ref{fig:compatible}.

  \begin{figure}[ht]
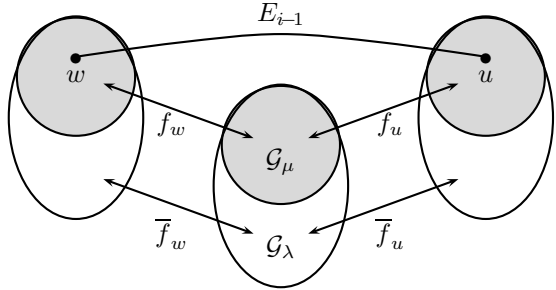

    \begin{center}
      \psset{xunit=3ex}
      \psset{yunit=3ex}
      \pspicture(0,-3)(14,5)
      \pscircle[fillstyle=solid,fillcolor=lightgray](1,3){.8}
      \psellipse(1,1.8)(2,3)
      \pscircle[fillstyle=solid,fillcolor=lightgray](7,1){.8}
      \psellipse(7,-.2)(2,3)
      \pscircle[fillstyle=solid,fillcolor=lightgray](13,3){.8}
      \psellipse(13,1.8)(2,3)
      \rput(1,3.5){$\bullet$}
      \rput(1,3){$w$}
      \rput(13,3.5){$\bullet$}
      \rput(13,3){$u$}
      \rput(7,.5){$\G_{\mu}$}
      \rput(7,-2){$\G_{\lambda}$}
      \pscurve(1.1,3.6)(4,4)(7,4.25)(10,4)(12.9,3.6)
      \rput(7,4.8){$E_{\imo}$}
      \psline{<->}(1.8,2.8)(6.2,1.2)
      \rput(3.8,1.6){$f_{w}$}
      \psline{<->}(7.8,1.2)(12.2,2.8)
      \rput(10.2,1.6){$f_{u}$}
      \psline{<->}(1.8,0)(6.2,-1.6)
      \rput(3.8,-1.6){$\overline{f}_{w}$}
      \psline{<->}(7.8,-1.6)(12.2,0)
      \rput(10.2,-1.6){$\overline{f}_{u}$}
      \endpspicture
    \end{center}
      \caption{\label{fig:compatible}Extending the isomorphism of $\imo$-packages to an isomorphism of $i$-packages}
  \end{figure}

  By Lemma~\ref{lem:extend-signs} and the hypothesis that the $(\imt,N)$-restriction of $\G$ is a dual equivalence graph, there exist isomorphisms, say $f_w$ and $f_u$, from the $(\imh,\imt)$-restrictions of the $i$-packages of $w$ and $u$ to the augmented dual equivalence graph $\G_{\mu,A}$ for a unique partition $\mu$ of $\imh$ and a unique single cell augmenting tableau $A$.  By Theorem~\ref{thm:cover}, the two isomorphism extend consistently across $E_{\imh}$ edges to give isomorphisms $\overline{f}_w$ and $\overline{f}_u$ from the $(\imt,\imt)$-restrictions of the connected components containing $w$ and $u$, respectively, to $\G_{\lambda}$ where $\lambda$ is the shape of $\mu$ augmented by $A$. In particular, the composition of these isomorphisms gives an isomorphism between the $(\imt,\imt)$-restrictions of the $i$-packages of $w$ and $u$.
\end{proof}

The hypotheses of Lemma~\ref{lem:phi-compatible} cannot be relaxed, so these vertices are of particular importance. Therefore we define the set $W_i^0(\G) \subseteq W_i(\G)$ by
\begin{equation}
  W_i^0(\G) = \{ w \in W_i(\G) \ | \ \mbox{every vertex on the $\imo$-package of $w$ has a flat $\imo$-edge} \}.
  \label{eqn:W-0}
\end{equation}
In a dual equivalence graph, these sets coincide, as follows from the following.

\begin{proposition}
  If every connected component of $E_{\imt} \cup E_{\imo}$ appears in Figure~\ref{fig:lambda4}, then $W_i^0(\G) = W_i(\G)$.
  \label{prop:phi-lambda4}
\end{proposition}

\begin{proof}
  If every connected component of $E_{\imt} \cup E_{\imo}$ appears in Figure~\ref{fig:lambda4}, then the $\imo$-edge at $v$ is not flat if and only if $E_{\imo}(v) = E_{\imt}(v)$. In this case, by axiom $2$, $\sigma(v)_{i} = \sigma(E_{\imt}(v))_{i} = \sigma(E_{\imo}(v))_{i}$, so $v$ does not have $i$-type W. Since $\sigma_i$ is constant on $i$-packages, no vertex on the $i$-package of $v$ has $i$-type W. Therefore, $W_i^0(\G) = W_i(\G)$.
\end{proof}

We use the isomorphism of Lemma~\ref{lem:phi-compatible} to define an involution $\varphi^{w}_i$ as follows.

\begin{definition}
  For $w \in W_i^0(\G)$, let $u = E_{\imo}(w)$, and let $\phi$ the isomorphism of Lemma~\ref{lem:phi-compatible}. Define the involution $\varphi^{w}_i$ on all vertices admitting an $i$-neighbor by
  \begin{equation}
    \varphi^w_i(v) = \left\{ \begin{array}{rl}
        \phi(v) & \mbox{if $v$ lies on the $i$-package of $w$ or $u$,} \\ [1ex]
        E_{i} \phi E_{i}(v) & \mbox{if $E_{i}(v)$ lies on the $i$-package of $w$ or $u$,} \\ [1ex]
        E_i(v) & \mbox{otherwise.}
      \end{array} \right.
    \label{eqn:phi}
  \end{equation}
  Define $E_i'$ to be the set of pairs $\{v,\varphi_i^w(v)\}$ for each $v$ admitting an $i$-neighbor. Define a signed, colored graph $\varphi^w_i(\G)$ of type $(n,N)$ by
  \begin{equation}
    \varphi^w_i(\G) = (V, \sigma, E_2 \cup\cdots\cup E_{\imo} \cup E_i' \cup E_{\ipo} \cup\cdots\cup E_{\nmo}).
    \label{eqn:G'}
  \end{equation}
  \label{defn:phi}
\end{definition}

Since the isomorphisms from Lemma~\ref{lem:phi-compatible} for $w$ and $u = E_{\imo}(w)$ are inverse to one another, we abuse notation in Definition~\ref{defn:phi} by letting $\phi$ denote either. Note as well that $\varphi_i^w = \varphi_i^u$.

The goal is to use the maps $\varphi_i$ to transform $\LSP_4$ into $\LSF_4$. For example, the graph in Figure~\ref{fig:box}, which arises from the graph for the Macdonald polynomial $\widetilde{H}_{(5)}(X;q,t)$, is locally Schur positive but fails $\LSF_4$. The transformation of this graph into a dual equivalence graph requires only $\varphi_3$ and $\varphi_4$. The result is the dual equivalence graph given in Figure~\ref{fig:open-box}. We note, however, that order matters. For the depicted transformation, we first applied $\varphi_3$ twice and then applied $\varphi_4$ twice. Had we applied the maps in a different order, the resulting graphs would be different, though in this case it would still be a dual equivalence graph.

\begin{figure}[ht]
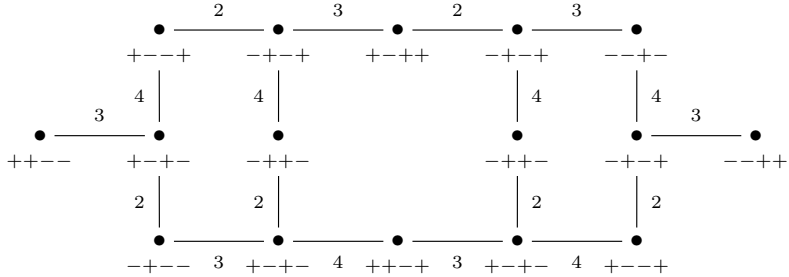

  \begin{displaymath}
    \begin{array}{\cs{7}\cs{7}\cs{7}\cs{7}\cs{7}\cs{7}c}
      & \sbull{t1}{+--+} & \sbull{t2}{-+-+} & \sbull{t3}{+-++} 
      & \sbull{t4}{-+-+} & \sbull{t5}{--+-} & \\[2\cellsize]
      \sbull{m0}{++--} & \sbull{m1}{+-+-} & \sbull{m2}{-++-} & 
      & \sbull{m4}{-++-} & \sbull{m5}{-+-+} & \sbull{m6}{--++} \\[2\cellsize]
      & \sbull{b1}{-+--} & \sbull{b2}{+-+-} & \sbull{b3}{++-+} 
      & \sbull{b4}{+-+-} & \sbull{b5}{+--+} &
    \end{array}
    \psset{linewidth=.1ex,nodesep=3pt}
    \everypsbox{\scriptstyle}
    \ncline {t1}{t2} \naput{2}
    \ncline {t2}{t3} \naput{3}
    \ncline {t3}{t4} \naput{2}
    \ncline {t4}{t5} \naput{3}
    \ncline {t1t1}{m1} \nbput{4}
    \ncline {t2t2}{m2} \nbput{4}
    \ncline {t4t4}{m4} \naput{4}
    \ncline {t5t5}{m5} \naput{4}
    \ncline {m0}{m1} \naput{3}
    \ncline {m5}{m6} \naput{3}
    \ncline {m1m1}{b1} \nbput{2}
    \ncline {m2m2}{b2} \nbput{2}
    \ncline {m4m4}{b4} \naput{2}
    \ncline {m5m5}{b5} \naput{2}
    \ncline {b1}{b2} \nbput{3}
    \ncline {b2}{b3} \nbput{4}
    \ncline {b3}{b4} \nbput{3}
    \ncline {b4}{b5} \nbput{4}
  \end{displaymath}
  \caption{\label{fig:box}A locally Schur positive graph with generating function $s_{3,2} + s_{3,1,1} + s_{2,2,1}$.}
\end{figure}  

\begin{figure}[ht]
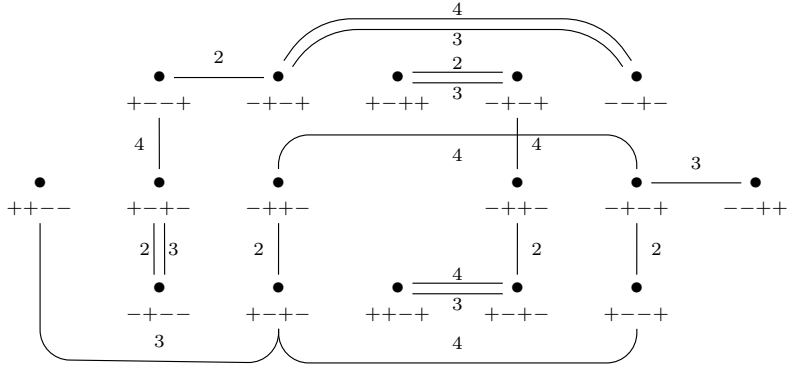

  \begin{displaymath}
    \begin{array}{\cs{7}\cs{7}\cs{7}\cs{7}\cs{7}\cs{7}c} \\[\cellsize]
      & \sbull{t1}{+--+} & \sbull{t2}{-+-+} & \sbull{t3}{+-++} 
      & \sbull{t4}{-+-+} & \sbull{t5}{--+-} & \\[2\cellsize]
      \sbull{m0}{++--} & \sbull{m1}{+-+-} & \sbull{m2}{-++-} & 
      & \sbull{m4}{-++-} & \sbull{m5}{-+-+} & \sbull{m6}{--++} \\[2\cellsize]
      & \sbull{b1}{-+--} & \sbull{b2}{+-+-} & \sbull{b3}{++-+} 
      & \sbull{b4}{+-+-} & \sbull{b5}{+--+} & \\[\cellsize]
    \end{array}
    \psset{linewidth=.1ex,nodesep=3pt}
    \everypsbox{\scriptstyle}
    \ncline {t1}{t2} \naput{2}
    \ncline[offset=2pt] {t3}{t4} \nbput{3}
    \ncline[offset=2pt] {t4}{t3} \nbput{2}
    \ncdiag[offset=2pt,angleA=55,angleB=125,arm=4.5ex,linearc=1] {t2}{t5} \nbput{3}
    \ncdiag[offset=2pt,angleB=55,angleA=125,arm=4ex,linearc=1] {t5}{t2} \nbput{4}
    \ncline {t1t1}{m1} \nbput{4}
    \ncline {t4t4}{m4} \naput{4}
    \ncdiag[angleA=90,angleB=90,arm=3ex,linearc=.4] {m2}{m5} \nbput{4}
    \ncdiag[angleA=-90,armA=12ex,angleB=-90,armB=3ex,linearc=.4] {m0m0}{b2b2} \naput{3}
    \ncline {m5}{m6} \naput{3}
    \ncline[offset=2pt] {m1m1}{b1} \nbput{2}
    \ncline[offset=2pt] {b1}{m1m1} \nbput{3}
    \ncline {m2m2}{b2} \nbput{2}
    \ncline {m4m4}{b4} \naput{2}
    \ncline {m5m5}{b5} \naput{2}
    \ncdiag[angleA=-90,angleB=-90,arm=3ex,linearc=.4] {b5b5}{b2b2} \nbput{4}
    \ncline[offset=2pt] {b3}{b4} \nbput{3}
    \ncline[offset=2pt] {b4}{b3} \nbput{4}
  \end{displaymath}
  \caption{\label{fig:open-box}The transformation of the graph in Figure~\ref{fig:box} using $\varphi_3$ and $\varphi_4$.}
\end{figure}  

The goal with $\varphi_i^w$ is to reduce the cardinality of $W_i(\G)$. The following result shows that this happens provided the $(i,N)$-restriction of $\G$ satisfies dual equivalence graph axiom $4$.

\begin{theorem}
  Let $\G$ be a locally Schur positive graph of type $(n,N)$, and suppose that the $(\imt,N)$-restriction of $\G$ is a dual equivalence graph and that the $(i,N)$-restriction of $\G$ satisfies dual equivalence axiom $4$. Then $W_i^0(\G) = W_i(\G)$, and for $w \in W_i(\G)$, $W_i(\varphi_i^w(\G))$ is a proper subset of $W_i(\G)$.
  \label{thm:phi-terminate}
\end{theorem}

\begin{proof}
  By Proposition~\ref{prop:phi-lambda4}, $W_i^0(\G) = W_i(\G)$. As the $i$-type of a vertex is determined by the connected component of $E_{\imt} \cup E_{\imo}$ containing it, the $i$-type of a vertex is the same in $\G$ and $\varphi_i^w(\G)$. Therefore, to show that $v \not\in W_i(\G)$ implies $v \not\in W_i(\varphi_i^w(\G))$, we must show that for $v$ with $i$-type W such that $E_i(v) = E_{\imo}(v)$, we have $\varphi_i^w(v) = E_{\imo}(v)$ as well.

  It suffices to consider $v$ on the $i$-packages of $w$ and $E_{i}(w)$. We claim that for any $v$ on the $i$-package of $w$, $E_{\imo}(v) \neq E_i(v)$. By axiom $5$, both $E_{\imo}$ and $E_i$ commute with $E_h$ for $h \leq \imf$ and $h \geq \iph$. Therefore, if the claim holds for some vertex $v$, then it holds for any vertex connected to $v$ by edges in $E_2 \cup \cdots \cup E_{\imf} \cup E_{\iph} \cup \cdots \cup E_{\nmo}$. It suffices to show the claim for $v = E_{\imh}(w)$ since, by axiom $6$, any vertex on the $i$-package of $w$ can be reached by crossing at most one $E_{\imh}$ edge. 

  \begin{figure}[ht]
    \begin{displaymath}
      \begin{array}{\cs{7}\cs{7}\cs{7}c}
        \rnode{D1}{\bullet} & \rnode{D2}{w} & \rnode{D3}{\bullet} & \rnode{D4}{\bullet}  \\[5ex]
        \rnode{DD1}{\bullet}& \rnode{DD2}{v}& \rnode{DD3}{\bullet} &  
      \end{array}
      \psset{nodesep=3pt,linewidth=.1ex}
      \everypsbox{\scriptstyle}
      \ncline {D1}{D2} \naput{i}
      \ncline {D2}{D3} \naput{\imo}
      \ncline {D3}{D4} \naput{i}
      \ncline {D2}{DD2} \nbput{\imh}
      \ncline {D3}{DD3} \naput{\imh}
      \ncline[offset=2pt] {DD1}{DD2} \nbput{i}
      \ncline[offset=2pt] {DD2}{DD1} \nbput{\imo}
    \end{displaymath}
    \caption{\label{fig:phi-cross}An illustration when $w \in W_i(\G)$ and $E_{\imh}(w)\not\in W_i(\G)$ has $i$-type W.}
  \end{figure}
  
  Let $v = E_{\imh}(w)$ and suppose that $E_{\imo}(v) = E_i(v)$ as illustrated in Figure~\ref{fig:phi-cross}. Consider the degree $5$ generating function of the component of $E_{\imh} \cup E_{\imt} \cup E_{\imo}$ containing $w$ and $v$. It cannot be $s_{(4,1)}$ or $s_{(2,1,1,1)}$ since $w$ admits both an $\imo$ and an $\imh$ neighbor. Therefore $E_{\imo}(w)$ also has an $\imh$-neighbor. Suppose, for contradiction, that the generating function is $s_{(3,2)}$ or $s_{(2,1,1,1)}$. Then either $w$ or $E_{\imo}(w)$ has a double edge for $\imh$ and $\imt$. By axioms $1$ and $3$, since both $w$ and $v=E_{\imh}(w)$ have an $\imo$-neighbor, we cannot have $v = E_{\imt}(w)$ as well, so it must be the case that $E_{\imo}(w)$ has a double edge for $\imh$ and $\imt$. Consulting Figure~\ref{fig:lambda5}, this forces the $\imo$-edge at $v$ to be a double edge with $\imt$, which is impossible since we are assuming it is a double edge with $i$. Therefore the generating function must be $s_{(3,1,1)}$.

  By Figure~\ref{fig:lambda5}, this means $E_{\imo}(v) = E_{\imh}E_{\imo}(w)$. Using this together with axiom $5$, we have $E_{\imh}E_i(w) = E_iE_{\imh}(w) = E_i(v) = E_{\imo}(v) = E_{\imh}E_{\imo}(w)$. By axiom $1$, this implies $E_{i}(w) = E_{\imo}(w)$, contradicting the assumption that $w \in W_i(\G)$. Thus for any $v$ on the $i$-package of $w$, $E_{\imo}(v) \neq E_i(v)$. By axiom $1$, the same now holds for vertices on the $i$-package of $E_i(w)$. Therefore $W_i(\varphi_i^w(\G))$ is a proper subset of $W_i(\G)$.
\end{proof}

\begin{figure}[ht]
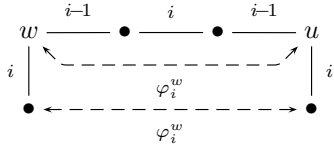

  \begin{displaymath}
    \begin{array}{\cs{7}\cs{7}\cs{7}c}
      \rnode{C}{w} & \rnode{cc}{\B} & \rnode{dd}{\B} & \rnode{D}{u}  \\[4ex] 
      \rnode{B}{\B} & & & \rnode{E}{\B}
    \end{array}
    \psset{nodesep=3pt,linewidth=.1ex}
    \everypsbox{\scriptstyle}
    \ncline {C}{cc}  \naput{\imo}
    \ncline {cc}{dd} \naput{i}
    \ncline {dd}{D}  \naput{\imo}
    \ncline {B}{C} \naput{i}
    \ncline {D}{E} \naput{i}
    \ncdiag[angleA=315,angleB=225,linearc=.15,linestyle=dashed] {<->}{C}{D} \nbput{\varphi_i^w}
    \ncline[linestyle=dashed] {<->}{B}{E} \nbput{\varphi_i^w}
  \end{displaymath}
  \caption{\label{fig:long_phi} The long version of the involution $\varphi_i^w$ for $r=1$.}
\end{figure}

\begin{remark}
  If $w$ is the second vertex of a non-flat $i$-chain of length greater than $4$, then rather than taking $u = E_{\imo}(w)$ in Definition~\ref{defn:phi}, we make take $u = E_{\imo}E_iE_{\imo}(w)$ instead as depicted in Figure~\ref{fig:long_phi}. Since both $w,u \in W_i^0(\G)$, Lemma~\ref{lem:phi-compatible} applies. More generally, we may take $u = E_{\imo} (E_i E_{\imo})^r (w)$ whenever all vertices on the $E_{\imo} \cup E_i$ path between $w$ and $u$ lie in $W_i^0(\G)$. By the proof of Theorem~\ref{thm:phi-terminate}, this generalization decreases $W_i(\G)$ if and only if $w,u$ are the second and penultimate vertices of a maximal non-flat $i$-chain, and if not, then it still does not increase $W_i(\G)$, and the number of possible application of this decreases.
\label{rmk:long_phi}
\end{remark}

For $\G$ a locally Schur positive graph of type $(n,N)$ such that the $(\imt,N)$-restriction of $\G$ is a dual equivalence graph, $\varphi^w_i(\G)$ also satisfies axioms $1, 2$ and $5$. This follows immediately from Lemma~\ref{lem:phi-compatible} and the definition of the maps on $i$-packages. It turns out that $\varphi_i^w$ also preserves axiom $3$, but this requires considerably more work to prove in general. However, when restricting to edges $E_i$ and lower, not only does axiom $3$ hold, but $\LSP_4$ does as well.

\begin{lemma}
  Let $\G$ be a locally Schur positive graph of type $(\ipo,\ipo)$, and suppose that the $(\imt,\ipo)$-restriction of $\G$ is a dual equivalence graph and that the $(i,\ipo)$-restriction of $\G$ satisfies dual equivalence axiom $4$. Then $\varphi_i^w(\G)$ has $\LSP_4$ for any $w \in W_i(\G)$.
  \label{lem:LSP4-phi}
\end{lemma}

\begin{proof}
  By Proposition~\ref{prop:phi-lambda4}, $W_i^0(\G) = W_i(\G)$. Afterward, the component containing $w$ and $E_{\imo}(w)$ has degree $4$ generating function $s_{(2,2)}$, which is Schur positive, and so the positivity for $E_i(w)$ and $E_i E_{\imo}(w)$ follows from Proposition~\ref{prop:LSP4} since the component was Schur positive in $\G$. By axiom $5$, if the component containing $v$ is $\LSP_4$, then so are the components on any vertex of the connected component of $E_2 \cup \cdots \cup E_{\imf}$ containing $v$. By axiom $6$, it suffices to consider the positivity across a single $E_{\imh}$ edge from $E_i(w), w, E_{\imo}(w)$ and $E_i E_{\imo}(w)$. Since all four vertices have isomorphic $i$-packages by Lemma~\ref{lem:phi-compatible}, if one of the four admits an $\imh$-neighbor, then they all do, so we may assume this is the case. By the symmetry between $w$ and $E_{\imo}(w)$ and the fact that the $\imo$-edge between them is flat, we may assume $w$ admits an $\imt$-neighbor and $E_i(w)$ does not, as depicted in Figure~\ref{fig:phi-degree4}.

  \begin{figure}[ht]
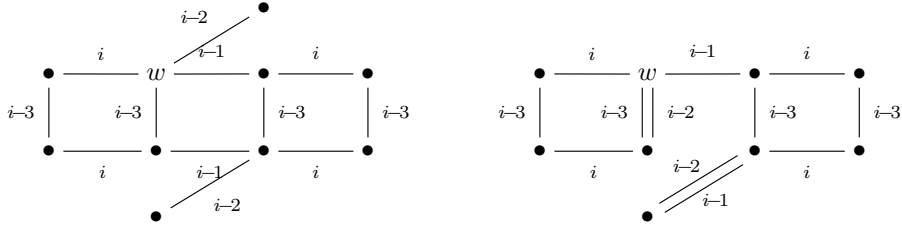

    \begin{displaymath}
      \begin{array}{\cs{8}\cs{8}\cs{8}c}
        & & \rnode{uu}{\B} & \\[3ex]
        \rnode{x}{\B} & \rnode{w}{w} & \rnode{u}{\B} & \rnode{v}{\B} \\[4ex]
        \rnode{X}{\B} & \rnode{W}{\B} & \rnode{U}{\B} & \rnode{V}{\B} \\[3ex]
        & \rnode{WW}{\B} & &       
      \end{array}
      \hspace{5em}
      \begin{array}{\cs{8}\cs{8}\cs{8}c}
        & & & \\[3ex]
        \rnode{x1}{\B} & \rnode{w1}{w} & \rnode{u1}{\B} & \rnode{v1}{\B} \\[4ex]
        \rnode{X1}{\B} & \rnode{W1}{\B} & \rnode{U1}{\B} & \rnode{V1}{\B} \\[3ex]
        & \rnode{WW1}{\B} & &
      \end{array}
      \psset{nodesep=3pt,linewidth=.1ex}
      \everypsbox{\scriptstyle}
      \ncline {w}{uu} \naput{\imt}
      \ncline {x}{w} \naput{i}
      \ncline {w}{u} \naput{\imo}
      \ncline {u}{v} \naput{i}
      \ncline {x}{X} \nbput{\imh}
      \ncline {w}{W} \nbput{\imh}
      \ncline {u}{U} \naput{\imh}
      \ncline {v}{V} \naput{\imh}
      \ncline {X}{W} \nbput{i}
      \ncline {U}{V} \nbput{i}
      \ncline {W}{U} \nbput{\imo}
      \ncline {WW}{U} \nbput{\imt}
      \ncline {x1}{w1} \naput{i}
      \ncline {w1}{u1} \naput{\imo}
      \ncline {u1}{v1} \naput{i}
      \ncline {x1}{X1} \nbput{\imh}
      \ncline[offset=2pt] {w1}{W1} \naput{\imt}
      \ncline[offset=2pt] {W1}{w1} \naput{\imh}
      \ncline {u1}{U1} \naput{\imh}
      \ncline {v1}{V1} \naput{\imh}
      \ncline {X1}{W1} \nbput{i}
      \ncline {U1}{V1} \nbput{i}
      \ncline[offset=2pt] {WW1}{U1} \nbput{\imo}
      \ncline[offset=2pt] {U1}{WW1} \nbput{\imt}
    \end{displaymath}
    \caption{\label{fig:phi-degree4} The two possibilities for $w \in W_i(\G)$ admitting an $\imh$-neighbor.}
  \end{figure}

  Consider the degree $5$ generating function for the connected component of $E_{\imh} \cup E_{\imt} \cup E_{\imo}$ containing $w$. Since $w$ admits both $E_{\imh}$ and $E_{\imo}$, it cannot have generating function $s_{(4,1)}$ or $s_{(2,1,1,1)}$.

  If the degree $5$ generating function for the connected component of $E_{\imh} \cup E_{\imt} \cup E_{\imo}$ containing $w$ is $s_{(3,1,1)}$, then by Figure~\ref{fig:lambda5}, $E_{\imo} E_{\imh}(w) = E_{\imh} E_{\imo}(w)$, and we have the situation depicted on the right side of Figure~\ref{fig:phi-degree4}. In this case, $E_{\imh}(w) \in W_i(\G)$ and $\varphi_i^{E_{\imh}(w)} = \varphi_i^w$. Therefore $\LSP_4$ is maintained.

  Since $w$ has a flat $\imo$-edge, $w$ cannot have $\imo$-type W, so if the degree $5$ generating function is $s_{(3,2)}$ or $s_{(2,2,1)}$, then $E_{\imh}(w) = E_{\imt}(w)$ and, by Figure~\ref{fig:lambda5}, we have the situation depicted on the right side of Figure~\ref{fig:phi-degree4}. In this case, $E_{\imh}(w)$ has no $\imo$-neighbor. On the other side, $E_{\imo} E_{\imh} E_{\imo}(w) = E_{\imt} E_{\imh} E_{\imo}(w)$ which does not admit an $i$-neighbor. Therefore the connected component $E_{\imo} \cup E_i$ containing $E_{\imh}(w)$ will have degree $4$ generating function $s_{(3,1)}$ or $s_{(2,1,1)}$, thus establishing $\LSP_4$ in this case as well. 
\end{proof}

The conclusion missing from Lemma~\ref{lem:LSP4-phi} and the impediment to applying $\varphi_i^w$ repeatedly is that $\LSP_5$ is not guaranteed to hold. Therefore we turn our attention next to degree $5$ components.

%
\section{Three color components}
%
\label{sec:LSP5}

To identify locally components failing $\LSF_5$, we introduce a more general notion of the $i$-type of a vertex. Vertices of $i$-types A, B, and C should, if $\LSF_5$ held, belong to degree $5$ components with generating functions $s_{4,1}$ or $s_{2,1,1,1}$, $s_{3,2}$ or $s_{2,2,1}$, and $s_{3,1,1}$, respectively. Moreover, in the case of a dual equivalence graph, the $i$-type determines the shape of the connected component of $(V, \sigma,E_{\imt} \cup E_{\imo} \cup E_{i})$ containing the vertex; for an illustration, compare Figure~\ref{fig:lambda5} with $i$-types A, B, and C in Figures~\ref{fig:type-A}, \ref{fig:type-B}, and \ref{fig:type-C}, respectively. 

\begin{definition}
  Let $\G$ be a signed, colored graph of type $(n,N)$ satisfying axioms $1, 2, 3$ and $5$. For $i \leq n$ with $i<N$, the \emph{$i$-type of a vertex $w$ of $\G$} admitting an $i$-neighbor such that $\sigma(w)_{i} = \sigma(E_{\imo}(w))_{i}$ is
    \begin{itemize}
    \item \emph{$i$-type A} if $\sigma(w)_{i} = \sigma(E_{\imo}(w))_{i}$ and $w$ does not admit an $\imt$-neighbor;
    \item \emph{$i$-type B} if $\sigma(w)_{i} = \sigma(E_{\imo}(w))_{i}$ and $w$ admits an $\imt$-neighbor and if $w$ admits an $\imo$-neighbor, then $\sigma(w)_{\imo} = -\sigma(E_{\imt}(w))_{\imo}$; otherwise, $\sigma(w)_{i} = -\sigma(E_{\imo}E_{\imt}(w))_{i}$;
    \item \emph{$i$-type C} if $\sigma(w)_{i} = \sigma(E_{\imo}(w))_{i}$ and $w$ admits an $\imt$-neighbor and if $w$ admits an $\imo$-neighbor, then $\sigma(w)_{\imo} = \sigma(E_{\imt}(w))_{\imo}$; otherwise, $\sigma(w)_{i} = \sigma(E_{\imo}E_{\imt}(w))_{i}$.
    \end{itemize}
  \label{defn:type}
\end{definition}

\begin{figure}[ht]
  \begin{displaymath}
    \begin{array}{\cs{7}\cs{7}\cs{10}\cs{7}\cs{7}c}
      \rnode{a1}{\bullet} & \rnode{a2}{w} & \rnode{a3}{\bullet} & & \rnode{a4}{\bullet} & \rnode{a5}{w} 
    \end{array}
    \psset{nodesep=3pt,linewidth=.1ex}
    \everypsbox{\scriptstyle}
    \ncline {a1}{a2} \naput{\imo}
    \ncline {a2}{a3} \naput{i}
    \ncline {a4}{a5} \naput{i}
  \end{displaymath}
  \caption{\label{fig:type-A} An illustration of vertices $w$ of $i$-type A and neighboring $E_{\imt}$ and $E_{\imo}$ edges.}
\end{figure}

\begin{figure}[ht]
  \begin{displaymath}
    \begin{array}{\cs{7}\cs{7}\cs{10}\cs{7}\cs{7}\cs{7}\cs{7}c}
      \rnode{b1}{\bullet} & \rnode{b2}{w} & \rnode{b3}{\bullet} &
      & \rnode{b4}{\bullet} & \rnode{b5}{w} & \rnode{b6}{\bullet} & \rnode{b7}{\bullet} \\[5ex]
      \rnode{B1}{\bullet} & \rnode{B2}{w} & \rnode{B3}{\bullet} &
      & \rnode{B4}{\bullet} & \rnode{B5}{w} & \rnode{B6}{\bullet} & \rnode{B7}{\bullet} \\[3ex]
      \rnode{BB1}{\bullet} & \rnode{BB2}{\bullet} & &
      & & & \rnode{BB6}{\bullet} & \rnode{BB7}{\bullet} 
    \end{array}
    \psset{nodesep=3pt,linewidth=.1ex}
    \everypsbox{\scriptstyle}
    \ncline[offset=2pt] {b1}{b2} \naput{\imt}
    \ncline[offset=2pt] {b2}{b1} \naput{\imo}
    \ncline {b2}{b3} \naput{i}
    \ncline {b4}{b5} \naput{i}
    \ncline {b5}{b6} \naput{\imt}
    \ncline[offset=2pt] {b6}{b7} \naput{\imo}
    \ncline[offset=2pt] {b7}{b6} \naput{i}
    \ncline {B1}{B2} \naput{\imt}
    \ncline {B1}{BB1} \nbput{\imo}
    \ncline {B2}{BB2} \nbput{\imo}
    \ncline {B2}{B3} \naput{i}
    \ncline {B4}{B5} \naput{i}
    \ncline {B5}{B6} \naput{\imt}
    \ncline {B6}{B7} \naput{\imo}
    \ncline {B6}{BB6} \naput{i}
    \ncline {B7}{BB7} \naput{i}
  \end{displaymath}
  \caption{\label{fig:type-B} An illustration of vertices $w$ of $i$-type B and neighboring $E_{\imt}$ and $E_{\imo}$ edges.}
\end{figure}

\begin{figure}[ht]
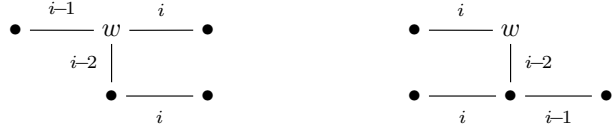

  \begin{displaymath}
    \begin{array}{\cs{7}\cs{7}\cs{10}\cs{7}\cs{7}\cs{7}c}
      \rnode{c1}{\bullet} & \rnode{c2}{w} & \rnode{c3}{\bullet} &
      & \rnode{c4}{\bullet} & \rnode{c5}{w} & \\[3ex]
      & \rnode{cc2}{\bullet} & \rnode{cc1}{\bullet} &
      & \rnode{cc4}{\bullet} & \rnode{cc5}{\bullet} &
      \rnode{cc6}{\bullet} 
    \end{array}
    \psset{nodesep=3pt,linewidth=.1ex}
    \everypsbox{\scriptstyle}
    \ncline {c1}{c2} \naput{\imo}
    \ncline {c2}{c3} \naput{i}
    \ncline {c2}{cc2} \nbput{\imt}
    \ncline {cc2}{cc1} \nbput{i}
    \ncline {c4}{c5} \naput{i}
    \ncline {c5}{cc5} \naput{\imt}
    \ncline {cc4}{cc5} \nbput{i}
    \ncline {cc5}{cc6} \nbput{\imo}
  \end{displaymath}
  \caption{\label{fig:type-C} An illustration of vertices $w$ of $i$-type C and neighboring $E_{\imt}$ and $E_{\imo}$ edges.}
\end{figure}

For $i \leq 3$, the $i$-type of a vertex cannot be A, B or C since there are not $E_{\imt}$ edges. Therefore when we refer to a vertex as having (or not having) $i$-type A, B or C, we implicitly assume that the vertex has an $i$-edge and that $i \geq 4$. Notice that $i$-type W is mutually exclusive of $i$-types A, B and C.

The $i$-type of $w$ is determined by the connected component of $E_{\imt} \cup E_{\imo}$ containing $w$ (note that this restriction includes $\sigma_i$ for every vertex on the restricted component). In general, for $i$-types B and C, if $w$ admits an $\imt$-neighbor but not an $\imo$-neighbor, then by axiom $3$, $E_{\imt}(w)$ admits an $\imo$-neighbor. Figures~\ref{fig:type-A}, \ref{fig:type-B}, and \ref{fig:type-C} show the $E_{\imt}$, $E_{\imo}$ and $E_{i}$ edges neighboring a vertex with a given $i$-type. The top row for $i$-type B are the possibilities in a dual equivalence graph, while the lower row gives the additional possibilities in the more general setting when axiom $4$ does not hold.

In a dual equivalence graph, a vertex $w$ has $i$-type C if and only if $E_{\imt}(w)$ has $i$-type C. Furthermore, in a dual equivalence graph where no vertex has $\imo$-type W, $w$ has $i$-type C if and only if both admit flat $i$-edges. This motivates the following definition.

\begin{definition}
  Let $\G$ be a signed, colored graph of type $(n,N)$ satisfying axioms $1, 2, 3$ and $5$. For $i < n$, a \emph{flat $i$-chain} is a sequence $(x_1,x_2,\ldots,x_{2h-1},x_{2h})$ of distinct vertices admitting $\imt$-edges such that
  \begin{displaymath}
    x_{2j-1} = E_i ( x_{2j} ) 
    \hspace{2em} \mbox{and} \hspace{2em}
    x_{2j+1} = E_{\imt} (E_{\imo} E_{\imt})^{m_j} ( x_{2j} ) ,
  \end{displaymath}
  for nonnegative integers $m_j$ such that $\left(E_{\imo} E_{\imt} \right)^{m_j} (x_{2j})$ does not have $\imo$-type W.
  \label{defn:flat-chain}
\end{definition}

\begin{figure}[ht]
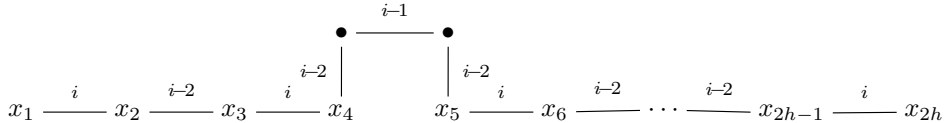

  \begin{displaymath}
    \begin{array}{\cs{7}\cs{7}\cs{7}\cs{7}\cs{7}\cs{7}\cs{7}\cs{7}c}
      & & & \rnode{D4}{\bullet} & \rnode{D5}{\bullet} & & & & \\[4ex]
      \rnode{d1}{x_{1}} & \rnode{d2}{x_{2}} & \rnode{d3}{x_{3}} & \rnode{d4}{x_{4}} & \rnode{d5}{x_{5}} & \rnode{d6}{x_{6}} & \rnode{c}{\cdots} & \rnode{d7}{x_{2h-1}} & \rnode{d8}{x_{2h}} 
    \end{array}
    \psset{nodesep=3pt,linewidth=.1ex}
    \everypsbox{\scriptstyle}
    \ncline {d1}{d2} \naput{i}
    \ncline {d2}{d3} \naput{\imt}
    \ncline {d3}{d4} \naput{i}
    \ncline {d4}{D4} \naput{\imt}
    \ncline {D4}{D5} \naput{\imo}
    \ncline {D5}{d5} \naput{\imt}
    \ncline {d5}{d6} \naput{i}
    \ncline {d6}{c} \naput{\imt}
    \ncline {c}{d7} \naput{\imt}
    \ncline {d7}{d8} \naput{i}
  \end{displaymath}
  \caption{\label{fig:flat-chain} An illustration of a flat $i$-chain.}
\end{figure}

Implicit in the definition is the fact that every vertex on a flat $i$-chain has an $i$-edge. However, we now require that each vertex also has an $\imt$-edge. By dual equivalence axioms $1$ and $2$, this forces each $i$-edge of a flat $i$-chain to be flat.

In a dual equivalence graph, a vertex of $i$-type A does not belong to a flat $i$-chain, a vertex of $i$-type B belongs to a flat $i$-chain of maximal length $2$, and a vertex of $i$-type C belongs to a flat $i$-chain of maximal length $4$. Define $C_i(\G)$ to be the set of vertices that lie on a flat $i$-chains of length greater than $4$, i.e.
\begin{equation}
  C_i(\G) \ = \ \left\{ x \in V \ | \ x = x_{j} \mbox{ on a flat $i$-chain of length $2h$ with } 2<j<2h-1 \right\}.
\label{eqn:X}
\end{equation}

Together with $W_i(\G)$, $C_i(\G)$ measures how far $\G$ is from satisfying dual equivalence axiom $4$, specifically, how many connected components of $(V,\sigma,E_{\imt} \cup E_{\imo} \cup E_{i})$ do not appear in Figure~\ref{fig:lambda5}.

\begin{proposition}
  Let $\G$ be a locally Schur positive graph of type $(n,n)$. Then $\G$ satisfies dual equivalence axiom $4$ if and only if both $W_i(\G)$ and $C_i(\G)$ are empty for all $1 < i < n$.
  \label{prop:empty}
\end{proposition}

\begin{proof}
  When axiom $4$ holds for $\G$, as discussed above, flat $i$-chains have length at most $4$. Therefore, by Proposition~\ref{prop:empty-W}, both $W_i(\G)$ and $C_i(\G)$ are empty.

  Now suppose that both $W_i(\G)$ and $C_i(\G)$ are empty for all $i$. By Proposition~\ref{prop:empty-W}, $\LSF_4$ holds. In particular, in \eqref{eqn:X}, we always have $m_j = 0$ for Definition~\ref{defn:flat-chain}. If a vertex $u$ admitting an $i$-edge lies on a flat $i$-chain of length $4$, say $(x_1,x_2,x_3,x_4)$, then we claim $x_4 = E_{\imt} (x_1)$ and the component of $E_{\imt} \cup E_{\imo} \cup E_i$ containing $u$ appears as the fourth graph in Figure~\ref{fig:lambda5} (type C). By axiom $2$ and the flatness of the $i$-edges of the chain, we may assume by symmetry that $x_1,x_3$ admit $\imo$-edges but $x_2,x_4$ do not. If the claim is false, then either $E_{\imt}(x_1) = E_{\imo}(x_1)$ or there is a flat $i$-edge at $E_{\imt}(x_1)$. Similarly, either $E_{\imt}(x_4)$ has $i$-type W or there is a flat $i$-edge at $E_{\imt}(x_4)$. Since no flat $i$-chain can have length greater than $4$, we must have $E_{\imt}(x_1) = E_{\imo}(x_1)$ and $E_{\imt}(x_4)$ has $i$-type W. Since $x_3$ does not have $i$-type W, $E_{\imo}(x_3)$ does not admit an $i$-edge. Since $E_{\imo}(x_3)$ cannot have $\imo$-type W, it also does not admit an $\imt$-edge. Therefore the component of $E_{\imt} \cup E_{\imo} \cup E_i$ containing $u$ appears as in Figure~\ref{fig:no_loop}. Neither of the two possible signature assignments for this graph results in a Schur positive generating function, so we have our contradiction.

  \begin{figure}[ht]
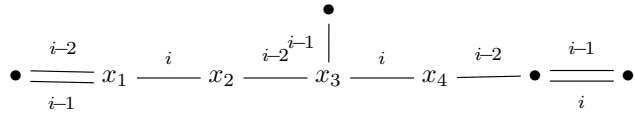

    \begin{displaymath}
      \begin{array}{\cs{7}\cs{7}\cs{7}\cs{7}\cs{7}\cs{7}c}
        & & & \rnode{t2}{\B} & & & \\[3ex]
        \rnode{x0}{\B} & \rnode{x1}{x_1} & \rnode{x2}{x_2} & \rnode{x3}{x_3} & \rnode{x4}{x_4} & \rnode{x5}{\B} & \rnode{t3}{\B}
      \end{array}
    \psset{nodesep=3pt,linewidth=.1ex}
    \everypsbox{\scriptstyle}
    \ncline {x3}{t2} \naput{\imo}
    \ncline[offset=2pt] {x5}{t3} \naput{\imo}
    \ncline[offset=2pt] {t3}{x5} \naput{i}
    \ncline[offset=2pt] {x0}{x1} \naput{\imt}
    \ncline[offset=2pt] {x1}{x0} \naput{\imo}
    \ncline {x1}{x2} \naput{i}
    \ncline {x2}{x3} \naput{\imt}
    \ncline {x3}{x4} \naput{i}
    \ncline {x4}{x5} \naput{\imt}
    \end{displaymath}
    \caption{\label{fig:no_loop}The flat $i$-chain of length $4$ that is not a loop.}
  \end{figure}

  If a vertex $u$ admitting an $i$-edge does not lie on a flat $i$-chain, then either $u$ or $E_i(u)$ does not admit an $\imt$-edge. If neither of them does, then the $i$-edge is flat, and exactly one of them, say $u$, must admit an $\imo$-edge. Then $E_i(u) \neq E_{\imo}(u)$, and so from the earlier discussion, $u$ cannot have $i$-type W. Therefore $E_{\imo}(u)$ does not admit an $i$-edge. Since $u$ does not admit an $\imt$-edge, $E_{\imo}(u)$ admits an $\imt$-edge and does not have $\imo$-type W. Therefore $E_{\imt} E_{\imo} (u)$ admits neither an $\imo$-edge nor an $i$-edge, and the component of $E_{\imt} \cup E_{\imo} \cup E_i$ containing $u$ appears as the second graph in Figure~\ref{fig:lambda5} (type A).

  If exactly one of $u$ and $E_i(u)$ has an $\imt$-edge, say $u$ does, then $u$ has $i$-type W and, by the previous arguments, $E_{\imo}(u) = E_{i}(u)$. Then since $u$ is assumed to admit an $\imt$-edge, it cannot have $\imo$-type W, and so $E_{\imt}(u)$ must admit an $i$-edge and no $\imo$-edge. Therefore $E_{\imt}(u)$ has a flat $i$-edge, and so $E_i E_{\imt}(u)$ admits an $\imt$-edge. If that $\imt$-edge has $\imo$-type W, then the component of $E_{\imt} \cup E_{\imo} \cup E_i$ containing $u$ appears as the third graph in Figure~\ref{fig:lambda5} (type B). If not, then $E_{\imt} E_i E_{\imt}(u)$ admits an $i$-edge but not an $\imo$-edge, and so the $i$-edge is flat and $E_i E_{\imt} E_i E_{\imt}(u)$ will admit an $\imt$-edge, resulting in a flat $i$-chain of length $4$. By the previous analysis, this must be a loop, giving us our contradiction.
\end{proof}

\begin{proposition}
  Let $\G$ be a signed, colored graph of type $(i,N)$ such that the $(\imt,N)$-restriction of $\G$ is a dual equivalence graph and the $(i,N)$-restriction of $\G$ satisfies axiom $4$. If a vertex $w$ of $\G$ has $i$-type W, then any vertex $v$ on the $i$-package of $w$ that has $i$-type C satisfies $E_{\imt}(v) = E_{\imh}(v)$.
  \label{prop:type-C}  
\end{proposition}

\begin{proof}
  For $\G$ a dual equivalence graph of type $(n,N)$ with $i \leq n$, we claim that if a vertex $w$ of $\G$ has $i$-type W, then any vertex $v$ on the $i$-package of $w$ that has $i$-type C satisfies $E_{\imt}(v) = E_{\imh}(v)$. By Theorem~\ref{thm:isomorphic} and Lemma~\ref{lem:extend-signs}, we may assume $\G = \G_{\mu,A}$ for some partition $\mu$ of $i$ and some augmenting tableau $A$ containing entries $\ipo,\ldots,N$. Let $\lambda$ be the uniquely determined shape of $\mu$ together with the cell in $A$ containing $\ipo$. A tableau $T \in \G_{\lambda}$ has $i$-type W if and only if both $\imt$ and $\ipo$ lie between $\imo$ and $i$ in the reading word of $T$. From the proof of Theorem~\ref{thm:cover}, a tableau $T \in \G_{\lambda}$ has $i$-type C if and only if $\imo$ lies between $i$ and $\ipo$ in the reading word of $T$. For $h \leq \imh$, an $E_h$ edge does not change the positions of entries greater than $\imt$, and for $h \geq \iph$, an $E_h$ edge does not change the positions of entries less than $\ipt$. In particular, the positions of $\imo,i,\ipo$ are constant on $i$-packages. The result follows, as does the proposition.
\end{proof}

By Proposition~\ref{prop:empty}, axiom $4$ holds if and only if $W_i$ and $C_i$ are empty. The map $\varphi_i^w$ reduces the cardinality of $W_i$. The second transformation, $\psi_i^x$, depicted in Figure~\ref{fig:psi}, takes as input an element $x \in C_i(\G)$. The goal with $\psi_i^x$ is to redefine $i$-edges so that $E_{\imt} E_i(x) = E_i E_{\imt}(x)$, that is, so that $x$ is not in $C_i(\G)$.

\begin{figure}[ht]
  \begin{displaymath}
    \begin{array}{\cs{8}\cs{8}c}
      & \rnode{O1}{\B} & \\[4ex] 
      \rnode{a1}{u}  & \rnode{b1}{x}  & \rnode{c1}{\B} \\[5ex] 
      \rnode{e1}{\B} & \rnode{f1}{\B} & \\[5ex] 
      \rnode{h1}{\B} & \rnode{i1}{\B} & 
    \end{array}
    \hspace{4em}
    \begin{array}{\cs{8}\cs{8}\cs{8}c}
      & & \rnode{O}{\B} & \\[4ex] 
      \rnode{T}{\B} & \rnode{a2}{\B} & \rnode{b2}{x} & \rnode{c2}{\B} \\[5ex] 
      \rnode{B}{u} & \rnode{e2}{\B} & \rnode{f2}{\B} & \\[5ex] 
      \rnode{g2}{\B} & \rnode{h2}{\B} & \rnode{i2}{\B} & 
    \end{array}
    \psset{nodesep=3pt,linewidth=.1ex}
    \ncline {a1}{O1} \naput{\imo}
    \ncline {a1}{b1} \naput{i}
    \ncline {a1}{e1} \nbput{\imt}
    \ncline {b1}{f1} \naput{\imt}
    \ncline {f1}{c1} \nbput{\imo}
    \ncline {e1}{h1} \nbput{i}
    \ncline {f1}{i1} \naput{i}
    \ncline[linestyle=dashed] {<->}{e1}{f1} \nbput{\psi_i^x}
    \ncline[linestyle=dashed] {<->}{h1}{i1} \nbput{\psi_i^x}
    \ncline {O}{a2} \nbput{\imo}
    \ncline {T}{a2} \naput{\imt}
    \ncline {T}{B} \nbput{\imo}
    \ncline {B}{e2} \nbput{\imt}
    \ncline {a2}{b2} \naput{i}
    \ncline {b2}{f2} \naput{\imt}
    \ncline {f2}{c2} \nbput{\imo}
    \ncline {B}{g2} \nbput{i}
    \ncline {e2}{h2} \nbput{i}
    \ncline {f2}{i2} \naput{i}
    \ncline[linestyle=dashed] {<->}{e2}{f2} \nbput{\psi_i^x}
    \ncline[linestyle=dashed] {<->}{h2}{i2} \nbput{\psi_i^x}
  \end{displaymath}
  \caption{\label{fig:psi} An illustration of $\psi_i^x$ where $x \in \widetilde{C}_i(\G)$ and $u = (E_{\imo}E_{\imt})^m E_{i}(x)$ does not have $\imo$-type W, for $m=0,1$.}
\end{figure}

The definition of $\psi_i^x$ on the connected component of $E_{\imt} \cup E_i$ containing $x$ is straightforward provided neither $x$ nor $E_i(x)$ has $\imo$-type W. In general, $\psi_i^x$ will be defined whenever some vertex on the connected component of $E_{\imt} \cup E_{\imo}$ containing $E_i(x)$ does not have $\imo$-type W. As before, the first step in defining the transformation is to extend it to $i$-packages.

\begin{lemma}
  Let $\G$ be a signed, colored graph of type $(n,N)$ satisfying dual equivalence axioms $1,2,3$ and $5$, and suppose that the $(\imt,N)$-restriction of $\G$ is a dual equivalence graph.  Let $x$ not admit an $\imo$-neighbor but have a flat $i$-edge such that neither $x$ nor $\left( E_{\imo} E_{\imt} \right)^{m} E_i(x)$ has $\imo$-type W for some $m \geq 0$, and suppose all vertices between $x$ and $\left( E_{\imo} E_{\imt} \right)^{m} E_i(x)$ have flat $\imt$-edges throughout their $\imt$-packages. Then the $i$-package of $E_{\imt}(x)$ is isomorphic to the $i$-package of $E_{\imt}\left(E_{\imo} E_{\imt} \right)^{m} E_i(x)$.
\label{lem:psi-compatible}
\end{lemma}

\begin{proof}
  Let $u = \left( E_{\imo} E_{\imt} \right)^{m} E_i(x)$. By axioms $1,2$ and $5$, $E_{\imt}, E_{\imo}, E_i$ all commute with $E_h$ for $h \geq \iph$, so the restriction of the $i$-package of any vertex on the connected component of $E_{\imt} \cup E_{\imo} \cup E_i$ containing $x$ to $E_{\iph} \cup \cdots \cup E_{\nmo}$ are isomorphic.  Therefore we focus our attention on extending the restriction to $E_{2} \cup \cdots \cup E_{\imh}$.

  Since all $E_{\imt}$ edges between $E_i(x)$ and $u$ are flat along their $\imt$-packages, Lemma~\ref{lem:phi-compatible} applies to each. Therefore, since $E_{\imo}$ always gives an isomorphism of $\imo$-packages, the $\imo$-package of $E_i(x)$ is isomorphic to the $\imo$-package of $u$. Further, each $E_{\imt}$ or $E_{\imo}$ edge changes $\sigma_j$ for $j = \imh,\imt,\imo$, and so $\sigma(u)_{j} = \sigma(E_i(x))_j$ for $j \leq \imo$ and $j \geq \ipo$. By axiom $2$, the $E_{\imt}$ edges preserve $\sigma_i$. Therefore, by axiom $1$, $E_{\imt} E_i(x)$ does not admit an $i$-neighbor, so neither $E_{\imt} E_i(x)$ nor $E_{\imo} E_{\imt} E_i(x)$ has $i$-type W. Continuing the argument along to $u$, no vertex of the form $E_{\imt} \left( E_{\imo} E_{\imt} \right)^{k} E_i(x)$ admits an $i$-neighbor for $0 \leq k < m$, and so none of the vertices after $E_i(x)$ can have $i$-type W. In particular, each $E_{\imo}$ edge from $E_i(x)$ to $u$ preserves $\sigma_i$ as well, and so $\sigma(u) = \sigma(E_i(x))$.
  
  Therefore we have an isomorphism between the $(\imh,\imt)$-restrictions of the $i$-packages of $E_i(x)$ and $u$. By the same argument used in the proof of Lemma~\ref{lem:phi-compatible}, we invoke Lemma~\ref{lem:extend-signs} and the hypothesis that the $(\imt,N)$-restriction of $\G$ is a dual equivalence graph to extend this to an isomorphism between the $i$-packages of $E_i(x)$ and $u$ and $\sigma(u) = \sigma(E_i(x))$. Regarding $E_i$ as an isomorphism of $i$-packages, it follows that $x$ and $u$ also have isomorphic $i$-packages. Thus, by Theorem~\ref{thm:isomorphic} and Lemma~\ref{lem:extend-signs}, the connected components of the $(\imt,\imo)$-restriction of $\G$ containing $x$ and $u$ are both isomorphic to $\G_{\mu,A}$ for the same partition $\mu$ of $\imt$ and the same augmenting tableau $A$ consisting of a single cell containing $\imo$. Denote these isomorphisms by $f_x$ and $f_u$, respectively, and let $\lambda$ be the shape of $\mu$ augmented by $A$. 

  Since the $(\imo,\imo)$-restriction of $\G$ satisfies the hypotheses of Theorem~\ref{thm:cover}, the isomorphisms $f_x$ and $f_u$ extend to morphisms $\overline{f}_x$ and $\overline{f}_u$ from the connected components of the $(\imo,\imo)$-restriction of $\G$ containing $x$ and $u$ to $\G_{\lambda}$. The picture is very similar to Figure~\ref{fig:compatible}, though now the top map is $E_i$ and the extended maps are surjective though not necessarily injective. Despite the lack of injectivity, the uniqueness of $\lambda$ and the extended maps ensures that the $(\imt,\imo)$-restriction of $\G_{\lambda}$ containing $E_{\imt}(x)$ is isomorphic to the $(\imt,\imo)$-restriction of $\G_{\lambda}$ containing $E_{\imt}(u)$, thereby establishing the desired isomorphism of $i$-packages.
\end{proof}

As with Lemma~\ref{lem:phi-compatible}, the hypotheses of Lemma~\ref{lem:psi-compatible} cannot be relaxed, so these vertices are of particular importance. Therefore we define the set $C_i^0(\G) \subseteq C_i(\G)$ by
\begin{equation}
  C_i^0(\G) = \{ x \in C_i(\G) \ | \ \mbox{all vertices between $x$ and $\left( E_{\imo} E_{\imt} \right)^{m} E_i(x)$ have flat $\imt$-edges} \},
  \label{eqn:X-0}
\end{equation}
where $m \geq 0$ is such that $\left(E_{\imo} E_{\imt} \right)^{m}E_i(x)$ does not have $\imo$-type W.

\begin{proposition}
  If every connected component of $E_{\imh} \cup E_{\imt}$ appears in Figure~\ref{fig:lambda4}, then $C_i(\G) = C_i^0(\G)$.   
  \label{prop:psi-lambda4}
\end{proposition}

\begin{proof}
  If every connected component of $E_{\imh} \cup E_{\imt}$ appears in Figure~\ref{fig:lambda4}, then if the $\imt$-edge at $x$ is not flat, by axiom $4$, $E_{\imt}(x) = E_{\imh}(x)$. By axiom $2$, this ensures that $E_{\imt}(x)$ does not have $\imo$-type W, so we are in the case where $m=0$. By axiom $5$, $E_{\imh}E_i(x) = E_iE_{\imh}(x) = E_iE_{\imt}(x)$. For $x \in C_i(\G)$, both $E_i(x)$ and $E_{\imh}E_i(x)$ admit $\imt$-neighbors, so by axiom $4$ we have $E_{\imt}(E_i(x)) = E_{\imh}(E_i(x))$. Therefore $E_{\imt}E_i(x) = E_iE_{\imt}(x)$, contradicting the assumption that $x \in C_i(\G)$. Therefore, whenever every connected component of $E_{\imh} \cup E_{\imt}$ appears in Figure~\ref{fig:lambda4}, provided there exists $m$ for which $\left( E_{\imo} E_{\imt} \right)^{m} E_i(x)$ does not have $\imo$-type W, if $x \in C_i(\G)$ then $x \in C_i^0(\G)$ as well. In particular, $C_i(\G) = C_i^0(\G)$ in this case.
\end{proof}

Given $x \in C^0_i(\G)$, we use the isomorphism of Lemma~\ref{lem:psi-compatible} to define an involution $\psi_i^x$ as follows.

\begin{definition}
  For $x \in C^0_i(\G)$, let $u = \left( E_{\imo} E_{\imt} \right)^{m} E_i(x)$ be the first vertex on the connected component of $E_{\imt} \cup E_{\imo}$ containing $E_i(x)$ not having $\imo$-type W. Let $\phi$ denote the isomorphism of Lemma~\ref{lem:psi-compatible}. Define the involution $\psi_i^x$ on all vertices admitting an $i$-neighbor as follows.
  \begin{equation}
    \psi_i^x(v) = \left\{ \begin{array}{rl}
        \phi(v) & \mbox{if $v$ lies on the $i$-package of $E_{\imt}(x)$ or $E_{\imt} (u)$,} \\[1ex]  
        E_i \phi E_i(v) & \mbox{if $E_i(v)$ lies on the $i$-package of $E_{\imt}(x)$ or $E_{\imt} (u)$,}\\[1ex]
        E_i(v) & \mbox{otherwise.}
      \end{array} \right.
    \label{eqn:psi}
  \end{equation}
  Define $E'_i$ to be the set of pairs $\{v,\psi_i^x(v)\}$ for each $v$ admitting an $i$-neighbor. Define a signed, colored graph $\psi_i^x(\G)$ of type $(n,N)$ by
  \begin{equation}
    \psi_i^x(\G) = (V, \sigma, E_2 \cup\cdots\cup E_{\imo} \cup E'_i \cup E_{\ipo} \cup\cdots\cup E_{\nmo}). 
  \end{equation}
\label{defn:psi}
\end{definition}

We again abuse notation by letting $\phi$ denote both the isomorphism from the $i$-package of $x$ to the $i$-package of $u$ and its inverse. Note that $\psi_i^x = \psi_i^u$ when $m = 0$.

Parallel to the case with $\varphi_i$, the goal is to use the maps $\psi_i$ to transform $\LSP_5$ into $\LSF_5$. For example, the graph in Figure~\ref{fig:frog} is not a dual equivalence graph. Figure~\ref{fig:frog} shows the resulting dual equivalence graph after implementing $\psi_4$ as well as $\varphi_3$ and $\varphi_4$. 

\begin{figure}[ht]
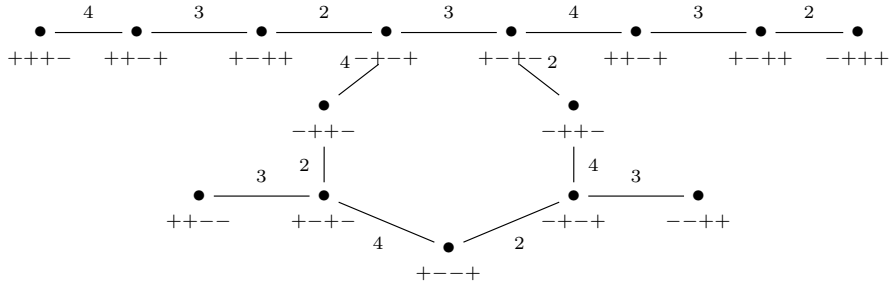

  \begin{displaymath}
    \begin{array}{\cs{3} \cs{2} \cs{2} \cs{2} \cs{2} \cs{2} \cs{2} \cs{2} \cs{2} \cs{2} \cs{2} \cs{2} \cs{2} \cs{3} c}
      \\[\cellsize]
      \sbull{A0}{+++-} & &
      \sbull{A1}{++-+} & & 
      \sbull{B1}{+-++} & & 
      \sbull{C1}{-+-+} & & 
      \sbull{D1}{+-+-} & & 
      \sbull{E1}{++-+} & & 
      \sbull{F1}{+-++} & &
      \sbull{F0}{-+++} \\[\cellsize] & & & & &
      \sbull{C2}{-++-} & & & & 
      \sbull{D2}{-++-} & & & & & \\[1.5\cellsize] & & &
      \sbull{B3}{++--} & & 
      \sbull{C3}{+-+-} & &  & & 
      \sbull{D3}{-+-+} & & 
      \sbull{E3}{--++} & & & \\ & & & & & & &
      \sbull{CD}{+--+} & & & & & & &
    \end{array}
    \psset{linewidth=.1ex,nodesep=3pt}
    \everypsbox{\scriptstyle}
    \ncline {A0}{A1} \naput{4}
    \ncline {F1}{F0} \naput{2}
    \ncline {A1}{B1} \naput{3}
    \ncline {B1}{C1} \naput{2}
    \ncline {C1}{D1} \naput{3}
    \ncline {D1}{E1} \naput{4}
    \ncline {E1}{F1} \naput{3}
    \ncline {C2}{C1C1} \naput{4}
    \ncline {D1D1}{D2} \naput{2}
    \ncline {C3}{C2C2} \naput{2}
    \ncline {D2D2}{D3} \naput{4}
    \ncline {B3}{C3} \naput{3}
    \ncline {C3}{CD} \nbput{4}
    \ncline {CD}{D3} \nbput{2}
    \ncline {D3}{E3} \naput{3}
  \end{displaymath}
  \caption{\label{fig:frog}A locally Schur positive graph with generating function $s_{4,1} + s_{3,2} + s_{3,1,1}$.}
\end{figure}

\begin{figure}[ht]
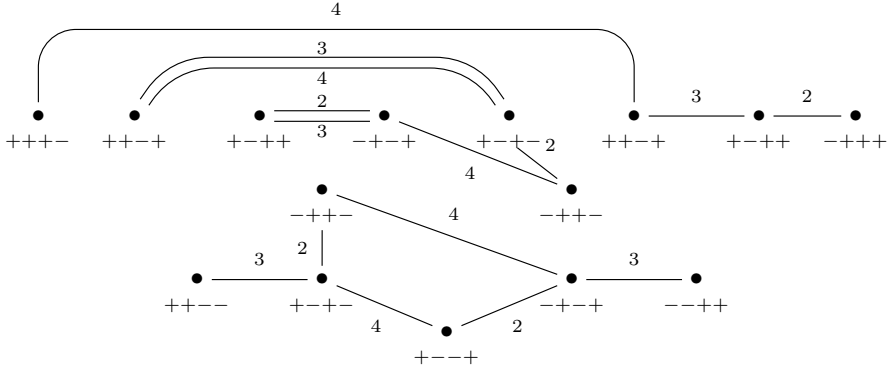

  \begin{displaymath}
    \begin{array}{\cs{3} \cs{2} \cs{2} \cs{2} \cs{2} \cs{2} \cs{2} \cs{2} \cs{2} \cs{2} \cs{2} \cs{2} \cs{2} \cs{3} c}
      \\[2\cellsize]
      \sbull{xA0}{+++-} & &
      \sbull{xA1}{++-+} & & 
      \sbull{xB1}{+-++} & & 
      \sbull{xC1}{-+-+} & & 
      \sbull{xD1}{+-+-} & & 
      \sbull{xE1}{++-+} & & 
      \sbull{xF1}{+-++} & &
      \sbull{xF0}{-+++} \\[\cellsize] & & & & &
      \sbull{xC2}{-++-} & & & & 
      \sbull{xD2}{-++-} & & & & & \\[1.5\cellsize] & & &
      \sbull{xB3}{++--} & & 
      \sbull{xC3}{+-+-} & &  & & 
      \sbull{xD3}{-+-+} & & 
      \sbull{xE3}{--++} & & & \\ & & & & & & &
      \sbull{xCD}{+--+} & & & & & & &
    \end{array}
    \psset{linewidth=.1ex,nodesep=3pt}
    \everypsbox{\scriptstyle}
    \ncdiag[angleA=90,angleB=90,arm=6.4ex,linearc=.5] {xA0}{xE1} \naput{4}
    \ncline[offset=2pt] {xB1}{xC1} \nbput{3}
    \ncline[offset=2pt] {xC1}{xB1} \nbput{2}
    \ncline {xE1}{xF1} \naput{3}
    \ncdiag[offset=2pt,angleA=55,angleB=125,arm=4.5ex,linearc=1]%
    {xA1}{xD1} \nbput{4}
    \ncdiag[offset=2pt,angleB=55,angleA=125,arm=4ex,linearc=1]%
    {xD1}{xA1} \nbput{3}
    \ncline {xF1}{xF0} \naput{2}
    \ncline {xD1xD1}{xD2} \naput{2}
    \ncline {xC2}{xD3} \naput{4}
    \ncline {xD2}{xC1} \naput{4}
    \ncline {xC3}{xC2xC2} \naput{2}
    \ncline {xCD}{xD3} \nbput{2}
    \ncline {xB3}{xC3} \naput{3}
    \ncline {xC3}{xCD} \nbput{4}
    \ncline {xD3}{xE3} \naput{3}
  \end{displaymath}
  \caption{\label{fig:dissect}The transformation of the graph in Figure~\ref{fig:frog} using $\varphi_3, \varphi_4$ and $\psi_4$.}
\end{figure}

The goal with $\psi_i^x$ is to reduce $C_i(\G)$ without increasing $W_i(\G)$. The following result shows that this happens whenever the $(i,N)$-restriction of $\G$ is a dual equivalence graph.

\begin{theorem}
  Let $\G$ be a locally Schur positive graph of type $(n,N)$ satisfying dual equivalence axioms $1,2,3$ and $5$, and suppose that $(\imt,N)$-restriction of $\G$ is a dual equivalence graph and that the $(i,N)$-restriction of $\G$ satisfies dual equivalence axiom $4$. Then $C_i^0(\G) = C_i(\G)$, and $C_i(\psi_i^x(\G))$ is a proper subset of $C_i(\G)$ and $W_i(\psi_i^x(\G)) = W_i(\G)$.
  \label{thm:psi-terminate}
\end{theorem}

\begin{proof}
  By Proposition~\ref{prop:psi-lambda4}, $C_i^0(\G) = C_i(\G)$. If $E_i(x)$ has $\imo$-type W, then its $\imt$-edge is non-flat, so $E_i(x) \not\in C_i(\G)$, contradicting the hypothesis that it is. Hence $x \in C_i(\G)$ with $m = 0$.

  The $i$-type of a vertex is determined by the connected component of $E_{\imt} \cup E_{\imo}$ containing it, so the $i$-type of a vertex remains unchanged by $\psi_i^x$. No $E_{\imh}$ edge on the $i$-package of $x$ or $E_i(x)$ is part of a double edge with $E_{\imt}$, so whether or not the vertex admits an $\imo$-neighbor is preserved. Therefore to show that $v \not\in C_i(\G)$ implies $v \not\in C_i(\psi_i^x(\G))$, we must show that if $E_{\imt}E_i(v) = E_iE_{\imt}(v)$, then $E_{\imt}\psi_i^x(v) = \psi_i^xE_{\imt}(v)$.

  It suffices to consider $v$ on the $i$-packages of $x$ and $E_{i}(x)$. By axiom $5$, $E_{\imt}$ and $E_i$ all commute with $E_h$ for $h \leq i\!-\!5$ and $h \geq \iph$. Therefore if the claim holds for some vertex $v$, the it holds for every vertex on the connected component of $E_2 \cup \cdots \cup E_{i\!-\!5} \cup E_{\iph} \cup \cdots \cup E_{\nmo}$ containing $v$. Since $\sigma(v)_{\imf} = \sigma(E_{\imt}(v))_{\imf}$ for $v = x, E_i(x)$, by axiom $3$ neither $E_{\imh}(x)$ nor $E_{\imh}E_i(x)$ admits an $\imt$-neighbor, so neither can have $i$-type C. By axiom $6$, it suffices to show the claim for $v = E_{\imf}(x)$.

  \begin{figure}[ht]
    \begin{displaymath}
      \begin{array}{\cs{8}\cs{8}\cs{8}\cs{8}\cs{8}\cs{8}\cs{8}c}
        \rnode{a1}{\B} & \rnode{b1}{\B} & \rnode{c1}{\B} & \rnode{d1}{u}  &
        \rnode{e1}{x}  & \rnode{f1}{\B} & \rnode{g1}{\B} & \rnode{h1}{\B} \\[5ex]
        \rnode{a2}{\B} & \rnode{b2}{\B} & \rnode{c2}{\B} & \rnode{d2}{\B} &
        \rnode{e2}{\B} & \rnode{f2}{\B} & \rnode{g2}{\B} & \rnode{h2}{\B} \\[5ex]
        & \rnode{b3}{\B} & \rnode{c3}{\B} & & & \rnode{f3}{\B} & \rnode{g3}{\B} & 
      \end{array}
      \psset{nodesep=3pt,linewidth=.1ex}
      \everypsbox{\scriptstyle}
      \ncline {a1}{b1} \naput{\imt}
      \ncline {b1}{c1} \naput{i}
      \ncline {c1}{d1} \naput{\imt}
      \ncline {d1}{e1} \naput{i}
      \ncline {e1}{f1} \naput{\imt}
      \ncline {f1}{g1} \naput{i}
      \ncline {g1}{h1} \naput{\imt}
      \ncline {b1}{b2} \nbput{\imf}
      \ncline {c1}{c2} \nbput{\imf}
      \ncline[offset=2pt] {a1}{a2} \nbput{\imf}
      \ncline[offset=2pt] {a2}{a1} \nbput{\imh}
      \ncline[offset=2pt] {d1}{d2} \nbput{\imf}
      \ncline[offset=2pt] {d2}{d1} \nbput{\imh}
      \ncline[offset=2pt] {e1}{e2} \nbput{\imf}
      \ncline[offset=2pt] {e2}{e1} \nbput{\imh}
      \ncline[offset=2pt] {h1}{h2} \nbput{\imf}
      \ncline[offset=2pt] {h2}{h1} \nbput{\imh}
      \ncline {f1}{f2} \naput{\imf}
      \ncline {g1}{g2} \naput{\imf}
      \ncline {b2}{c2} \naput{i}
      \ncline {d2}{e2} \naput{i}
      \ncline {f2}{g2} \naput{i}
      \ncline[offset=2pt] {b2}{b3} \nbput{\imh}
      \ncline[offset=2pt] {b3}{b2} \nbput{\imt}
      \ncline[offset=2pt] {c2}{c3} \nbput{\imh}
      \ncline[offset=2pt] {c3}{c2} \nbput{\imt}
      \ncline {b3}{c3} \naput{i}
      \ncline {f3}{g3} \naput{i}
      \ncline[offset=2pt] {f2}{f3} \nbput{\imh}
      \ncline[offset=2pt] {f3}{f2} \nbput{\imt}
      \ncline[offset=2pt] {g2}{g3} \nbput{\imh}
      \ncline[offset=2pt] {g3}{g2} \nbput{\imt}
    \end{displaymath}
    \caption{\label{fig:psi-X} Components of $E_{\imf} \cup E_{\imh} \cup E_{\imt}$ when $x$ has $\imt$-type B.}
  \end{figure}
  
  Consider the $\imt$-type of $x$ and $E_i(x)$, which must be the same since, by axiom $5$, $E_i$ commutes with both $E_{\imf}$ and $E_{\imh}$. From before, both $x$ and $E_i(x)$ have flat $\imt$-edge, and so, by axiom $4$, they cannot have $\imt$-type W. Since by assumption both admit an $\imf$-neighbor, they cannot have $\imt$-type A. If they have $\imt$-type C, then the top row of Figure~\ref{fig:psi-X} commutes with $E_{\imf}$, so $E_{\imf}(x) \in C_i(\G)$ and $E_{\imf}(x) \not\in C_i(\psi_i^x(\G))$. If they have $\imt$-type B, then, by axioms $4$ and $5$, the situation is as depicted in Figure~\ref{fig:psi-X} since none of the endpoints of the $i$-edges has $i$-type W by Proposition~\ref{prop:type-C}. From the figure, it is clear that applying $\psi_i^x$ adds no vertices to $C_i(\G)$, and so $C_i(\psi_i^x(\G))$ is indeed a proper subset of $C_i(\G)$.  Moreover, by Proposition~\ref{prop:type-C}, none of the vertices involved has $i$-type W, and so $W_i(\psi_i^x(\G)) = W_i(\G)$.
\end{proof}

\begin{figure}[ht]
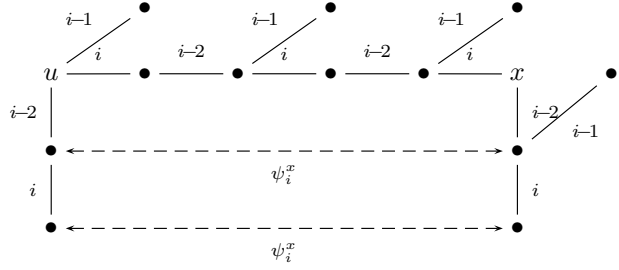

  \begin{displaymath}
    \begin{array}{\cs{7}\cs{7}\cs{7}\cs{7}\cs{7}\cs{7}\cs{7}c}
      & \rnode{O0}{\B} & & \rnode{O1}{\B} & & \rnode{O2}{\B} & \\[3ex] 
      \rnode{a1}{u}  & \rnode{a2}{\B} & \rnode{a3}{\B} & \rnode{b1}{\B} & \rnode{b2}{\B} & \rnode{c1}{x} &  \rnode{c2}{\B}\\[4ex] 
      \rnode{e1}{\B} & & & & & \rnode{f1}{\B} & \\[4ex] 
      \rnode{h1}{\B} & & & & & \rnode{i1}{\B} &
    \end{array}
    \psset{nodesep=3pt,linewidth=.1ex}
    \everypsbox{\scriptstyle}
    \ncline {a1}{O0} \naput{\imo}
    \ncline {a3}{O1} \naput{\imo}
    \ncline {b2}{O2} \naput{\imo}
    \ncline {a1}{a2} \naput{i}
    \ncline {a2}{a3} \naput{\imt}
    \ncline {a3}{b1} \naput{i}
    \ncline {b1}{b2} \naput{\imt}
    \ncline {b2}{c1} \naput{i}
    \ncline {a1}{e1} \nbput{\imt}
    \ncline {c1}{f1} \naput{\imt}
    \ncline {f1}{c2} \nbput{\imo}
    \ncline {e1}{h1} \nbput{i}
    \ncline {f1}{i1} \naput{i}
    \ncline[linestyle=dashed] {<->}{e1}{f1} \nbput{\psi_i^x}
    \ncline[linestyle=dashed] {<->}{h1}{i1} \nbput{\psi_i^x}
  \end{displaymath}
  \caption{\label{fig:long_psi} The long version of the involution $\psi_i^x$ for $m=0$ and $r=1$.}
\end{figure}

\begin{remark}
  Given $x \in C_i^0(\G)$ the third vertex of a flat $i$-chain of length greater than $6$, rather than taking $u = E_i(x)$ in Definition~\ref{defn:psi}, we make take $u = E_i E_{\imt}E_iE_{\imt}E_{i} (x)$ instead as indicated in Figure~\ref{fig:long_psi}. As with $\varphi_i^w$, Lemma~\ref{lem:psi-compatible} applies, and, in general, we may take $u = (E_{\imo}E_{\imt})^m E_i (E_{\imt}E_iE_{\imt}E_{i})^r (x)$. Then, from the proof of Theorem~\ref{thm:psi-terminate}, $C_i(\G)$ is decreased if and only if $x,u$ are the third and third from last vertices of the flat $i$-chain, but neither $C_i(\G)$ nor $W_i(\G)$ is ever increased.
  \label{rmk:long_psi}
\end{remark}

For $\G$ a locally Schur positive graph of type $(n,N)$ such that the $(\imt,N)$-restriction of $\G$ is a dual equivalence graph, $\psi_i^x(\G)$ also satisfies axioms $1, 2$ and $5$ by Lemma~\ref{lem:psi-compatible} and the definition of the map on $i$-packages. It turns out that $\psi_i^x$ also preserves axiom $3$, but this requires considerably more work to prove. However, when restricting to edges $E_i$ and lower, not only does axiom $3$ hold, but $\LSP_4$ does as well.

\begin{lemma}
  Let $\G$ be a locally Schur positive graph of type $(\ipo,\ipo)$ and suppose that the $(\imt,\ipo)$-restriction of $\G$ is a dual equivalence graph and that the $(i,\ipo)$-restriction of $\G$ satisfies dual equivalence axiom $4$. Then $C_i^0(\G) = C_i(\G)$, and $\psi_i^x(\G)$ has $\LSP_4$ for any $x \in C_i(\G)$.
  \label{lem:LSP4}
\end{lemma}

\begin{proof}
  Since the $(i,\ipo)$-restriction of $\G$ satisfies dual equivalence axiom $4$, by Proposition~\ref{prop:psi-lambda4}, $\psi_i^x$ may be applied for any $x \in C_i(\G)$ and, in this case, $E_i(x) \in C_i(\G)$ with $\psi_i^{E_i(x)} = \psi_i^x$. Since the $i$-edge between $x$ and $E_i(x)$ is flat, exactly one of these vertices admits an $\imo$-neighbor, say $E_i(x)$ does and $x$ does not. Since $x$ does not admit an $\imo$-neighbor, by axiom $3$, $E_{\imt}(x)$ does. Since $x \in C_i(\G)$, $E_{\imt}(x)$ has a flat $i$-edge, and so since $E_{\imt}(x)$ admits an $\imo$-neighbor, $E_i E_{\imt}(x)$ does not. On the other side, since $E_i(x)$ admits an $\imo$-neighbor but does not have $\imo$-type W, $E_{\imt} E_i(x)$ does not admit an $\imo$-neighbor. Therefore neither $E_i(E_{\imt}(x))$ nor $\psi_i^x(E_{\imt}(x)) = E_{\imt} E_i(x)$ admits an $\imo$-neighbor, so $\LSP_4$ of the connected component containing $E_{\imt}(x)$ is preserved. Similarly, since neither $E_i( E_i E_{\imt} E_i (x)) = E_{\imt} E_i(x)$ nor $\psi_i^x( E_i E_{\imt} E_i (x) ) = E_i E_{\imt}(x)$ admit an $\imo$-neighbor, $\LSP_4$ for the connected component containing $E_{\imt} E_i (x)$ is preserved. By axioms $2$ and $5$, both $E_{\imo}$ and $E_i$ commute with $E_h$ for $h \leq \imf$, so $\LSP_4$ is maintained on $E_2 \cup \cdots \cup E_{\imf}$. Since both $E_{\imt}(x)$ and $E_{\imt} E_i(x)$ have flat $\imt$-edges, neither has $\imt$-type W, and so if either, and hence both, admits an $\imh$-neighbor, the $\imh$-edge must preserve $\sigma_{\imo}$. Thus $\LSP_4$ extends across a single $E_{\imh}$ edge as well. By axiom $6$, the claim follows.
\end{proof}

As with Lemma~\ref{lem:LSP4-phi} for $\varphi_i^w$, the conclusion missing from Lemma~\ref{lem:LSP4} and the impediment to applying $\psi_i^x$ repeatedly is that $\LSP_5$ is not guaranteed to hold.

%
\section{Resolving axiom $4$}
%
\label{sec:LSP}

Unfortunately, neither $\varphi_i^w(\G)$ nor $\psi_i^x(\G)$ always has $\LSP_5$. If $\varphi_i^w(\G)$ or $\psi_i^x(\G)$ has $\LSP_5$ for at least one $w \in W_i^0(\G)$ or at least one $x \in C_i^0(\G)$, then by Theorems~\ref{thm:phi-terminate} and \ref{thm:psi-terminate}, we could always apply one of the maps until both $W_i$ and $C_i$ were both empty. With this idea in mind, define a set $U_i(\G) \subset W_i(\G) \cup C_i(\G)$ by
\begin{equation}
  U_i(\G) = \left\{ 
    \begin{array}{c}
      w \in W_i^0(\G) \\
      x \in C_i^0(\G)
    \end{array}
    \big| \ \mbox{components of $E_{\imt} \cup E_{\imo} \cup E_i$ are Schur positive in} 
    \begin{array}{c}
      \varphi_i^w(\G) \\
      \psi_i^x(\G)
    \end{array} 
  \right\} .
  \label{eqn:U}
\end{equation}

If $U_i(\G)$ is nonempty, then we may select either $w \in W_i(\G)$ or $x \in C_i(\G)$ and apply the corresponding map, either $\varphi_i^w$ or $\psi_i^x$, to transform $\G$ into a locally Schur positive graph that is measurably closer to being a dual equivalence graph. Therefore we focus on when $U_i(\G)$ is empty. 

We note that when $W_i(\G) \cup C_i(\G)$ is nonempty and $U_i(\G)$ is empty, no vertex has $i$-type A. 

\begin{lemma}
  Let $\G$ be a locally Schur positive graph of type $(\ipo,\ipo)$, and suppose that the $(\imt,\ipo)$-restriction of $\G$ is a dual equivalence graph and that the $(i,\ipo)$-restriction of $\G$ satisfies dual equivalence axiom $4$. If a connected component of $E_{\imt} \cup E_{\imo} \cup E_i$ not appearing in Figure~\ref{fig:lambda5} has no element in $U_i(\G)$, then no vertex on the component has $i$-type A.
  \label{lem:no-A}
\end{lemma}

\begin{proof}
  If there is a vertex of $i$-type A, then we claim there exists a vertex $u$ admitting an $\imt$-neighbor but not an $\imo$-neighbor nor an $i$-neighbor. Suppose $w$ has $i$-type A. Then either there exists $w$ admitting neither an $\imo$-neighbor nor an $\imt$-neighbor, or there exists $w$ admitting an $\imo$-neighbor but not an $\imt$-neighbor such that $E_{\imo}(w)$ does not admit an $i$-neighbor. In the former case, $\sigma_{\imh,\imt,\imo,i}(w) = +++-$ or $---+$, which, by $\LSP_5$, must be contribute to the Schur function $s_{4,1}$ or $s_{2,1,1,1}$, respectively. Therefore there must be a vertex $u$ with signature $\sigma_{\imh,\imt,\imo,i}(u) = -+++$ or $+---$, respectively, and so $u$ admits an $\imt$-neighbor but not an $\imo$-neighbor nor an $i$-neighbor. In the latter case, since the $(i,\ipo)$-restriction satisfies axiom $4$, $E_{\imo}(w)$ cannot have $\imo$-type W, and so $u = E_{\imt} E_{\imo} (w)$ will admit neither an $\imo$-neighbor nor an $i$-neighbor, thereby establishing the claim. 

  If $E_{\imo}E_{\imt}(u)$ has a flat $i$-edge, then the component appears in Figure~\ref{fig:lambda5} after all, so assume it has a non-flat $i$-edge. Since the component of $E_{\imo} \cup E_i$ begins at $E_{\imt}(u)$ with $\sigma(E_{\imt}(u))_{\imh,\imt,\imo,i} = +-++$ or $-+--$, $\LSP_4$ ensures that after an even number of alternating $E_{\imo}$ and $E_i$ edges, the components ends after a flat $i$-edge at a vertex $v$ with $\sigma(v)_{\imt,\imo,i} = ++-$ or $--+$, respectively. Each $E_{\imo}$ edge on the component must be flat, since otherwise by Figure~\ref{fig:lambda4} it would be a double edge with $E_{\imt}$, and by axiom $4$ $E_i$ fixes $\sigma_{\imh}$, so $\sigma(v)_{\imh} = \sigma(E_{\imt}(u))_{\imh}$. Therefore applying the longest possible $\varphi_i^{w}$, as discussed in Remark~\ref{rmk:long_phi}, removes a component of $E_{\imt} \cup E_{\imo} \cup E_i$ with generating function $s_{4,1}$ or $s_{2,1,1,1}$, respectively, contradicting the assumption that $U_i(\G)$ is empty. Hence no vertex on the component has $i$-type A.
\end{proof}

\begin{remark}
  The usefulness of Lemma~\ref{lem:no-A} lies in the fact that the degree $5$ generating function for a locally Schur positive nontrivial connected component of $E_{\imt} \cup E_{\imo} \cup E_i$ with no vertex of $i$-type A must be $ f = b s_{(3,2)} + c s_{(3,1,1)} + d s_{(2,2,1)}$ for $b,c,d \geq 0$. If $g$ is any Schur positive function, then $f-g$ is Schur positive if and only if it is a nonnegative sum of fundamental quasisymmetric functions. Therefore if a locally Schur positive piece is removed from such a component, the remaining structure is also locally Schur positive.
  \label{rmk:typeA}
\end{remark}

When $W_i(\G) \cup C_i(\G)$ is nonempty but $U_i(\G)$ is empty, we can apply a slight variant of the map $\psi_i$, denoted by $\gamma_i$, to $\G$ so that $U_i(\G)$ is nonempty. The map $\gamma_i$ is depicted in Figure~\ref{fig:gamma}. As usual, we begin by establishing the necessary isomorphism of $i$-packages.

\begin{lemma}
  Let $\G$ be a signed, colored graph of type $(n,N)$, and suppose that the $(\imt,N)$-restriction of $\G$ is a dual equivalence graph.  Let $z$ have a non-flat $i$-edge such that no vertex between $z$ and $\left( E_{\imo} E_{i} \right)^{m} (z)$ has $\imo$-type W, and suppose $z$ and $\left( E_{\imo} E_{i} \right)^{m} (z)$ have $i$-neighbors and flat $\imt$-edges. Then the $i$-package of $E_{\imt}(z)$ is isomorphic to the $i$-package of $E_{\imt}\left( E_{\imo} E_{i} \right)^{m} (z)$.
  \label{lem:gamma-compatible}
\end{lemma}

\begin{proof}
  Since $z$ has a non-flat $i$-edge, $E_{i}(z)$ must have $i$-type W, and so, too, must $E_{\imo} E_{i}(z)$. If $E_{\imo} E_{i}(z)$ has a non-flat $i$-edge, then the pattern persists so that all vertices between $z$ and $u = \left( E_{\imo} E_{i} \right)^{m} (z)$ have $i$-type W, and all $E_i$ edges between them are non-flat. Thus each $E_{\imo}$ edge between $z$ and $u$ toggles $\sigma_{\imt,\imo}$ by axiom $2$ and toggles $\sigma_i$ since it has $\imo$-type W. Similarly, each $E_i$ edge between $z$ and $u$ toggles $\sigma_{\imo,i}$ by axiom $2$ and toggles $\sigma_{\imt}$ since it is non-flat. Finally, since there is an even number of edges between $u$ and $z$ each of which toggles $\sigma_{\imt,\imo,i}$, we have $\sigma(z)_{\imt,\imo,i} = \sigma(u)_{\imt,\imo,i}$. In particular, both or neither admit an $\imt$-neighbor. If neither does, the result is clearly true, so assume both do. By Lemma~\ref{lem:phi-compatible} and the fact that $E_i$ gives an isomorphism of $i$-packages, the $i$-package of $z$ is isomorphic to the $i$-package of $u$. Now the same argument in Lemma~\ref{lem:psi-compatible} applies to extend the $i$-package isomorphisms across the flat $E_{\imt}$ edges since neither has $\imo$-type W.
\end{proof}

\begin{figure}[ht]
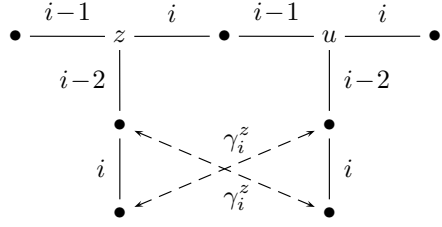

  \begin{displaymath}
    \begin{array}{\cs{8}\cs{8}\cs{8}\cs{8}c}
      \rnode{a0}{\B} & \rnode{b0}{z} & \rnode{c0}{\B} &
      \rnode{d0}{u} & \rnode{e0}{\B} \\[5ex] 
      & \rnode{b1}{\B} & & \rnode{d1}{\B} & \\[5ex]
      & \rnode{b2}{\B} & & \rnode{d2}{\B} & 
    \end{array}
    \psset{nodesep=3pt,linewidth=.1ex}
    \ncline {a0}{b0} \naput{\imo}
    \ncline {b0}{c0} \naput{i}
    \ncline {c0}{d0} \naput{\imo}
    \ncline {d0}{e0} \naput{i}
    \ncline {b0}{b1} \nbput{\imt}
    \ncline {d0}{d1} \naput{\imt}
    \ncline {b1}{b2} \nbput{i}
    \ncline {d1}{d2} \naput{i}
    \ncline[linestyle=dashed] {<->}{b1}{d2} \naput{\gamma_i^z}
    \ncline[linestyle=dashed] {<->}{b2}{d1} \nbput{\gamma_i^z}
  \end{displaymath}
  \caption{\label{fig:gamma} An illustration of $\gamma_i^z$.}
\end{figure}

Following the familiar pattern, we use the isomorphism of Lemma~\ref{lem:gamma-compatible} to define an involution $\gamma_i^z$.

\begin{definition}
  For $z$ not $\imo$-type W with a non-flat $i$-edge and a flat $\imt$-edge, let $u = \left( E_{\imo} E_{i} \right)^{m} (z)$, $m>0$, such that $u$ does not have $\imo$-type W and has a flat $\imt$-edge. Let $\phi$ denote the isomorphism of Lemma~\ref{lem:gamma-compatible}. Define the involution $\gamma_i^z$ on all vertices admitting an $i$-neighbor as follows.
  \begin{equation}
    \gamma_i^z(v) = \left\{ \begin{array}{rl}
        E_i\phi(v) & \mbox{if $v$ lies on the $i$-package of $E_{\imt}(z)$ or $E_{\imt} (u)$,} \\[1ex]  
        \phi E_i(v) & \mbox{if $E_i(v)$ lies on the $i$-package of $E_{\imt}(z)$ or $E_{\imt} (u)$,}\\[1ex]
        E_i(v) & \mbox{otherwise.}
      \end{array} \right.
    \label{eqn:gamma}
  \end{equation}
  Define $E'_i$ to be the set of pairs $\{v,\gamma_i^x(v)\}$ for each $v$ admitting an $i$-neighbor. Define a signed, colored graph $\gamma_i^x(\G)$ of type $(n,N)$ by
  \begin{equation}
    \gamma_i^x(\G) = (V, \sigma, E_2 \cup\cdots\cup E_{\imo} \cup E'_i \cup E_{\ipo} \cup\cdots\cup E_{\nmo}). 
  \end{equation}
\label{defn:gamma}
\end{definition}

\begin{remark}
  Note that the $m>0$ case for $\psi_i$ handles the situation where vertices have $\imo$-type W, that is, components of $E_{\imt} \cup E_{\imo}$ that do not appear in Figure~\ref{fig:lambda4}. The map $\gamma_i$ is similar, but it handles the situation where vertices have $i$-type W, that is, components of $E_{\imo} \cup E_{i}$ do not appear in Figure~\ref{fig:lambda4}. Therefore while applying $\psi_i$ for $m>0$ is relatively rare (e.g. does not arise when axiom $4$ holds for the $(i,N)$-restriction), $\gamma_i$ is often indispensable, at least in theory.
  \label{rmk:psi-v-gamma}
\end{remark}

The map $\gamma_i^z(\G)$ maintains $\LSP_4$ for the same reasons that $\psi_i^x$ does. Unlike $\varphi_i^w$ and $\psi_i^x$, the map $\gamma_i^z$ does not separate connected components of $E_{\imt} \cup E_{\imo} \cup E_i$, so $\LSP_5$ is trivially maintained. While $\gamma_i^z(\G)$ does not decrease $W_i$ or $C_i$, neither does it increase them. Its usefulness lies in the fact that it increases $U_i$. Specifically, when $W_i(\G) \cup C_i(\G)$ is nonempty and $U_i(\G)$ is empty, the flatness of the $\imt$-edges in Definition~\ref{defn:gamma} always holds, that is, $\gamma_i$ applies. 

\begin{lemma}
  Let $\G$ be a locally Schur positive graph of type $(\ipo,\ipo)$, and suppose that the $(\imt,\ipo)$-restriction of $\G$ is a dual equivalence graph and that the $(i,\ipo)$-restriction of $\G$ satisfies dual equivalence axiom $4$. If a connected component of $E_{\imt} \cup E_{\imo} \cup E_i$ not appearing in Figure~\ref{fig:lambda5} has no element in $U_i(\G)$, then for any $z$ such that $z$ has an $\imt$-edge and a non-flat $i$-edge but does not have $\imo$-type W, both $z$ and $u = (E_{\imo} E_i)^m(z)$ have flat $\imt$-edges, where $m>0$ and $u$ has an $i$-neighbor and does not have $\imo$-type W.
  \label{lem:gamma-flat}
\end{lemma}

\begin{proof}
  Since the $(i,\ipo)$-restriction of $\G$ satisfies dual equivalence axiom $4$, an $E_{\imt}$ edge between two vertices is non-flat if and only if it is a double edge with $E_{\imh}$. Similarly, a vertex has $\imo$-type W if and only if it has a double edge for $E_{\imt}$ and $E_{\imo}$. Keeping these in mind, we consider the $\imt$-edges at $u$ and $z$ in turn.

  First suppose that the $\imt$-edge at $u$ is non-flat, and so has a double edge for $E_{\imh}$ and $E_{\imt}$. By axiom $5$, we have $E_{\imh} E_i (u) = E_i E_{\imh} (u) = E_{i} E_{\imt} (u)$. By axiom 4 for $E_{\imh} \cup E_{\imt} \cup E_{\imo}$, $x = E_{\imh} E_{\imo} (u)$ must have a double edge for $E_{\imt}$ and $E_{\imo}$. By axiom $2$, $E_{\imh}(u) = E_{\imt}(u)$ does not admit an $\imo$-neighbor. Therefore the $i$-edge at $E_{\imh}(u) = E_{\imt}(u)$ is flat and $w = E_i E_{\imt} (u)$ admits an $\imo$-neighbor. Therefore we have the case depicted in Figure~\ref{fig:gamma-u}.

  \begin{figure}[ht]
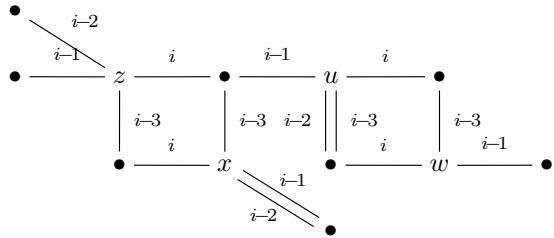

    \begin{displaymath}
      \begin{array}{\cs{8}\cs{8}\cs{8}\cs{8}\cs{8}c}
        \rnode{z4}{\B} & & & & & \\[3ex]
        \rnode{a0}{\B} & \rnode{a1}{z} & \rnode{a2}{\B} & \rnode{a3}{u} & \rnode{a4}{\B} & \\[5ex]
        & \rnode{b1}{\B} & \rnode{b2}{x} & \rnode{b3}{\B} & \rnode{b4}{w} & \rnode{b5}{\B} \\[3ex]
        & & & \rnode{c1}{\B} & &
      \end{array}
      \psset{nodesep=3pt,linewidth=.1ex}
      \everypsbox{\scriptstyle}
      \ncline {a1}{z4} \nbput{\imt}
      \ncline {a0}{a1} \naput{\imo}
      \ncline {a1}{a2} \naput{i}
      \ncline {a2}{a3} \naput{\imo}
      \ncline {a3}{a4} \naput{i}
      \ncline {a1}{b1} \naput{\imh}
      \ncline {a2}{b2} \naput{\imh}
      \ncline[offset=2pt] {a3}{b3} \naput{\imh}
      \ncline[offset=2pt] {b3}{a3} \naput{\imt}
      \ncline {a4}{b4} \naput{\imh}
      \ncline {b1}{b2} \naput{i}
      \ncline {b3}{b4} \naput{i}
      \ncline {b4}{b5} \naput{\imo}
      \ncline[offset=2pt] {b2}{c1} \nbput{\imt}
      \ncline[offset=2pt] {c1}{b2} \nbput{\imo}
    \end{displaymath}
    \caption{\label{fig:gamma-u} Situation when $u$ has a non-flat $\imt$-edge.} 
  \end{figure}

  Since the $i$-edge at $w$ is flat and $E_i(w)$ admits an $\imt$-neighbor, $w$ also admits an $\imt$-neighbor. In particular, $\sigma_{\imh,\imt,\imo,i}(w) = \sigma_{\imh,\imt,\imo,i}(x)$. If $E_{\imt} (w) = E_{\imh} (w)$, then $\varphi_i^u$ carves the connected component of $E_{\imt} \cup E_{\imo} \cup E_i$ into two pieces, one of which (the piece containing $u$ and $x$) has generating function $s_{(3,2)}$ or $s_{(2,2,1)}$. By Lemma~\ref{lem:no-A} and Remark~\ref{rmk:typeA}, this contradicts the assumption that $U_i(\G)$ is empty. If $E_{\imt} (w) = E_{\imo} (w)$, then applying $\varphi_i^u$ carves the connected component of $E_{\imt} \cup E_{\imo} \cup E_i$ into three pieces, one of which (the piece containing $u$ and $x$) has generating function $s_{(3,2)}$ or $s_{(2,2,1)}$, and the generating function of the piece containing $E_{\imh}(z)$ remains unchanged. Thus we once again have $u \in U_i(\G)$ contradicting the assumption that $U_i(\G)$ is empty. Therefore by axiom $4$ for $E_{\imh} \cup E_{\imt} \cup E_{\imo}$, we must have $E_{\imo} E_{\imh} (w) = E_{\imh} E_{\imo} (w)$. In particular, the $i$-edge at $u$ is not flat, so $E_i(u)$ admits an $\imo$-neighbor. Furthermore, $E_{i} (u)$ does not admit an $\imt$-neighbor since the $i$-edge is non-flat, and so by axiom $3$, $E_{\imo} E_i (u)$ must admit an $\imt$-neighbor. Therefore $E_{\imo} E_i (u)$ must also admit an $i$-neighbor, since otherwise it has $i$-type A, which is disallowed by Lemma~\ref{lem:no-A}. Since $E_{\imh} E_{\imo} (w) = E_{\imo} E_i (u)$, by axioms $2$ and $5$ $E_{\imo}(w)$ also admits an $i$-neighbor, and the scenario repeats by replacing $u$ with $E_{\imo} E_i (u)$. By the finiteness of the graph, a contradiction must eventually be reached. Therefore it must be that $u$ has a flat $\imt$-edge.

  Now suppose that the $\imt$-edge at $z$ is non-flat, and so has a double edge for $E_{\imh}$ and $E_{\imt}$. Let $x = E_i E_{\imt} (z) = E_i E_{\imh} (z)$. By axiom $2$, since $z$ has an $\imo$-neighbor, $E_{\imh}(z) = E_{\imt}(z)$ does not. Therefore, by axiom $3$, $x$ has a flat $i$-edge and both an $\imo$-neighbor and an $\imt$-neighbor. If $x$ has $\imo$-type W, then $E_{\imt}(x) = E_{\imo}(x)$. By axiom $5$, $x = E_{\imh} E_i (z)$. By axiom $4$ for $E_{\imh} \cup E_{\imt} \cup E_{\imo}$, since $u = E_{\imo} E_{\imh} (x)$, $u$ must have a double edge for $E_{\imh}$ and $E_{\imt}$ contradicting the assumption that $u$ has a flat $\imt$-edge. Thus we may assume that $x$ does not have $\imo$-type W.

  By axiom $4$ for $E_{\imh} \cup E_{\imt} \cup E_{\imo}$, since $E_{\imt}(x) \neq E_{\imo}(x)$, we must have $E_{\imh} E_{\imo} (x) = E_{\imo} E_{\imh} (x) = u$. Since $u$ admits an $i$-neighbor, by axiom $2$ so must $E_{\imh} (u) = E_{\imo} (x)$, forcing $x$ to have $i$-type W. Set $w = E_i E_{\imo} (x)$. By axiom $5$, we must have $E_{\imh} E_i (u) = w$. Since $u$ admits an $\imt$-neighbor and has a flat $i$-edge, axiom $3$ ensures that $E_i (u)$ admits an $\imt$-neighbor as well. Since $x$ admits an $\imt$-neighbor and does not have $\imo$-type W, by axiom $3$ again $E_{\imo} (x)$ does not admit an $\imt$-neighbor. By axiom $2$, $E_{\imh} (z) = E_{\imt} (z)$ does not admit an $\imo$-neighbor, and so by $\LSP_4$, $w$ must admit an $\imo$-neighbor. Therefore the $i$-edge at $w$ is not flat, and so $w$ must also admit an $\imt$-neighbor. Since both $w$ and $E_{\imh}(w) = E_i (u)$ admit $\imt$-edges, by axiom $4$ we must have $E_{\imh}(w) = E_{\imt}(w)$. Therefore the situation is as depicted in Figure~\ref{fig:gamma-flat}.

  \begin{figure}[ht]
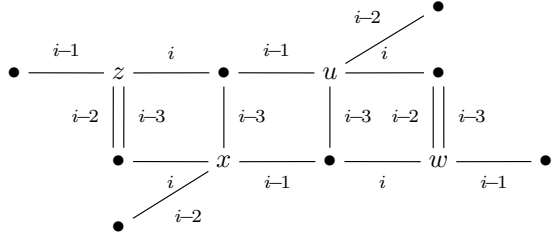

    \begin{displaymath}
      \begin{array}{\cs{8}\cs{8}\cs{8}\cs{8}\cs{8}c}
        & & & & \rnode{z4}{\B} & \\[3ex]
        \rnode{a0}{\B} & \rnode{a1}{z} & \rnode{a2}{\B} & \rnode{a3}{u} & \rnode{a4}{\B} & \\[5ex]
        & \rnode{b1}{\B} & \rnode{b2}{x} & \rnode{b3}{\B} & \rnode{b4}{w} & \rnode{b5}{\B} \\[3ex]
        & \rnode{c1}{\B} & & & & 
      \end{array}
      \psset{nodesep=3pt,linewidth=.1ex}
      \everypsbox{\scriptstyle}
      \ncline {a3}{z4} \naput{\imt}
      \ncline {a0}{a1} \naput{\imo}
      \ncline {a1}{a2} \naput{i}
      \ncline {a2}{a3} \naput{\imo}
      \ncline {a3}{a4} \naput{i}
      \ncline[offset=2pt] {a1}{b1} \naput{\imh}
      \ncline[offset=2pt] {b1}{a1} \naput{\imt}
      \ncline {a2}{b2} \naput{\imh}
      \ncline {a3}{b3} \naput{\imh}
      \ncline[offset=2pt] {a4}{b4} \naput{\imh}
      \ncline[offset=2pt] {b4}{a4} \naput{\imt}
      \ncline {b1}{b2} \nbput{i}
      \ncline {b2}{b3} \nbput{\imo}
      \ncline {b3}{b4} \nbput{i}
      \ncline {b4}{b5} \nbput{\imo}
      \ncline {b2}{c1} \naput{\imt}
    \end{displaymath}
    \caption{\label{fig:gamma-flat} Situation when $z$ has a non-flat $\imt$-edge and $x$ does not have $\imo$-type W.} 
  \end{figure}

  In this case, $u,x \in W_i^0(\G)$ and applying $\varphi_i^u = \varphi_i^x$ breaks the graph up to three components. By axiom $4$ for $E_{\imt} \cup E_{\imo}$, neither $E_{\imo}(z)$ nor $E_{\imo}(w)$ admits an $\imt$-neighbor. If neither admits an $i$-neighbor, then a component with generating function $s_{(3,1,1)}$ has been removed. By axiom $4$, $E_{\imt}(x)$ does not admit an $\imo$-edge, and so by axioms $2$ and $3$ it must admit a necessarily flat $i$-edge. If $E_i E_{\imt} (x)$ has $\imo$-type W, then a component with generating function $s_{(3,2)}$ or $s_{(2,2,1)}$ has been removed thereby ensuring that the third component remains locally Schur positive as well, contradicting the assumption that $U_i(\G)$ is empty. If $E_i E_{\imt} (x)$ does not have $\imo$-type W, then axioms $2,3,4$ again ensure that $E_{\imt} E_i E_{\imt} (x)$ admits a necessarily flat $i$-edge. In this case, $\psi_i^{E_i E_{\imt} (x)}$ applies to the original component without breaking it, once again contradicting the assumption that $U_i(\G)$ is empty. Thus, by symmetry, we may assume $E_{\imo} (w)$ admits an $i$-neighbor. In this case, $w$ has $i$-type W, and so $w \in W_i^0(\G)$. Now we are once again in the case of the right side of Figure~\ref{fig:phi-degree4} (with $w$ representing $w$), and $\varphi_i^w$ removes a component with generating function $s_{(3,2)}$ or $s_{(2,2,1)}$, either way maintaining local Schur positivity and contradicting that $U_i(\G)$ is empty.
\end{proof}

Next we show that when $W_i(\G) \cup C_i(\G)$ is nonempty and $U_i(\G)$ is empty, the structure of connected components of $E_{\imt} \cup E_{\imo} \cup E_i$ is that of a rooted tree.

\begin{lemma}
  Let $\G$ be a locally Schur positive graph of type $(\ipo,\ipo)$, and suppose that the $(\imt,\ipo)$-restriction of $\G$ is a dual equivalence graph and that the $(i,\ipo)$-restriction of $\G$ satisfies dual equivalence axiom $4$. If a connected component of $E_{\imt} \cup E_{\imo} \cup E_i$ not appearing in Figure~\ref{fig:lambda5} has no element in $U_i(\G)$, then, treating double edges as single edges, the component is a tree. Moreover, the component contains a unique vertex $w$ not admitting an $\imt$-neighbor such that $E_{\imo}(w) = E_i(w)$.
  \label{lem:tree}
\end{lemma}

\begin{proof}
  Suppose there is a sequence of at least three edges forming a loop. If the loop consists entirely of $E_{\imt}$ and $E_i$ edges, then following signatures around the loop using axiom $2$, the number of edges must be a multiple of $4$. If there are more than $4$ edges in the loop, one of the vertices lies in $C_i(\G)$ and applying $\psi_i$ will remove a component of $E_{\imt} \cup E_{\imo} \cup E_i$ with generating function $s_{(3,1,1)}$, contradicting the assumption that $U_i(\G)$ is empty. If there are $4$ edges, then some vertex on the loop is in $W_i(\G)$ and applying $\varphi_i$ does not split the component, thus maintaining local Schur positivity and contradicting that $U_i(\G)$ is empty. Therefore the loop must contain a vertex of $i$-type W. In this case, there are only two ways for $\varphi_i$ to break the component: if the loop consists of more than two $E_{\imo}$ and $E_i$ edges or if the loop is as in Figure~\ref{fig:gamma-flat}. The latter case is resolved as in the proof of Lemma~\ref{lem:gamma-flat}, so the only possible loops are with $E_{\imo}$ and $E_i$ edges. Chasing signatures using axioms $2$ and $3$ shows that there are an even number, say $2k$, of edges in the loop and that every other vertex admits an $\imt$-neighbor.

  We next claim that there is a vertex $w$ such that $E_{\imo}(w) = E_i(w)$ and that this is the only loop consisting solely of $E_{\imo}$ and $E_i$ edges. If no vertex has left $i$-type B, meaning the component of $i$-type B on the left side of Figure~\ref{fig:type-B}, then $\LSP_5$ dictates that no vertex can have right $i$-type B either since both vertices can only contribute to the Schur function $s_{(3,2)}$ or $s_{(2,2,1)}$, and so no vertex can have $i$-type W. Thus all vertices have $i$-type C, in which case the finiteness of the graph ensures there is a closed loop of $E_{\imt}$ and $E_i$ edges, contradicting the previous result. Therefore there must be a vertex with left $i$-type B. Starting from this vertex, we can follow the graph outwards never looping back. If we reach a vertex with $i$-type C, then the $E_{\imo}$ leads to a leaf and the other edge continues on. If we reach a vertex with left $i$-type B, then we reach a leaf since we must follow the double edge between $E_{\imt}$ and $E_{\imo}$. If we reach a vertex with right $i$-type B, then we reach a vertex with $i$-type W. At this point, if $E_{\imo}$ and $E_i$ form a double edge, then we have reached a leaf. Otherwise, we branch in two directions. Therefore every path must end in a double edge. If all endings are at vertices with a double edge between $E_{\imt}$ and $E_{\imo}$, then there will be one more left $i$-type B component than right $i$-type B component, contradicting that $\G$ is $\LSP_5$.

  Follow edges from $w$, say starting with $E_{\imt}$. Each time we reach a vertex $v$ admitting an $i$-edge, either $v$ does not admit an $\imo$-neighbor, thereby forcing the $i$-edge to be flat, or $v \in W_i(\G)$. In either case, we cannot have $E_{\imo}(v) = E_i(v)$, so the vertex $w$ is unique up to interchanging $w$ and $E_{\imo}(w) = E_i(w)$. Since exactly one of the two admits an $\imt$-neighbor, the lemma follows.
\end{proof}

With the structure of graphs where $W_i(\G) \cup C_i(\G)$ is nonempty and $U_i(\G)$ is empty rigidly established, we now show that there exists $z$ such that $\gamma_i^z$ applies and facilitates the application of either $\phi_i^w$ or $\psi_i^x$.

\begin{theorem}
  Let $\G$ be a locally Schur positive graph of type $(\ipo,\ipo)$, and suppose that the $(\imt,\ipo)$-restriction of $\G$ is a dual equivalence graph and that the $(i,\ipo)$-restriction of $\G$ satisfies dual equivalence axiom $4$. If $W_i(\G) \cup C_i(\G)$ is nonempty and $U_i(\G)$ is empty, then there exists $z$ such that $U_i(\gamma_i^z(\G))$ is nonempty.
  \label{thm:nonempty}
\end{theorem}

\begin{proof}
  Fix a connected component of $E_{\imt} \cup E_{\imo} \cup E_i$ not appearing in Figure~\ref{fig:lambda5}. Recall that under the hypothesis that axiom $4$ holds for the $(i,\ipo)$-restriction of $\G$, if $w \in W_i(\G)$ then $\varphi_i^w$ applies and if $x \in C_i(\G)$ then $\psi_i^x$ applies, though neither necessarily preserves $\LSP_5$. Also, by Lemma~\ref{lem:gamma-flat}, if $z$ has a non-flat $i$-edge but does not have $\imo$-type W, then $\gamma_i^z$ applies.

  By Lemma~\ref{lem:tree}, the component is a rooted tree consisting of vertices of $i$-types W, B and C, with the root being the unique vertex not admitting an $\imt$-neighbor with a double edge for $\imo$ and $i$. Identify each connected component of $E_{\imt} \cup E_{\imo}$ as $i$-type C ($C$-node), left $i$-type B ($L$-node), or right $i$-type B and $i$-type W ($R$-node), where this last case has a vertex $v$ of right $i$-type B and a vertex $E_{\imt}(v)$ with $i$-type W. Consider the graph with nodes given by these components and directed edges given by $i$-edges directed away from the root. Since the graph is a tree, every node has a unique incoming $i$-edge. Furthermore, $L$-nodes correspond precisely to leaves, a $C$-node has one outgoing flat $i$-edge, and an $R$-node has one outgoing flat $i$-edge and one outgoing non-flat $i$-edge. In the case of the root, an $R$-node, the outgoing non-flat $i$-edge is a double edge with $E_{\imo}$, i.e. a loop back to the root, and this is the only loop in the graph.

  Figure~\ref{fig:peel} illustrates three situations that cannot arise in this graph when $U_i(\G)$ is empty. First, if an $R$-node goes to an $L$-node by a flat $i$-edge, then $\varphi_i$ preserves local Schur positivity as depicted in the left case of Figure~\ref{fig:peel}. Second, if a $C$-node goes to another $C$-node (necessarily by a flat $i$-edge), then $\psi_i$ preserves local Schur positivity as depicted in the middle case of Figure~\ref{fig:peel}. The third case is more complicated. If an $R$-node goes to another $R$-node by a flat $i$-edge and each of these $R$-nodes goes to an $L$-node by a non-flat $i$-edge, then $\psi_i$ preserves local Schur positivity. This is the rightmost case depicted in Figure~\ref{fig:peel}. The assumption that $U_i(\G)$ is empty forbids these cases.

  \begin{figure}[ht]
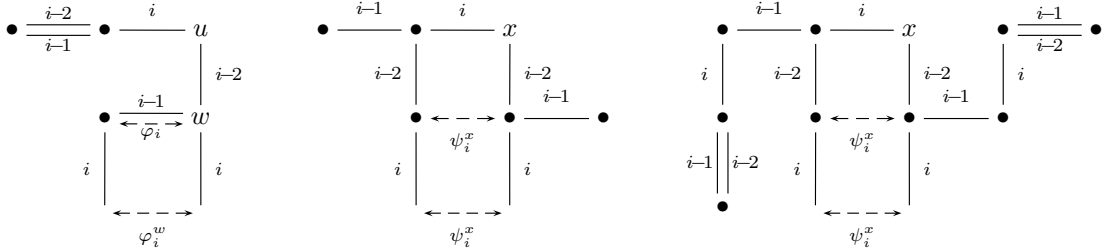

    \begin{displaymath}
      \begin{array}{\cs{7}\cs{7}c}
        \rnode{a3}{\B} & \rnode{a2}{\B} & \rnode{a1}{u} \\[5ex]
        & \rnode{b2}{\B} & \rnode{b1}{w}  \\[5ex]
        & \rnode{c2}{} & \rnode{c1}{}
      \end{array}
      \hspace{3em}
      \begin{array}{\cs{7}\cs{7}\cs{7}c}
        \rnode{A1}{\B} & \rnode{A2}{\B} & \rnode{A3}{x} &  \\[5ex]
        & \rnode{B2}{\B} & \rnode{B3}{\B} & \rnode{B4}{\B} \\[5ex]
        & \rnode{C2}{} & \rnode{C3}{} &
      \end{array}
      \hspace{3em}
      \begin{array}{\cs{7}\cs{7}\cs{7}\cs{7}c}
        \rnode{AA1}{\B} & \rnode{AA2}{\B} & \rnode{AA3}{x} & \rnode{AA4}{\B} & \rnode{ZZ4}{\B}  \\[5ex]
        \rnode{BB1}{\B} & \rnode{BB2}{\B} & \rnode{BB3}{\B} & \rnode{BB4}{\B}&  \\[5ex]
        \rnode{CC1}{\B} & \rnode{CC2}{} & \rnode{CC3}{} & & 
      \end{array}
      \psset{nodesep=3pt,linewidth=.1ex}
      \everypsbox{\scriptstyle}
      \ncline{a2}{a1} \naput{i}
      \ncline[offset=2pt]{a3}{a2} \nbput{\imo}
      \ncline[offset=2pt]{a2}{a3} \nbput{\imt}
      \ncline{b1}{a1} \nbput{\imt}
      \ncline[offset=2pt,linestyle=dashed]{<->}{b1}{b2} \nbput{\imo}
      \ncline[offset=2pt]{b2}{b1} \nbput{\varphi_i}
      \ncline{c1}{b1} \nbput{i}
      \ncline{c2}{b2} \naput{i}
      \ncline[linestyle=dashed]{<->}{c2}{c1} \nbput{\varphi^w_i}
      \ncline{A1}{A2} \naput{\imo}
      \ncline{A2}{A3} \naput{i}
      \ncline{A2}{B2} \nbput{\imt}
      \ncline{A3}{B3} \naput{\imt}
      \ncline[linestyle=dashed]{<->}{B2}{B3} \nbput{\psi^x_i}
      \ncline{B3}{B4} \naput{\imo}
      \ncline{B2}{C2} \nbput{i}
      \ncline{B3}{C3} \naput{i}
      \ncline[linestyle=dashed]{<->}{C2}{C3} \nbput{\psi^x_i}
      \ncline[offset=2pt]{AA4}{ZZ4} \nbput{\imt}
      \ncline[offset=2pt]{ZZ4}{AA4} \nbput{\imo}
      \ncline{AA1}{AA2} \naput{\imo}
      \ncline{AA2}{AA3} \naput{i}
      \ncline{AA1}{BB1} \nbput{i}
      \ncline{AA2}{BB2} \nbput{\imt}
      \ncline{AA3}{BB3} \naput{\imt}
      \ncline{AA4}{BB4} \naput{i}
      \ncline[linestyle=dashed]{<->}{BB2}{BB3} \nbput{\psi^x_i}
      \ncline{BB3}{BB4} \naput{\imo}
      \ncline[offset=2pt]{BB1}{CC1} \nbput{\imo}
      \ncline[offset=2pt]{CC1}{BB1} \nbput{\imt}
      \ncline{BB2}{CC2} \nbput{i}
      \ncline{BB3}{CC3} \naput{i}
      \ncline[linestyle=dashed]{<->}{CC2}{CC3} \nbput{\psi^x_i}
    \end{displaymath}
    \caption{\label{fig:peel} Three cases where $\varphi_i^w$ (left) or $\psi_i^x$ (middle and right) preserve $\LSP_5$.}
  \end{figure}

  Figure~\ref{fig:bar} depicts two cases that are easily resolved with $\gamma_i$. The left hand case depicts the situation when an $R$-node goes to a $C$-node by a non-flat $i$-edge and that $C$-node goes to an $L$-node (necessarily by a flat $i$-edge). In this case, applying $\gamma_i$ interchanges the subtree below the $R$-node with that subtree below the $C$-node, and the result is an instance of the leftmost case of Figure~\ref{fig:peel} where $\varphi_i$ can by applied. The right hand case of Figure~\ref{fig:bar} depicts that situation when both the flat and non-flat $i$-edges from an $R$-node go to $C$-nodes. Once again, applying $\gamma_i$ interchanges the subtrees, now resulting in an instance of the middle case of Figure~\ref{fig:peel} where $\psi_i$ can by applied.

  \begin{figure}[ht]
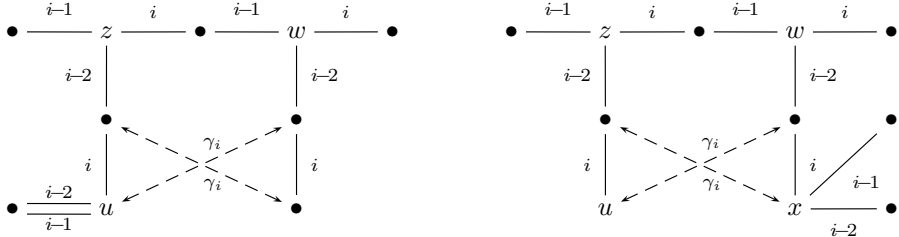

    \begin{displaymath}
      \begin{array}{\cs{7}\cs{7}\cs{7}\cs{7}c}
        \rnode{a1}{\B} & \rnode{a2}{z} & \rnode{a3}{\B} & \rnode{a4}{w} & \rnode{a5}{\B} \\[5ex]
        & \rnode{b2}{\B} & & \rnode{b4}{\B} & \\[5ex]
        \rnode{c1}{\B} & \rnode{c2}{u} & & \rnode{c4}{\B} &
      \end{array}
      \hspace{3em}
      \begin{array}{\cs{7}\cs{7}\cs{7}\cs{7}c}
        \rnode{A1}{\B} & \rnode{A2}{z} & \rnode{A3}{\B} & \rnode{A4}{w} & \rnode{A5}{\B}  \\[5ex]
        & \rnode{B2}{\B} & & \rnode{B4}{\B} & \rnode{B5}{\B}  \\[5ex]
        & \rnode{C2}{u} & & \rnode{C4}{x} & \rnode{C5}{\B} 
      \end{array}
      \psset{nodesep=3pt,linewidth=.1ex}
      \everypsbox{\scriptstyle}
      \ncline{a1}{a2} \naput{\imo}
      \ncline{a2}{a3} \naput{i}
      \ncline{a3}{a4} \naput{\imo}
      \ncline{a4}{a5} \naput{i}
      \ncline{a2}{b2} \nbput{\imt}
      \ncline{a4}{b4} \naput{\imt}
      \ncline{b2}{c2} \nbput{i}
      \ncline{b4}{c4} \naput{i}
      \ncline[offset=2pt]{c1}{c2} \nbput{\imo}
      \ncline[offset=2pt]{c2}{c1} \nbput{\imt}
      \ncline[linestyle=dashed]{<->}{b2}{c4} \naput{\gamma_i}
      \ncline[linestyle=dashed]{<->}{b4}{c2} \naput{\gamma_i}
      \ncline{A1}{A2} \naput{\imo}
      \ncline{A2}{A3} \naput{i}
      \ncline{A3}{A4} \naput{\imo}
      \ncline{A4}{A5} \naput{i}
      \ncline{A2}{B2} \nbput{\imt}
      \ncline{A4}{B4} \naput{\imt}
      \ncline{B2}{C2} \nbput{i}
      \ncline{B4}{C4} \naput{i}
      \ncline[linestyle=dashed]{<->}{B2}{C4} \naput{\gamma_i}
      \ncline[linestyle=dashed]{<->}{B4}{C2} \naput{\gamma_i}
      \ncline{C4}{B5} \nbput{\imo}
      \ncline{C4}{C5} \nbput{\imt}
    \end{displaymath}
    \caption{\label{fig:bar} Two cases where $\gamma_i^z$ allows $\varphi_i^w$ (left) or $\psi_i^x$ (right) to preserve $\LSP_5$.}
  \end{figure}
  
  We claim that this analysis resolves all configurations for edges coming from an $R$-node, except for the four shown in Figure~\ref{fig:node} or the case where the non-flat edge connected to another $R$-node. For the figures, we draw flat $i$-edges vertically and non-flat $i$-edges horizontally. If the non-flat $i$-edge of the $R$-node goes to an $L$-node, then the flat $i$-edge must either go to a $C$-node or another $R$-node, since the left side of Figure~\ref{fig:peel} precludes an $L$-node. In the former case, the (necessarily flat) $i$-edge from the $C$-node must go either to an $L$-node or an $R$-node, the left two cases of Figure~\ref{fig:node}, since the middle case of Figure~\ref{fig:peel} precludes another $C$-node. In the latter case, the non-flat $i$-edge of the second $R$-node must go to a $C$-node, since the right case of Figure~\ref{fig:peel} precludes another $L$-node. The flat $i$-edge from the second $R$-node cannot go to an $L$-node (by the left case of Figure~\ref{fig:peel}) nor to a $C$-node (by the right case of Figure~\ref{fig:bar}), so it must go to another $R$-node. Similarly, the (necessarily flat) $i$-edge from the $C$-node cannot go to another $C$-node (by the middle case of Figure~\ref{fig:peel}) nor to an $L$-node (by the left case of Figure~\ref{fig:bar}), so it must go to yet another $R$-node. The resulting case is the third of Figure~\ref{fig:node}. This handles all cases where the non-flat $i$-edge of an $R$-node goes to an $L$-node, so consider the alternative case in which the non-flat $i$-edge must go to a $C$-node. The analysis here is identical to the previous case, resulting in the rightmost case in Figure~\ref{fig:node}. Thus the claim is proved, and Figure~\ref{fig:node} contains all the remaining cases. Moreover, the root, necessarily an $R$-node, must be one of the middle two cases but with the non-flat edge looping instead of going to an $L$-node.

  \begin{figure}[ht]
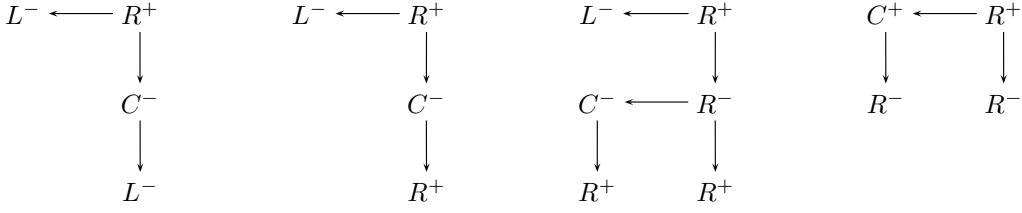

    \begin{displaymath}
      \begin{array}{\cs{7}c}
        \rnode{a1}{L^-} & \rnode{x1}{R^+} \\[5ex]
        \rnode{b1}{ } & \rnode{y1}{C^-} \\[5ex]
        \rnode{c1}{ } & \rnode{z1}{L^-}
      \end{array}
      \hspace{4em}
      \begin{array}{\cs{7}c}
        \rnode{a2}{L^-} & \rnode{x2}{R^+} \\[5ex]
        \rnode{b2}{ } & \rnode{y2}{C^-} \\[5ex]
        \rnode{c2}{ } & \rnode{z2}{R^+}
      \end{array}
      \hspace{4em}
      \begin{array}{\cs{7}c}
        \rnode{a3}{L^-} & \rnode{x3}{R^+} \\[5ex]
        \rnode{b3}{C^-} & \rnode{y3}{R^-} \\[5ex]
        \rnode{c3}{R^+} & \rnode{z3}{R^+}
      \end{array}
      \hspace{4em}
      \begin{array}{\cs{7}c}
        \rnode{a4}{C^+} & \rnode{x4}{R^+} \\[5ex]
        \rnode{b4}{R^-} & \rnode{y4}{R^-} \\[5ex]
        &
      \end{array}
      \psset{nodesep=3pt,linewidth=.1ex}
      \everypsbox{\scriptstyle}
      \ncline{->}{x1}{a1}
      \ncline{->}{x1}{y1}
      \ncline{->}{y1}{z1}
      \ncline{->}{x2}{a2}
      \ncline{->}{x2}{y2}
      \ncline{->}{y2}{z2}
      \ncline{->}{x3}{a3}
      \ncline{->}{x3}{y3}
      \ncline{->}{y3}{b3}
      \ncline{->}{b3}{c3}
      \ncline{->}{y3}{z3}
      \ncline{->}{x4}{a4}
      \ncline{->}{a4}{b4}
      \ncline{->}{x4}{y4}
    \end{displaymath}
    \caption{\label{fig:node} The four possible scenarios for edges emanating from an $R^+$-node, where horizontal edges are flat and vertical edges are non-flat.}
  \end{figure}

  For the case where long chains of $R$-nodes are connected by non-flat edges, the finiteness of the graph ensures that eventually one of these $R$-nodes must connect to either an $L$-node or a $C$-node, so this last $R$-node will also fall into one of the four cases depicted in Figure~\ref{fig:node}. 

  Associate a sign to each node as follows. For $C$-nodes, the sign is positive if $\sigma_{\imh}(v)=++--$ where $v$ is the vertex admitting neither an $\imt$-neighbor nor an $i$-neighbor and negative otherwise. For $L$-nodes and $R$-nodes, the sign is positive if the component belongs to $\G_{(3,2)}$ and negative if it belongs to $\G_{(2,2,1)}$. Then the graph described in this way has $\LSP_5$ if and only if
  \begin{equation}
    \# C^+ = \# C^- \mbox{ and } \# L^+ = \# R^+ \mbox{ and } \# L^- = \# R^-. 
    \label{e:LSP5}
  \end{equation}
  Note that a flat edge changes the sign except for leaves, and a non-flat edge preserves the sign except for leaves.

  Note that if the leaf reached from the root using only flat edges has the same sign as the root, then the longest application of $\psi_i$, as discussed in Remark~\ref{rmk:long_psi}, may be applied to remove this leaf and the root, which has generating function $s_{(3,2)}$ if positive or $s_{(2,2,1)}$ if negative. Given the four possibilities in Figure~\ref{fig:node}, the only terminal case is the leftmost. Since the graph is locally Schur positive, \eqref{e:LSP5} ensures that there must be some leaf with the same sign as the root and a flat incoming edge. In the two rightmost cases in Figure~\ref{fig:node}, the map $\gamma_i$ may always be applied and doing so swaps the subtrees from the lower two $R$-nodes, similar to the scenarios in Figure~\ref{fig:bar}. Therefore we may use $\gamma_i$ to swap subtrees until this leaf lies on the flat path from the root.
\end{proof}

We can summarize the state of the transformations with the following result showing that a locally Schur positive graph of degree $n$ for which the $n-1$-restriction is a dual equivalence graph can itself be transformed into a dual equivalence graph. In particular, it has a Schur positive quasisymmetric generating function.

\begin{theorem}
  Let $\G$ be a locally Schur positive graph of type $(n,n)$, and suppose that the $(\imt,n)$-restriction of $\G$ is a dual equivalence graph and that the $(i,n)$-restriction of $\G$ satisfies dual equivalence axiom $4$. Then we can apply $\varphi_i, \psi_i, \gamma_i,$ and $\theta_i$ in such a way that the resulting graph still satisfies axioms $1,2$ and $5$, the $(\ipo,N)$-restriction is a dual equivalence graph.
  \label{thm:one-step}
\end{theorem}

\begin{proof}
  By Theorem~\ref{thm:nonempty}, we may always apply either $\varphi_i$ or $\psi_i$, perhaps with an intermediate application of $\gamma_i$. By Theorem~\ref{thm:phi-terminate}, each application of $\varphi_i$ strictly decreases $|W_i|$, and by Theorem~\ref{thm:psi-terminate} applying $\psi_i$ does not increase $|W_i|$, so eventually $W_i$ will be empty. By Theorem~\ref{thm:psi-terminate}, applying $\psi_i$ strictly decreases $|C_i|$, so once $W_i$ is empty, $\varphi_i$ will no longer be applied, and repeated applications of $\psi_i$ will result in $C_i$ being empty as well. At this point, by Proposition~\ref{prop:empty}, axiom $4$ holds for the $(\ipo,N)$-restriction. By construction, these maps maintain axioms $1,2$ and $5$, and axiom $4$ implies axiom $3$ for the $(\ipo,N)$-restriction. 

  By Proposition~\ref{prop:theta-reasonable}, if axiom $6$ fails for the $(\ipo,N)$-restriction, then we may choose a negatively dominant $(i,i)$-restricted component $\C$ and apply $\theta_i^{\C}$ while maintaining axioms $1,2,3$ and $5$. By Theorem~\ref{thm:theta-LSP}, the resulting graph remains locally Schur positive. Moreover, there are only two cases where the the $(\iph,N)$-restriction of $\theta_i^{\C}(\G)$ fails to satisfy axiom $4$: if one of $u,w,x,v$ has $\ipo$-type W or if one of $u,w,x,v$ has $\ipt$-type C. Suppose $w$ is the offending vertex. Then $w \in W_{\ipo}(\G)$ in the former case and $z = E_{\ipt}(w) \in C_{\ipt}(\G)$ in the latter. Given the strong hypotheses of the Proposition, $W_{\ipo}^0(\G) = W_{\ipo}(\G)$ in the former case and $C_{\ipt}^0(\G) = C_{\ipt}(\G)$ in the latter. Therefore $\varphi_{\ipo}$ or $\psi_{\ipt}$ may be used to restore axiom $4$ for the $(\iph,N)$-restriction. Therefore we may choose another negatively dominant component and continue thus. By Proposition~\ref{prop:theta-reasonable}, this process terminates exactly when axiom $6$ is satisfied for the $(\ipo,N)$-restriction, thus completing step 6. The result satisfies axioms $1, 2$ and $5$ by construction, and once again axiom $4$ implies axiom $3$ for the $(\iph,N)$-restriction. 
\end{proof}

At this point, we have only to consider the effect of the transformations on edges $E_j$ for $j>i$. For $j\geq \iph$, the extension of the maps to $i$-packages ensures that local Schur positivity is preserved. However, in general, these maps do not preserve local Schur positivity for $E_{\ipo}$ or $E_{\ipt}$.

%
\section{Additional axioms}
%
\label{sec:axiom4p}

It is not the case, in general, that $\varphi_i$ maintains local Schur positivity. The following condition is essential for ensuring that there is a way to maintain $\LSP_4$ for edges greater than $i$.

\begin{definition}
  A locally Schur positive graph $\G$ \emph{satisfies axiom $4'$(a)} if the following condition holds. If $w \in W_i(\G)$ has a non-flat $\imo$-edge, then the components of $E_{\imt} \cup E_{\imo}$ and $E_{\imo} \cup E_i$ containing $w$ have the same quasisymmetric functions in their degree $4$ generating functions
  \label{defn:axiom4pa}
\end{definition}

\begin{figure}[ht]
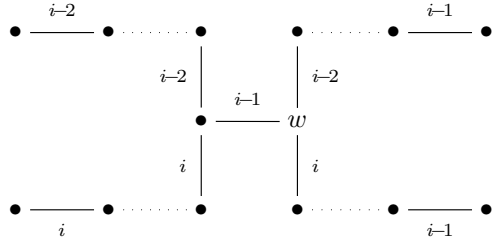

  \begin{displaymath}
    \begin{array}{\cs{7}\cs{7}\cs{7}\cs{7}\cs{7}c}
      \rnode{A0}{\B} & \rnode{B0}{\B} & \rnode{C0}{\B} & \rnode{D0}{\B} & \rnode{E0}{\B} & \rnode{F0}{\B} \\[5ex]
      & & \rnode{C1}{\B} & \rnode{D1}{w} & & \\[5ex]      
      \rnode{A2}{\B} & \rnode{B2}{\B} & \rnode{C2}{\B} & \rnode{D2}{\B} & \rnode{E2}{\B} & \rnode{F2}{\B}
    \end{array}
    \psset{nodesep=3pt,linewidth=.1ex}
    \everypsbox{\scriptstyle}
    \ncline {A0}{B0} \naput{\imt}
    \ncline[linestyle=dotted] {B0}{C0} 
    \ncline[linestyle=dotted] {D0}{E0}
    \ncline {E0}{F0} \naput{\imo}
    \ncline {C0}{C1} \nbput{\imt}
    \ncline {D0}{D1} \naput{\imt}
    \ncline {C1}{D1} \naput{\imo}
    \ncline {C1}{C2} \nbput{i}
    \ncline {D1}{D2} \naput{i}
    \ncline {A2}{B2} \nbput{i}
    \ncline[linestyle=dotted] {B2}{C2} 
    \ncline[linestyle=dotted] {D2}{E2}
    \ncline {E2}{F2} \nbput{\imo}
  \end{displaymath}
  \caption{\label{fig:axiom4pa} The case forbidden by axiom $4'a$.}
\end{figure}

The hypotheses of axiom $4'a$ ensure that both $w$ and $E_{\imo}(w)$ admit an $\imt$, an $\imo$ and an $i$-neighbor. If $E_{\imo} \cup E_i$ forms a closed loop through $w$, then each edge toggles $\sigma_{\imt}$. By axiom $2$, $E_i$ preserves $\sigma_{\imh}$. Since $w$ admits an $\imt$-neighbor, $E_i(w)$ does not. By axiom $2$, $E_{\imo}E_i(w)$ therefore admits an $\imt$-neighbor and the cycle continues so that $(E_{\imo}E_i)^m (w)$ admits an $\imt$-neighbor and $E_i (E_{\imo}E_i)^m (w)$ does not for all $m > 0$. By the assumption that the component is a loop, we must have $w = (E_{\imo}E_i)^m (w)$ for some $m > 0$, so then $E_i(w) = E_i (E_{\imo}E_i)^m (w)$ does not admit an $\imt$-neighbor. This contradiction works for $E_{\imt} \cup E_{\imo}$ as well, therefore neither can be a loop. This leaves two ways to align the two-color strings sharing an $E_{\imo}$ edge. One way results in the same degree $4$ signatures while the other is given in Figure~\ref{fig:axiom4pa}. Note that applying $\varphi_{\imo}^w$ in this case breaks $\LSP_4$, if it held for the graph, for $E_{\imo} \cup E_{i}$ and $\varphi_{i}^w$ cannot be applied since the $\imo$-edge at $w$ is not flat. Moreover, this is the only case where both maps fail.

Figure~\ref{fig:4'a} shows a graph with $\LSP_4$ and $\LSP_5$ that violates axiom $4'a$. The generating function is not Schur positive. Here $\varphi_4$ is needed in two places, and in both instances breaks local Schur positivity. There are two places requiring $\varphi_5$, however neither satisfies the hypotheses necessary to apply the map.

\begin{figure}[ht]
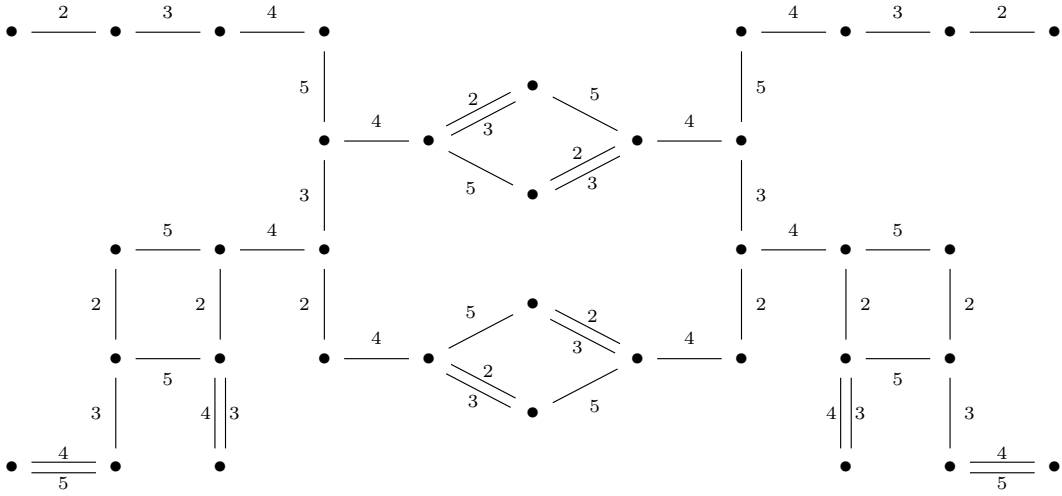

  \begin{displaymath}
    \begin{array}{\cs{8}\cs{8}\cs{8}\cs{8}\cs{8}\cs{8}\cs{8}\cs{8}\cs{8}\cs{8}c}
      \rnode{z1}{\B} & \rnode{a0}{\B} & \rnode{b0}{\B} & \rnode{c0}{\B} & \rnode{d0}{  } & \rnode{e0}{  } & \rnode{f0}{  } & \rnode{g0}{\B} & \rnode{h0}{\B} & \rnode{i0}{\B} & \rnode{y1}{\B} \\[2ex]
      & \rnode{a1}{  } & \rnode{b1}{  } & \rnode{c1}{  } & \rnode{d1}{  } & \rnode{e1}{\B} & \rnode{f1}{  } & \rnode{g1}{  } & \rnode{h1}{  } & \rnode{i1}{  } & \\[2ex]
      & \rnode{a2}{  } & \rnode{b2}{  } & \rnode{c2}{\B} & \rnode{d2}{\B} & \rnode{e2}{  } & \rnode{f2}{\B} & \rnode{g2}{\B} & \rnode{h2}{  } & \rnode{i2}{  } & \\[2ex]
      & \rnode{a3}{  } & \rnode{b3}{  } & \rnode{c3}{  } & \rnode{d3}{  } & \rnode{e3}{\B} & \rnode{f3}{  } & \rnode{g3}{  } & \rnode{h3}{  } & \rnode{i3}{  } & \\[2ex]
      & \rnode{a4}{\B} & \rnode{b4}{\B} & \rnode{c4}{\B} & \rnode{d4}{  } & \rnode{e4}{  } & \rnode{f4}{  } & \rnode{g4}{\B} & \rnode{h4}{\B} & \rnode{i4}{\B} & \\[2ex]
      & \rnode{a5}{  } & \rnode{b5}{  } & \rnode{c5}{  } & \rnode{d5}{  } & \rnode{e5}{\B} & \rnode{f5}{  } & \rnode{g5}{  } & \rnode{h5}{  } & \rnode{i5}{  } & \\[2ex]
      & \rnode{a6}{\B} & \rnode{b6}{\B} & \rnode{c6}{\B} & \rnode{d6}{\B} & \rnode{e6}{  } & \rnode{f6}{\B} & \rnode{g6}{\B} & \rnode{h6}{\B} & \rnode{i6}{\B} & \\[2ex]
      & \rnode{a7}{  } & \rnode{b7}{  } & \rnode{c7}{  } & \rnode{d7}{  } & \rnode{e7}{\B} & \rnode{f7}{  } & \rnode{g7}{  } & \rnode{h7}{  } & \rnode{i7}{  } & \\[2ex]
      \rnode{z8}{\B} & \rnode{a8}{\B} & \rnode{b8}{\B} & \rnode{c8}{  } & \rnode{d8}{  } & \rnode{e8}{  } & \rnode{f8}{  } & \rnode{g8}{  } & \rnode{h8}{\B} & \rnode{i8}{\B} & \rnode{y8}{\B} 
    \end{array}
    \psset{nodesep=5pt,linewidth=.1ex}
    \everypsbox{\scriptstyle}
    \ncline {z1}{a0} \naput{2}
    \ncline {a0}{b0} \naput{3}
    \ncline {b0}{c0} \naput{4}
    \ncline {g0}{h0} \naput{4}
    \ncline {h0}{i0} \naput{3}
    \ncline {i0}{y1} \naput{2}
    \ncline {c0}{c2} \nbput{5}
    \ncline {g0}{g2} \naput{5}
    \ncline {c2}{d2} \naput{4}
    \ncline[offset=2pt] {d2}{e1} \nbput{3}
    \ncline[offset=2pt] {e1}{d2} \nbput{2}
    \ncline {d2}{e3} \nbput{5}
    \ncline[offset=2pt] {e3}{f2} \nbput{3}
    \ncline[offset=2pt] {f2}{e3} \nbput{2}
    \ncline {e1}{f2} \naput{5}
    \ncline {f2}{g2} \naput{4}
    \ncline {c2}{c4} \nbput{3}
    \ncline {g2}{g4} \naput{3}
    \ncline {a4}{b4} \naput{5}
    \ncline {b4}{c4} \naput{4}
    \ncline {g4}{h4} \naput{4}
    \ncline {h4}{i4} \naput{5}
    \ncline {a4}{a6} \nbput{2}
    \ncline {b4}{b6} \nbput{2}
    \ncline {c4}{c6} \nbput{2}
    \ncline {g4}{g6} \naput{2}
    \ncline {h4}{h6} \naput{2}
    \ncline {i4}{i6} \naput{2}
    \ncline {a6}{b6} \nbput{5}
    \ncline {c6}{d6} \naput{4}
    \ncline[offset=2pt] {d6}{e7} \nbput{3}
    \ncline[offset=2pt] {e7}{d6} \nbput{2}
    \ncline {d6}{e5} \naput{5}
    \ncline[offset=2pt] {e5}{f6} \nbput{3}
    \ncline[offset=2pt] {f6}{e5} \nbput{2}
    \ncline {e7}{f6} \nbput{5}
    \ncline {f6}{g6} \naput{4}
    \ncline {h6}{i6} \nbput{5}
    \ncline {a6}{a8} \nbput{3}
    \ncline[offset=2pt] {b6}{b8} \nbput{4}
    \ncline[offset=2pt] {b8}{b6} \nbput{3}
    \ncline[offset=2pt] {h6}{h8} \nbput{4}
    \ncline[offset=2pt] {h8}{h6} \nbput{3}
    \ncline {i6}{i8} \naput{3}
    \ncline[offset=2pt] {z8}{a8} \nbput{5}
    \ncline[offset=2pt] {a8}{z8} \nbput{4}
    \ncline[offset=2pt] {i8}{y8} \nbput{5}
    \ncline[offset=2pt] {y8}{i8} \nbput{4}
  \end{displaymath}
  \caption{\label{fig:4'a}A locally Schur positive graph that fails axiom $4'a$.}
\end{figure}

As with $\varphi_i$, it is not the case, in general, that $\psi_i$ maintains local Schur positivity. The following condition is essential for ensuring that there is a way to maintain $\LSP_4$ for edges greater than $i$.

\begin{definition}
  A locally Schur positive graph $\G$ \emph{satisfies axiom $4'b$} if the following condition holds: if $x \in C_i(\G)$ has $\ipo$-type W, then there is a maximal length flat $i$-chain such that every vertex before $x$ or every vertex after $x$ has $\ipo$-type W.
  \label{defn:axiom4p}
\end{definition}

\begin{figure}[ht]
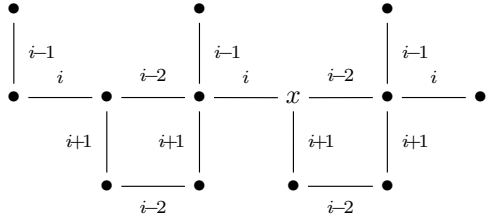

  \begin{displaymath}
    \begin{array}{\cs{7}\cs{7}\cs{7}\cs{7}\cs{7}c}
      \rnode{a0}{\B} & \rnode{b0}{  } & \rnode{c0}{\B} &
      \rnode{d0}{  } & \rnode{e0}{\B} & \rnode{f0}{  } \\[5ex]
      \rnode{a1}{\B} & \rnode{b1}{\B} & \rnode{c1}{\B} &
      \rnode{d1}{x}  & \rnode{e1}{\B} & \rnode{f1}{\B} \\[5ex]
      \rnode{a2}{  } & \rnode{b2}{\B} & \rnode{c2}{\B} &
      \rnode{d2}{\B} & \rnode{e2}{\B} & \rnode{f2}{  }
    \end{array}
    \psset{nodesep=3pt,linewidth=.1ex}
    \everypsbox{\scriptstyle}
    \ncline {a1}{b1} \naput{i}
    \ncline {b1}{c1} \naput{\imt}
    \ncline {c1}{d1} \naput{i}
    \ncline {d1}{e1} \naput{\imt}
    \ncline {e1}{f1} \naput{i}
    \ncline {a1}{a0} \nbput{\imo}
    \ncline {c1}{c0} \nbput{\imo}
    \ncline {e1}{e0} \nbput{\imo}
    \ncline {b1}{b2} \nbput{\ipo}
    \ncline {c1}{c2} \nbput{\ipo}
    \ncline {d1}{d2} \naput{\ipo}
    \ncline {e1}{e2} \naput{\ipo}
    \ncline {b2}{c2} \nbput{\imt}
    \ncline {d2}{e2} \nbput{\imt}
  \end{displaymath}
  \caption{\label{fig:axiom4pb} The case forbidden by axiom $4'b$.}
\end{figure}

The hypotheses of axiom $4'b$ ensure that both $x$ and $E_i(x)$ admit an $\ipo$-neighbor (though $E_{\ipo}$ edges need not exist). By axioms $2$ and $1$, both $E_{\imt}(x)$ and $E_{\imt}(E_i(x))$ must admit an $\ipo$-neighbor (again, $E_{\ipo}$ edges need not exist, though if they do, axiom $5$ ensures the shown commutativity). The forbidden conclusion is that neither $E_iE_{\imt}(x)$ nor $E_iE_{\imt}(E_i(x))$ admits an $\ipo$-neighbor as depicted on the right side of Figure~\ref{fig:axiom4pb}. Note that applying $\psi_{i}^x$ in this case breaks $\LSP_4$, if it held for the graph, for $E_{i} \cup E_{\ipo}$ and fails axiom $3$. Applying $\varphi_{\ipo}^x$ breaks $\LSP_4$, if it held for the graph, for $E_{i} \cup E_{\ipo}$ across the $E_{\imt}$ edges, which are part of the $\ipo$-package of $w$. Again, this is the only case where both maps fail. 

Figure~\ref{fig:4'b} shows a graph that violates axiom $4'b$. The generating function is not Schur positive. Neither $\varphi_3$ nor $\varphi_4$ is needed. Each of $\varphi_5, \psi_4$ and $\psi_5$ can be applied in exactly one place, and none of these preserves local Schur positivity. In fact, both $\varphi_5$ and $\psi_4$ would create a violation of axiom $3$.

\begin{figure}[ht]
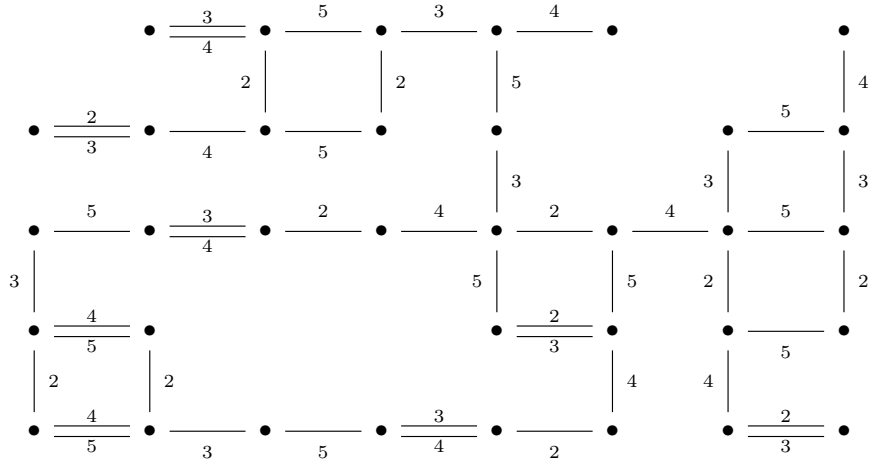

  \begin{displaymath}
    \begin{array}{\cs{9}\cs{9}\cs{9}\cs{9}\cs{9}\cs{9}\cs{9}c}
      \rnode{a1}{  } & \rnode{b1}{\B} & \rnode{c1}{\B} & \rnode{d1}{\B} & \rnode{e1}{\B} & \rnode{f1}{\B} & \rnode{g1}{  } & \rnode{h1}{\B} \\[6ex]
      \rnode{a2}{\B} & \rnode{b2}{\B} & \rnode{c2}{\B} & \rnode{d2}{\B} & \rnode{e2}{\B} & \rnode{f2}{  } & \rnode{g2}{\B} & \rnode{h2}{\B} \\[6ex]
      \rnode{b4}{\B} & \rnode{b3}{\B} & \rnode{c3}{\B} & \rnode{d3}{\B} & \rnode{e3}{\B} & \rnode{f3}{\B} & \rnode{g3}{\B} & \rnode{h3}{\B} \\[6ex]
      \rnode{b5}{\B} & \rnode{c5}{\B} & \rnode{c4}{  } & \rnode{d4}{  } & \rnode{e4}{\B} & \rnode{f4}{\B} & \rnode{g4}{\B} & \rnode{h4}{\B} \\[6ex]
      \rnode{b6}{\B} & \rnode{c6}{\B} & \rnode{d6}{\B} & \rnode{e6}{\B} & \rnode{f6}{\B} & \rnode{f5}{\B} & \rnode{g5}{\B} & \rnode{g6}{\B}
    \end{array}
    \psset{nodesep=5pt,linewidth=.1ex}
    \everypsbox{\scriptstyle}
    \ncline[offset=2pt] {b1}{c1} \nbput{4}
    \ncline[offset=2pt] {c1}{b1} \nbput{3}
    \ncline {c1}{d1} \naput{5}
    \ncline {d1}{e1} \naput{3}
    \ncline {e1}{f1} \naput{4}
    \ncline {c1}{c2} \nbput{2}
    \ncline {d1}{d2} \naput{2}
    \ncline {e1}{e2} \naput{5}
    \ncline {h1}{h2} \naput{4}
    \ncline[offset=2pt] {a2}{b2} \nbput{3}
    \ncline[offset=2pt] {b2}{a2} \nbput{2}
    \ncline {b2}{c2} \nbput{4}
    \ncline {c2}{d2} \nbput{5}
    \ncline {g2}{h2} \naput{5}
    \ncline {e2}{e3} \naput{3}
    \ncline {g2}{g3} \nbput{3}
    \ncline {h2}{h3} \naput{3}
    \ncline[offset=2pt] {b3}{c3} \nbput{4}
    \ncline[offset=2pt] {c3}{b3} \nbput{3}
    \ncline {c3}{d3} \naput{2}
    \ncline {d3}{e3} \naput{4}
    \ncline {e3}{f3} \naput{2}
    \ncline {f3}{g3} \naput{4}
    \ncline {g3}{h3} \naput{5}
    \ncline {b3}{b4} \nbput{5}
    \ncline {e3}{e4} \nbput{5}
    \ncline {f3}{f4} \naput{5}
    \ncline {g3}{g4} \nbput{2}
    \ncline {h3}{h4} \naput{2}
    \ncline[offset=2pt] {e4}{f4} \nbput{3}
    \ncline[offset=2pt] {f4}{e4} \nbput{2}
    \ncline {g4}{h4} \nbput{5}
    \ncline {b4}{b5} \nbput{3}
    \ncline {f4}{f5} \naput{4}
    \ncline {g4}{g5} \nbput{4}
    \ncline[offset=2pt] {b5}{c5} \nbput{5}
    \ncline[offset=2pt] {c5}{b5} \nbput{4}
    \ncline {b5}{b6} \naput{2}
    \ncline {c5}{c6} \naput{2}
    \ncline {f5}{f6} \naput{2}
    \ncline[offset=2pt] {g5}{g6} \nbput{3}
    \ncline[offset=2pt] {g6}{g5} \nbput{2}
    \ncline[offset=2pt] {b6}{c6} \nbput{5}
    \ncline[offset=2pt] {c6}{b6} \nbput{4}
    \ncline {c6}{d6} \nbput{3}
    \ncline {d6}{e6} \nbput{5}
    \ncline[offset=2pt] {e6}{f6} \nbput{4}
    \ncline[offset=2pt] {f6}{e6} \nbput{3}
  \end{displaymath}
  \caption{\label{fig:4'b}A locally Schur positive graph satisfying axioms $1,2,3$ and $5$ along with axiom $4'a$ but not $4'b$.}
\end{figure}

Note that both of these examples fail to have $\LSP_6$. It remains open as to whether $\LSP_6$ implies axiom $4'$. The following result shows that it is, in some sense, axiom $4'$ is already superfluous.

\begin{theorem}
  Let $\G$ be a locally Schur positive graph satisfying axiom $4'$ such that the $(\imt,N)$-restriction of $\G$ is a dual equivalence graph. For any $w \in W_i^0(\G)$ or $x \in C_i^0(\G)$, if $\varphi_i^w(\G)$ or $\psi_i^x(\G)$ or $\gamma_i^z(\G)$ has $\LSP_4$, then $\varphi_i^w(\G)$ or $\psi_i^x(\G)$ or $\gamma_i^z(\G)$ maintains axiom $4'$, respectively, and for $\C$ a negatively dominant restricted component, $\theta_i^{\C}(\G)$ vacuously satisfies axiom $4'$.
\label{thm:axiom4p}
\end{theorem}

\begin{proof}
  There are three cases to consider for axiom $4'a$: $E_{\imt} \cup E_{\imo} \cup E_i$, $E_{\imo} \cup E_i \cup E_{\ipo}$, and $E_i \cup E_{\ipo} \cup E_{\ipt}$. By $\LSP_4$ of $\G$, the overlapping two color strings that satisfy the hypotheses for axiom $4'a$ both have generating functions $s_{(3,1)} + m s_{(2,2)}$ or both have $s_{(2,1,1)} + m s_{(2,2)}$ for some $m>0$. Since $\LSP_4$ is assumed to be maintained by each of the maps in question, the only change in generating function is to reduce $m$. If $m$ becomes $0$, the hypotheses of axiom $4'a$ are not met, and if $m$ remains positive then the quasisymmetric functions that appear remain unchanged. Therefore none of the maps can create a violation of axiom $4'a$ provided $\LSP_4$ holds.

  There are two cases to consider for axiom $4'b$: $E_{\imt} \cup E_i$ and $E_i \cup E_{\ipt}$. The latter case is easily resolved by axiom $2$ since all edges involved in an application of $\varphi_i^w$ or $\psi_i^x$ or $\gamma_i^z$ preserve $\sigma_{\ipt,\iph}$, and so preserve $\iph$-type W. Thus any violation after applying either map must have already existed in $\G$. Any violation of axiom $4'b$ for $E_{\imt} \cup E_i$ created by $\psi_i^x$ or $\gamma_i^z$ must have existed already in $\G$ since these maps remove or exchange sequences within flat $i$-chains. Therefore we consider how $\varphi_i^w$ might result in a component as depicted in the right side of Figure~\ref{fig:axiom4pb}. There are three $i$-edges that could have resulted from $\varphi_i^w$.

  \begin{figure}[ht]
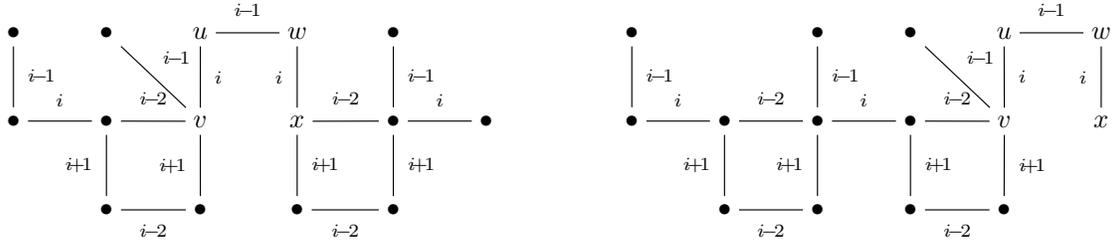

    \begin{displaymath}
      \begin{array}{\cs{7}\cs{7}\cs{7}\cs{7}\cs{7}c}
        \rnode{a0}{\B} & \rnode{b0}{\B} & \rnode{c0}{u}  &
        \rnode{d0}{w}  & \rnode{e0}{\B} & \rnode{f0}{  } \\[5ex]
        \rnode{a1}{\B} & \rnode{b1}{\B} & \rnode{c1}{v}  &
        \rnode{d1}{x}  & \rnode{e1}{\B} & \rnode{f1}{\B} \\[5ex]
        \rnode{a2}{  } & \rnode{b2}{\B} & \rnode{c2}{\B} &
        \rnode{d2}{\B} & \rnode{e2}{\B} & \rnode{f2}{  }
      \end{array}
      \hspace{4em}
      \begin{array}{\cs{7}\cs{7}\cs{7}\cs{7}\cs{7}c}
        \rnode{A0}{\B} & \rnode{B0}{  } & \rnode{C0}{\B} &
        \rnode{D0}{\B} & \rnode{E0}{u}  & \rnode{F0}{w}  \\[5ex]
        \rnode{A1}{\B} & \rnode{B1}{\B} & \rnode{C1}{\B} &
        \rnode{D1}{\B} & \rnode{E1}{v}  & \rnode{F1}{x}  \\[5ex]
        \rnode{A2}{  } & \rnode{B2}{\B} & \rnode{C2}{\B} &
        \rnode{D2}{\B} & \rnode{E2}{\B} & \rnode{F2}{  }
      \end{array}
      \psset{nodesep=3pt,linewidth=.1ex}
      \everypsbox{\scriptstyle}
      \ncline {a1}{b1} \naput{i}
      \ncline {b1}{c1} \naput{\imt}
      \ncline {c1}{c0} \nbput{i}
      \ncline {c0}{d0} \naput{\imo}
      \ncline {d0}{d1} \nbput{i}
      \ncline {d1}{e1} \naput{\imt}
      \ncline {e1}{f1} \naput{i}
      \ncline {a1}{a0} \nbput{\imo}
      \ncline {c1}{b0} \nbput{\imo}
      \ncline {e1}{e0} \nbput{\imo}
      \ncline {b1}{b2} \nbput{\ipo}
      \ncline {c1}{c2} \nbput{\ipo}
      \ncline {d1}{d2} \naput{\ipo}
      \ncline {e1}{e2} \naput{\ipo}
      \ncline {b2}{c2} \nbput{\imt}
      \ncline {d2}{e2} \nbput{\imt}
      \ncline {A1}{B1} \naput{i}
      \ncline {B1}{C1} \naput{\imt}
      \ncline {C1}{D1} \naput{i}
      \ncline {D1}{E1} \naput{\imt}
      \ncline {E1}{E0} \nbput{i}
      \ncline {E0}{F0} \naput{\imo}
      \ncline {F0}{F1} \nbput{i}
      \ncline {A1}{A0} \nbput{\imo}
      \ncline {C1}{C0} \nbput{\imo}
      \ncline {E1}{D0} \nbput{\imo}
      \ncline {B1}{B2} \nbput{\ipo}
      \ncline {C1}{C2} \nbput{\ipo}
      \ncline {D1}{D2} \naput{\ipo}
      \ncline {E1}{E2} \naput{\ipo}
      \ncline {B2}{C2} \nbput{\imt}
      \ncline {D2}{E2} \nbput{\imt}
    \end{displaymath}
    \caption{\label{fig:phi-ax4p} An illustration of how $\varphi_i^w$ might result in a violation of axiom $4'b$.}
  \end{figure}

  For the right edge, let $x,w,u,v$ be as depicted in the right side of Figure~\ref{fig:phi-ax4p}. If $x$ does not admit an $\ipo$-neighbor, then by axiom $2$, $w$ must. By axioms $1$ and $2$, this ensures that $u$ does not admit an $\ipo$-neighbor. Therefore the $i$-edge between $v$ and $x$ is the right edge in a violation of axiom $4'b$ in $\varphi_i^w(\G)$ only if the $i$-edge between $v$ and $u$ is the right edge in a violation of axiom $4'b$ in $\G$. The case for left edge is similarly resolved.

  For the middle edge, let $x,w,u,v$ be as depicted in the left side of Figure~\ref{fig:phi-ax4p}. By axiom $2$ and the fact that the $i$-edge between $u$ and $v$ is not flat, $u$ has no $\imt$-neighbor since $v$ does. By axiom $2$ again, $w$ must admit an $\imt$-neighbor. Since $x$ also admits an $\imt$-neighbor, the $i$-edge between $w$ and $x$ must be flat by axiom $3$, and so they lie on a flat $i$-chain. Now for $\ipo$-neighbors, by axiom $2$, $\sigma(u)_{\ipo} = \sigma(w)_{\ipo}$ and, by axiom $1$, $\sigma(u)_{i} = -\sigma(w)_{i}$. Therefore exactly one of $u$ and $w$ admits an $\ipo$-neighbor, so we consider each case in turn. If $w$ admits an $\ipo$-neighbor, then the $i$-edge between $w$ and $x$ is the middle edge in a violation of axiom $4'b$ for $E_{\imt} \cup E_i$ in $\G$, a contradiction. Alternatively, if $u$ admits an $\ipo$-neighbor, then by axiom $4'a$ for $E_{\imo} \cup E_i \cup E_{\ipo}$ in $\G$, following the $E_i \cup E_{\ipo}$ string from $u$ through $v$ and onwards must terminate in an $i$-edge. In particular, $E_{\ipo}(v)$ admits an $i$-neighbor, say $z = E_i E_{\ipo}(v)$. Axioms $2$ and $3$ ensure that both $v$ and $z$ admit an $\imo$-neighbor while $E_{\ipo}(v) = E_i(z)$ does not. Moreover, all of $v, E_{\ipo}(v) = E_i(z)$ and $z$ admit an $\imt$-neighbor. In particular, the $i$-edge between $E_i(z)$ and $z$ is flat, and, since $E_i(z)$ does not admit an $\imo$-neighbor, neither $E_i(z)$ nor $E_{\imt}E_i(z)$ has $\imo$-type W. Therefore by axioms $2$ and $3$, $E_{\imt}E_i(z)$ admits an $i$-neighbor, so we again have a flat $i$-chain. By axiom $5$, $E_{\imt}E_{\ipo}(v) = E_{\ipo}E_{\imt}(v)$, and the assumption on $\G$ is that $E_i E_{\imt}(v)$ does not admit an $\ipo$-neighbor. Therefore by local Schur positivity of $E_i \cup E_{\ipo}$ in $\G$, the $E_i \cup E_{\ipo}$ string beginning at $E_i E_{\imt}(v)$ ends with an $\ipo$-edge. In particular, this implies $E_i E_{\imt} E_i(z)$ admit an $\ipo$-neighbor. At long last, this creates a violation of axiom $4'b$ in $\G$, another contradiction.

  By Theorem~\ref{thm:theta-LSP}, it suffices to show that axiom $4'$ is maintained.  Since axiom $4$ holds for the $(\iph,N)$-restriction of $\G$, even after applying $\theta_i^{\C}$, axiom $4'a$ is vacuously satisfied since only $E_i \cup E_{\ipo}$ strings have the potential not to appear in Figure~\ref{fig:lambda4}. Similarly, axiom $4'b$ is vacuous for $E_{\imt} \cup E_i$. 

  For axiom $4'b$ for $E_{i} \cup E_{\ipt}$, consider $\{w,x\}, \{u,v\} \in E_i(\G)$ with $\theta_i^{\C}(w) = u$ and $\theta_i^{\C}(x) = v$. If neither $w,x$ nor $u,v$ has $\ipt$-type C, then $C_i(\theta_i^{\C}(\G))$ remains empty, so assume $w,x$ have $\ipt$-type C. In this case, either all of $w,x,E_i(w),E_i(x)$ have $\iph$-type W or none does. Therefore if $u,v$ also have $\ipt$-type C, then we have the situation depicted in the left side of Figure~\ref{fig:theta-4p}. In this case, the maximal flat $\ipt$-chain $(w, E_i(w), E_i(x), x, v, E_i(v), E_i(u), u)$ has either all, none, the first four, or the last four vertices of $\iph$-type W, thereby satisfying axiom $4'b$.

  \begin{figure}[ht]
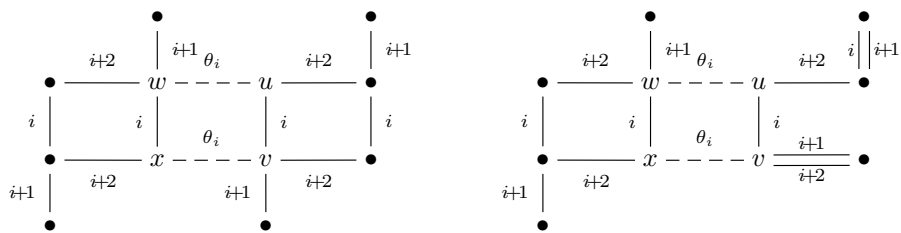

    \begin{displaymath}
      \begin{array}{\cs{8}\cs{8}\cs{8}c}
        & \rnode{uu}{\B} & & \rnode{vv}{\B} \\[3ex]
        \rnode{x}{\B} & \rnode{w}{w} & \rnode{u}{u} & \rnode{v}{\B} \\[4ex]
        \rnode{X}{\B} & \rnode{W}{x} & \rnode{U}{v} & \rnode{V}{\B} \\[3ex]
        \rnode{XX}{\B} & & \rnode{WW}{\B} &
      \end{array}
      \hspace{5em}
      \begin{array}{\cs{8}\cs{8}\cs{8}c}
        & \rnode{uu1}{\B} & & \rnode{vv1}{\B} \\[3ex]
        \rnode{x1}{\B} & \rnode{w1}{w} & \rnode{u1}{u} & \rnode{v1}{\B} \\[4ex]
        \rnode{X1}{\B} & \rnode{W1}{x} & \rnode{U1}{v} & \rnode{V1}{\B} \\[3ex]
        \rnode{WW1}{\B} & & & 
      \end{array}
      \psset{nodesep=3pt,linewidth=.1ex}
      \everypsbox{\scriptstyle}
      \ncline {v}{vv}\nbput{\ipo}
      \ncline {w}{uu}\nbput{\ipo}
      \ncline {x}{w} \naput{\ipt}
      \ncline[linestyle=dashed] {w}{u} \naput{\theta_i}
      \ncline {u}{v} \naput{\ipt}
      \ncline {x}{X} \nbput{i}
      \ncline {w}{W} \nbput{i}
      \ncline {u}{U} \naput{i}
      \ncline {v}{V} \naput{i}
      \ncline {X}{W} \nbput{\ipt}
      \ncline {U}{V} \nbput{\ipt}
      \ncline[linestyle=dashed] {W}{U} \naput{\theta_i}
      \ncline {WW}{U} \naput{\ipo}
      \ncline {XX}{X} \naput{\ipo}
      \ncline {w1}{uu1}\nbput{\ipo}
      \ncline {X1}{WW1}\nbput{\ipo}
      \ncline {x1}{w1} \naput{\ipt}
      \ncline[offset=2pt] {U1}{V1} \nbput{\ipt}
      \ncline[offset=2pt] {V1}{U1} \nbput{\ipo}
      \ncline[linestyle=dashed] {w1}{u1} \naput{\theta_i}
      \ncline[linestyle=dashed] {W1}{U1} \naput{\theta_i}
      \ncline {x1}{X1} \nbput{i}
      \ncline {W1}{w1} \naput{i}
      \ncline {u1}{U1} \naput{i}
      \ncline {X1}{W1} \nbput{\ipt}
      \ncline {u1}{v1} \naput{\ipt}
      \ncline[offset=2pt] {vv1}{v1} \nbput{i}
      \ncline[offset=2pt] {v1}{vv1} \nbput{\ipo}
    \end{displaymath}
    \caption{\label{fig:theta-4p} The two possible cases where $C_i(\theta_i^{\C}(\G))$ is nonempty.}
  \end{figure}

  If $u,v$ instead have $\ipt$-type B, say where $E_{\ipt}(v) = E_{\ipo}(v)$ and so $E_i E_{\ipt}(u) = E_{\ipo} E_{\ipt} (u)$, then we have the situation depicted on the right side of Figure~\ref{fig:theta-4p}. In this case, the maximal flat $\ipt$-chain is $(E_{\ipt}(u), u, w, E_{\ipt}(w), E_{\ipt}(x), x)$, where all or none of the last four vertices have $\iph$-type W, again satisfying axiom $4'b$.
\end{proof}

The question remains as to whether axiom $4'$ is sufficient to guarantee that these transformations can be applied in a suitable way to maintain local Schur positivity throughout. Extensive computer evidence suggests that the graphs for Macdonald polynomials, LLT polynomials, $k$-Schur functions, and Schubert times Schur coefficients all enjoy Schur positive components, and may all be transformed systematically by the maps presented in this paper (interestingly, without the need for $\gamma_i^z$). Studying these transformations in detail has led to a formula for Macdonald polynomials in cases where the corresponding graph is not a dual equivalence graph \cite{Ass-B}. A deeper study of these transformations applied to these specific graphs may yield a comprehensive theory that encompasses these and further applications.

%
%

\bibliographystyle{amsalpha} 
\bibliography{part2.bib}

\end{document}